\documentclass[reqno]{amsart}
\usepackage[numbers,sort]{natbib}

\usepackage[foot]{amsaddr}
\usepackage{amsmath}
\usepackage{amsthm}
\usepackage{amssymb}
\usepackage{bbm} 
\def\1{\mathbbm{1}}
\usepackage{bbm}
\usepackage{bm}
\usepackage{stmaryrd}
\usepackage{tikz}
\usepackage{mathtools}
\usepackage{ mathrsfs }

\usepackage{pgfplots}
\pgfplotsset{compat=1.17}
\usepackage{subcaption}
\usepackage[hidelinks]{hyperref}
\usepackage[stretch=10]{microtype}
\usepackage{float}
\usepackage{graphicx}
\usepackage{caption}
\usepackage{url}
\usepackage{scalerel}
\usepackage{enumitem}
\usepackage{cleveref}
\crefformat{equation}{(#2#1#3)}
\usepackage[normalem]{ulem}
\usepackage{booktabs}
\usepackage[framemethod=TikZ]{mdframed}
\usepackage{xcolor}
\usepackage{siunitx}

\theoremstyle{plain}
\newtheorem{theorem}{Theorem}[section]
\newtheorem{corollary}[theorem]{Corollary}
\newtheorem{lemma}[theorem]{Lemma}
\newtheorem{proposition}[theorem]{Proposition}

\theoremstyle{definition}
\newtheorem{definition}[theorem]{Definition}

\newcounter{cnstcnt}
\newcommand{\newconstant}[1]{
  \refstepcounter{cnstcnt}
  \ensuremath{c_{\thecnstcnt}}
  \label{#1}
  }
\newcommand{\cte}[1]{ \begin{NoHyper}\ensuremath{c_{\ref{#1}}}\end{NoHyper} }

\usepackage{etoolbox}
\newcounter{prfpart}
\newcounter{prfparthref}

\AtBeginEnvironment{proof}{\setcounter{prfpart}{0}}

\newcommand{\proofpart}[1]{
  \par
  \addvspace{\medskipamount}
  \noindent\underline{\emph{Part \refstepcounter{prfpart}\theprfpart
  \stepcounter{prfparthref}: #1}}
  \par\nobreak
  \addvspace{\smallskipamount}
}

\newcounter{subprfpart}[prfpart]
\renewcommand{\thesubprfpart}{\theprfpart.\arabic{subprfpart}}

\newcommand{\subproofpart}[1]{
  \par
  \addvspace{\medskipamount}
  \noindent\underline{\emph{Part \refstepcounter{subprfpart}\thesubprfpart: #1}}
  \par\nobreak
  \addvspace{\smallskipamount}
}

\title{Singular value distribution of dense random matrices with block Markovian dependence}
\author{Jaron Sanders} 
\email{jaron.sanders@tue.nl}
\author{Alexander Van Werde}
\email{a.van.werde@tue.nl}

\address{Eindhoven University of Technology\\ Department of Mathematics \& Computer Science\\ The Netherlands}

\date{}

\begin{document}

\begin{abstract}
    A block Markov chain is a Markov chain whose state space can be partitioned into a finite number of clusters such that the transition probabilities only depend on the clusters.  
    Block Markov chains thus serve as a model for Markov chains with communities.
    This paper establishes limiting laws for the singular value distributions of the empirical transition matrix and empirical frequency matrix associated to a sample path of the block Markov chain whenever the length of the sample path is $\Theta(n^2)$ with $n$ the size of the state space.

    The proof approach is split into two parts. 
    First, we introduce a class of symmetric random matrices with dependent entries called approximately uncorrelated random matrices with variance profile. We establish their limiting eigenvalue distributions by means of the moment method.    
    Second, we develop a coupling argument to show that this general-purpose result applies to the singular value distributions associated with the block Markov chain. 
\end{abstract}

\maketitle

 

\section{Introduction}
\label{sec: Introduction}

Understanding hidden structures that underlie sequential data is an important challenge in data science. 
Not only do these structures give insight into the complex process which generates the data; once the structure is determined, any subsequent analysis can benefit from a reduction in dimensionality.

An attribute of sequential data is that past and future samples are dependent. 
Many tools for data analysis however have only been analyzed for independent data, because dependencies make the mathematical analysis of a system more difficult.
Still, one can simply apply a tool originally designed for independent data and hope for insights. 
The conclusions one derives assuming independence will in the best scenario, but not necessarily, coincide with the conclusions that one would have derived in the appropriate dependent setting.

One such tool of the trade are spectral methods \cite{shi2000normalized,stratos2013spectral,abbe2017community,athreya2017statistical}.
The term spectral methods refers here to any algorithm that makes use of the eigenvalues and eigenvectors of a matrix built from the data. 
Eigenvalues and eigenvectors are indeed routinely used to understand the underlying structure in data.
If, for instance, the empirical eigenvalue distribution does not match the theoretical predictions associated with some model, then this may signify that the model is missing some important component of the process which generated the data \cite{sarkar2018spectral,rai2018network,li2019testing,mosam2021breakdown}. 
Another example occurs in principal component analysis, where the eigenvalue distribution of an empirical covariance matrix can be used to detect the appropriate number of principal components \cite{veraart2016denoising, ke2021estimation}.
If the matrix which is built from the data is non-Hermitian then algorithms typically employ singular values and singular vectors instead of eigenvalues and eigenvectors. 
Here recall that the $i$th largest singular value $s_i(M)$ of a square real matrix $M$ is defined in terms of the $i$th largest eigenvalue of $MM^{\rm{T}}$ as $s_i(M) := (\lambda_i(MM^{\rm{T}}))^{1/2}$.   

We are specifically interested in the singular value distribution of the \emph{sample frequency matrix}
\begin{equation}
    \hat{N}_X 
    := 
    ( \hat{N}_{X,ij} )_{i,j=1}^n
    \quad
    \textnormal{where}
    \quad 
    \hat{N}_{X,ij}
    := 
    \sum_{t=0}^{\ell-1} \1_{ X_t = i, X_{t+1} = j}
    \label{eq: Definition__NhatX}
\end{equation} 
built from a dependent data sequence $X_0, X_1, \ldots, X_\ell$ taking values in $\{1,\ldots,n\}$ for some positive integer $n\in \mathbb{Z}_{\geq 1}$.
The term \emph{singular value distribution of $\hat{N}_X$} here refers to the measure $\nu_{\hat{N}_X}$ defined by 
\begin{equation}
    \nu_{\hat{N}_X}([a,b]) := \frac{1}{n}\#\{i \in \{1,\ldots,n \}: a\leq s_i(\hat{N}_X) \leq b \} \label{eq: Def_SingvalDistribution}
\end{equation}
for any $a < b$. 
While the singular value distribution of a random matrix is well understood when the elements are independent, this is not so when the matrix is constructed from dependent sequential data as is the case for $\hat{N}_X$. 

This paper models the sequential data $X_0, \allowbreak X_1, \allowbreak  \ldots, \allowbreak X_{\ell}$ as being generated by a block Markov chain.
Block Markov chains are a model for dependent sequential data with an underlying community structure and have been used to develop and analyse community detection algorithms for sequential data; see \cite{zhang2020spectral,sanders2020clustering}.  
Besides these two papers and the present paper, the only other rigorous analysis of the spectral properties of $\hat{N}_X$ when $X$ is a block Markov chain can be found in \cite{sanders2021spectral}. 
There, an asymptotic distance between the $K$ largest singular values and the $n-K$ smallest singular values is established. 

The current paper establishes a limiting law for the singular value distribution of block Markovian random matrices such as $\hat{N}_X$ as the size $n$ of the state space tends to infinity and the length of the data sequence satisfies $\ell = \Theta(n^2)$.
For example, \Cref{thm: SingValN} describes the limiting law associated to $\hat{N}_X$, which is visualized in \Cref{fig: SingVals}. 
\Cref{thm: SingValN} furthermore implies that the singular value distribution in block Markovian random matrices behaves as one would predict assuming the entries $\hat{N}_{X,ij}$ are independent. 
This gives some legitimacy to the practice of assuming some independence, particularly when the dependencies are asymptotically equivalent to those of a block Markov chain and the data sequence is sufficiently long.
One can however not ignore the Markovian dependence entirely since it causes the frequencies for different transitions to have different distributions.  
Indeed, the singular value distribution following from \Cref{thm: SingValN} does not necessarily agree with the singular value distribution which one would find assuming that the different entries $\hat{N}_{X,ij}$ are identically distributed. 

We next state our main results in \Cref{sec: results}.
This is followed by an overview of the literature in \Cref{sec: RelatedLiterature}. 
\Cref{sec: Notation} then provides notation and preliminaries in preparation of the proofs.
Proof outlines are given in \Cref{sec: SketchOfProofs}, and a brief comparison between our theoretical predictions and the singular value distribution obtained from an actual dataset is done in \Cref{sec: Taxi}. 
Finally, all details for the proof are provided in \Cref{sec: RemainingProofs}.

\subsection{Results}
\label{sec: results}

Our main object of study, which will subsequently be formally defined, are Markov chains which have a community structure.
More precisely, the transition probabilities between states should only depend on the communities to which these states belong. 

Fix a positive integer $K\in \mathbb{Z}_{\geq 1}$. 
For any $n \in \mathbb{Z}_{\geq K}$, pick a partition $(\mathcal{V}_k)_{k=1}^K$ of $\mathcal{V}:= \{1,\ldots,n \}$ consisting only of nonempty sets. 
Let $p$ be the transition matrix for an irreducible acyclic Markov chain on $\{1,\ldots,K \}$ with equilibrium distribution $\pi$. 
The \emph{block Markov chain }with cluster transition matrix $p$ and clusters $(\mathcal{V}_k)_{k =1}^K$ is then the Markov chain on $\mathcal{V}$ with transition probability $P_{i,j}:= p_{x,y}/\# \mathcal{V}_{y}$ for every $i\in \mathcal{V}_x$ and $j\in \mathcal{V}_y$. 
\Cref{fig: BMC} schematically depicts a block Markov chain. 

\begin{figure}[tb]
    \centering
    \includegraphics[width=0.6\textwidth]{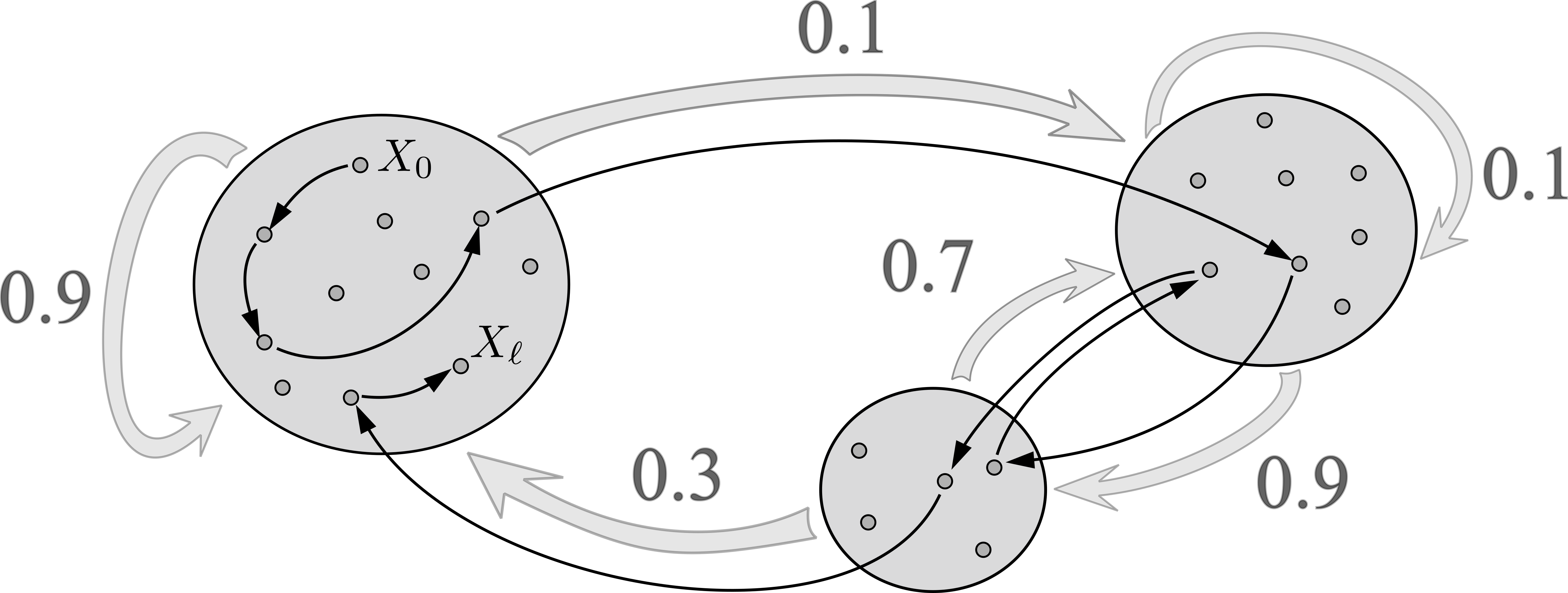}
    \caption{Visualization of a block Markov chain on $K=3$ clusters with $p = [[0.9,0.1,0],\allowbreak [0,0.1,0.9], \allowbreak [0.3,0.7,0]]$. The thick arrows visualize the cluster transition probabilities $p_{k,l}$ and the thin arrow visualise the transitions $(X_t,X_{t+1})$ of a sample path $(X_t)_{t=0}^\ell$. The starting point $X_0$ was chosen to lie in the leftmost cluster $\mathcal{V}_1$.       
    }
    \label{fig: BMC}
\end{figure} 

In the subsequent results we are concerned with the asymptotic regime where $n$ tends to infinity. 
Fix a sequence of strictly positive real numbers $\alpha := (\alpha_1,\ldots,\alpha_K)$ with $\sum_{k= 1}^K \alpha_k = 1$ and assume that for every $k\in \{1,\ldots,K \}$ it holds that $\# \mathcal{V}_k = \alpha_k n  + o(n)$.
Let $X:= (X_t)_{t=0}^\ell$ denote a sample path from the block Markov chain with an arbitrary starting distribution for $X_0$ and with length $\ell = \lambda n^2 + o(n^2)$ for some fixed $\lambda \in \mathbb{R}_{>0}$.
Recall the definition of the empirical frequency matrix $\hat{N}_X$ from \cref{eq: Definition__NhatX} and note that $\hat{N}_{X,ij}$ counts the number of traversals of edge $(i,j)$.
The \emph{empirical transition matrix} $\hat{P}_X$ associated with the sample path is given by  
\begin{align}
    \hat{P}_X := (\hat{P}_{X,ij})_{i,j=1}^n
    \quad
    \textnormal{where}
    \quad
    \hat{P}_{X,ij}:= \frac{\hat{N}_{X,ij}}{\sum_{k=1}^n\hat{N}_{X,ik}}.
    \label{eq: EmpTransition}
\end{align} 
It will be shown in \Cref{cor: Pi1Exists} that there is no division by zero in \cref{eq: EmpTransition} asymptotically almost surely. 
 
Some terminology is required to state the main results. 
The \emph{Stieltjes transform} of a finite nonzero measure $\mu$ on $\mathbb{R}$ is the analytic function $s:\mathbb{C}^+ \to \mathbb{C}^-$ given by $s(z) = \int 1/(z-x) {\mathop{}\!\mathrm{d}}\mu(x)$ where $\mathbb{C}^+ := \{z\in\mathbb{C}: \operatorname{Im}(z)>0\}$ is the upper half-plane and $\mathbb{C}^- := \{z\in\mathbb{C}: \operatorname{Im}(z)<0\}$ is the lower half-plane. 
Let us remark that some authors refer to the Stieltjes transform by another name such as \emph{Cauchy transform} or \emph{Cauchy--Stieltjes transform}. 
Further, let us warn that some authors employ a convention which differs by a minus sign from the notation employed here; they instead consider the map $z\mapsto \int 1/(x-z) {\mathop{}\!\mathrm{d}}\mu(x)$.  
The relevance of the Stieltjes transform for our purposes is that $\mu$ can be recovered from $s(z)$ by the Stieltjes inversion formula \cite[Theorem B.8.]{bai2010spectral} which states that for any continuity points $a<b$
\begin{align}
    \mu([a,b]) 
    = 
    -\frac{1}{\pi}\lim_{\varepsilon\to 0^+} \int_{a}^b \operatorname{Im}(s(x + \sqrt{-1}\varepsilon)) {\mathop{}\!\mathrm{d}}x.
    \label{eq: StieltjesInv}
\end{align} 
A sequence of random measures $\mu_n$ on $\mathbb{R}$ is said to \emph{converge weakly in probability} to a finite measure $\mu$ if $\int f(x){\mathop{}\!\mathrm{d}}\mu_n(x) \to \int f(x){\mathop{}\!\mathrm{d}}\mu(x)$ in probability for every continuous bounded function $f\in \mathcal{C}_b(\mathbb{R})$.
Finally, the \emph{symmetrization} of a measure $\mu$ on $\mathbb{R}_{\geq 0}$ is the measure $\operatorname{sym}(\mu)$ on $\mathbb{R}$ given by $A\mapsto (\mu(A \cap  \mathbb{R}_{\geq 0}) + \mu((- A) \cap \mathbb{R}_{\geq 0}))/2$ where $A$ ranges over all measurable subsets of $\mathbb{R}$ and $-A:= \{-a: a\in A \}$. 

\begin{figure}[tb]
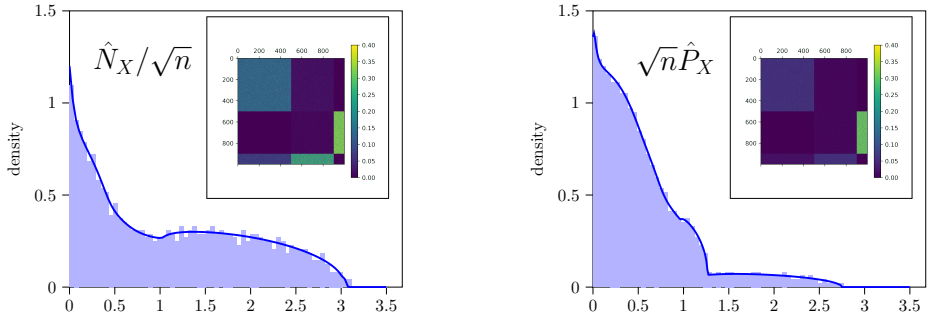
 
    \centering
    \begin{subfigure}[b]{0.45\textwidth}
        \centering
        \resizebox{.95\linewidth}{!}{\input{figures/BMC_N3.tex}}  
    \end{subfigure}
    \hfill 
    \begin{subfigure}[b]{0.45\textwidth}
        \centering
        \resizebox{.95\linewidth}{!}{\input{figures/BMC_P3.tex}}  
    \end{subfigure}
    \caption{
    On the left: $\hat{N}_X/\sqrt{n}$ and a frequency-based histogram of the singular values (bars) compared to our theoretical predictions (solid line). 
    On the right: $\sqrt{n}\hat{P}_X$ and its singular values. 
    These figures were made by sampling a path with $n = 1000, \ell = 2n^2$, $\alpha = (0.5,0.4,0.1)$ and cluster transition matrix $p = [[0.9,0.1,0],[0,0.1,0.9],[0.3,0.7,0]]$.
    The systems of equations in \Cref{thm: SingValN} and \Cref{thm: SingValP} were solved using the algorithm in \cite[Proposition 4.1]{helton2007operator} after which we used the Stieltjes inversion forumula \cref{eq: StieltjesInv} to recover the measures. 
    }
    \label{fig: SingVals}
\end{figure}

Visualizations of the following results are shown in \Cref{fig: SingVals}. 
\begin{theorem}
    \label{thm: SingValN}

    The empirical singular value distribution $\nu_{\hat{N}_X/\sqrt{n}}$ converges weakly in probability to a compactly supported probability measure $\nu$ on $\mathbb{R}_{\geq 0}$. 
    Moreover, the symmetrization $\operatorname{sym}(\nu)$ has Stieltjes transform $s(z) = \sum_{i=1}^{K} \alpha_i (a_i(z)\allowbreak + a_{K + i}(z))/2$ where $a_1,\ldots,a_{2K}$ are the unique analytic functions from $\mathbb{C}^+$ to $\mathbb{C}^-$ such that the following system of equations is satisfied  
    \begin{align}
        a_i(z)^{-1} &= z -\sum_{j=1}^{K}  \lambda \pi(i)\alpha_i^{-1} p_{i,j} a_{K+j}(z)\\ 
        a_{i+K}(z)^{-1} &= z - \sum_{j=1}^{K} \lambda \pi(j)\alpha_i^{-1} p_{j,i} a_{j}(z)  
    \end{align}
    for $i = 1,\ldots,K$.    
\end{theorem}
\newpage

\begin{theorem}
    \label{thm: SingValP}
    The empirical singular value distribution $\nu_{\sqrt{n}\hat{P}_X}$ converges weakly in probability to a compactly supported probability measure $\nu$ on $\mathbb{R}_{\geq 0}$. 
    Moreover, the symmetrization $\operatorname{sym}(\nu)$ has Stieltjes transform $s(z) = \sum_{i=1}^{K} \alpha_i (a_i(z)\allowbreak + a_{K + i}(z))/2$ where $a_1,\ldots,a_{2K}$ are the unique analytic functions from $\mathbb{C}^+$ to $\mathbb{C}^-$ such that the following system of equations is satisfied  
    \begin{align}
        a_i(z)^{-1} &= z - \sum_{j=1}^{K}  \lambda^{-1}\pi(i)^{-1}\alpha_ip_{i,j} a_{K+j}(z)\\ 
        a_{i+K}(z)^{-1} &= z - \sum_{j=1}^{K} \lambda^{-1} \pi(j)^{-1}\alpha_i^{-1}\alpha_j^2  p_{j,i} a_{j}(z)  
    \end{align}
    for $i = 1,\ldots,K$.    
\end{theorem}

Observe that the limiting law described in \Cref{thm: SingValN} is the same as occurs for a random matrix $M$ with mean-zero independent entries matching the variance profile of $\hat{N}_X$ \cite[Theorem 6.1]{zhu2020graphon}. 
By matching variance profile, it is here meant that $\operatorname{Var}[M_{ij}]= \operatorname{Var}[\hat{N}_{X,ij}]$ for all $i,j$.
Similarly, the limiting law in \Cref{thm: SingValP} corresponds to the limiting law of a random matrix with independent entries which instead matches the variance profile of $\sqrt{n}\hat{P}_X$.  
This hints at an underlying more general universality principle, which is commonplace in random matrix theory.
Informally, universality states that the spectrum of a random matrix often only depends on the variance of the entries.

Indeed, the first step of our proof establishes a precise version of this universality statement in \Cref{cor: Stieltjes}.
The proof strategy is essentially a modification of the moment method. 
More specifically, we generalize a result from \cite{hochstattler2016semicircle} concerning the eigenvalue distributions of approximately uncorrelated random matrices to include the possibility of a variance profile. 
The explicit description of the Stieltjes transform of the limiting laws relies on \cite{zhu2020graphon}.  

The second step of our proof establishes that this general-purpose universality principle applies to the block Markovian random matrices $\hat{N}_X$ and $\hat{P}_X$. 
\Cref{prop: ApproxUncorM} contains the corresponding result.
The key difficulty is to control the dependence. 
To this end, we use a coupling-based approach which is original to this paper and in fact the main new ingredient to establish the results.
The reader is referred to \Cref{prop: MomentSketch} for a special case of \Cref{prop: ApproxUncorM} whose proof contains the key ideas.  

For future research, an investigation into what happens when the sample path is much shorter, i.e., $\ell = o(n^2)$, can be considered. 
We anticipate that the results of this paper can be extended to such regimes, in which the empirical frequency matrix is sparser.
Nontrivial modifications would however be required in the part of the proof which relies on the moment method. 
This is because, in a sparse random matrix, normalizing for the variance causes all higher moments of the entries to diverge.
This issue can already be observed at the level of a scalar random variable. 
Namely consider a sequence of Bernoulli random variables $\xi_n$ with probability of success $p_n= o(1)$ as $n$ tends to infinity, and set $\zeta_n := (\xi_n - p_n)/\sqrt{p_n(1 - p_n)}$.
Then $\operatorname{Var}[\zeta_n] = 1$ whereas $\mathbb{E}[\zeta_n^3] = \Theta(p_n^{-1/2})$ diverges.

\section{Related literature}
\label{sec: RelatedLiterature}

\subsection{Block Markov chains}
\label{sec: NumCom}

Block Markov chains are the Markov chain analogue of the stochastic block model, and can similarly be used as a benchmark to investigate community detection problems. Community detection has been studied extensively within the context of the stochastic block model, and we refer interested readers to \cite{gao2017achieving} for an overview. Community detection problems within the context of block Markov chains received attention more recently \cite{du2019mode,duan2019state,sanders2020clustering,zhang2020spectral,zhou2020optimal,zhu2021learning}.

Spectral clustering algorithms for learning low-rank structures in Markov chains from trajectories that specifically utilize the sample frequency matrix $\hat{N}_X$ or the empirical transition matrix $\hat{P}_X$ have been analyzed in \cite{zhang2020spectral,sanders2020clustering}. 
The fact that $\hat{N}_X$ and $\hat{P}_X$ have previously been used in the context of community detection algorithms for Markov chains is also what motivated us to consider these two specific matrices.
One could in principle however also build different matrices from the data.
Appropriately modified, the methods of the current paper should still apply and thereby allow one to derive the associated singular value distributions.

In order to compare algorithmic performance to an information-theoretical lower bound on the detection error rate satisfied under any clustering algorithm, \cite{sanders2020clustering} required a sufficiently sharp upper bound to the largest singular value of $\hat{N}_X - \mathbb{E}[\hat{N}_X]$.
The singular values of $\hat{N}_X$ were recently also considered in \cite{sanders2021spectral}.  
It is established there that $\hat{N}_X$ has $K$ informative singular values of size $\Theta(\ell/n)$, and that the remaining $n-K$ singular values are $O(\sqrt{\ell/n})$. 
Besides the dense regime $\ell = \Omega(n^2)$, the sparser regimes $\ell = \Omega(n\ln n)$ and $\ell = \omega(n)$ are also considered in \cite{sanders2020clustering,sanders2021spectral}.

\subsection{Random matrices generated by stochastic processes}
\label{sec: randomstochastic}

The spectral distributions of matrices whose entries are sampled by means of a stochastic process, such as a Markov chain, were considered in \cite{pfaffel2012limiting,friesen2013semicircle,chakrabarty2016random,peligrad2016limiting,lowe2018semicircle,lowe2019limiting}.
These results are similar to ours in that the randomness is due to the sampling noise of the stochastic process, but differ in the precise construction of the matrix. 

Sample covariance matrices for time series have been considered in \cite{jin2009limiting,yao2012note,pfaffel2012limitingb,banna2015limitingb,banna2016limiting,merlevede2016empirical}. 
Let us note that covariance matrices of time series can also be viewed as an instance of the aforementioned study of random matrices with entries sampled from a stochastic process by consideration of the entries of the data matrix.   
Sample autocovariance matrices of time series have been considered in \cite{liu2015marvcenko,li2015singular,bhattacharjee2016large,bose2020smallest,yao2020eigenvalue,bose2021spectral}.

\subsection{Coupling arguments}

The critical new ingredient in our proof are the coupling arguments which are used to establish \Cref{prop: ApproxUncorM}. 
Coupling arguments are a natural way to deal with the dependence in a Markov chain. 
They have been used in this setting since the seminal paper \cite{doeblin1938expose}. 
Coupling arguments in random matrix theory are however not common place; exceptions we are aware of are \cite{banna2016bernstein,banna2016limiting}.

\subsection{Approximately uncorrelated random matrices}
\label{sec: Dependence}

A class of self-adjoint random matrices with dependent entries and decaying covariances, called approximately uncorrelated random matrices, was studied in \cite{hochstattler2016semicircle}. 
The authors establish that the empirical eigenvalue distribution of an approximately uncorrelated random matrix converges weakly in probability to the semicircular law.   
Our \Cref{cor: Stieltjes} generalizes their approach to admit the possibility that not all entries have the same variance.
Improved results to convergence weakly almost surely are established in \cite{fleermann2021almost,catalano2022random} under additional assumptions. 
It would be interesting to establish similar results in the presence of variance profiles (see the remarks preceding \Cref{prop: PowerSharperSharper}).

\subsection{Random matrices with a variance profile}
\label{sec: variance}

Let $M$ be a random matrix with independent centered entries and variance profile $S_{ij} := \operatorname{Var}[M_{ij}]$. 
The classical results on the spectral properties of random matrices assume that this variance profile is constant, but extensions to nonconstant variance profiles have also been considered \cite{hachem2006empirical,far2008slow,ding2014spectral,aljadeff2015eigenvalues,ding2015some,adhikari2021linear}. 

It is typically necessary to assume that the variance profile has tractable asymptotic behavior.
One notion of tractable asymptotic behavior may be found by employing notions from graphon theory \cite{cook2018non,zhu2020graphon,chatterjee2021spectral}.
In this case the variance profile converges to an integrable function $W: [0,1]^2 \to \mathbb{R}$. 
Let us note that results characterizing eigenvalue distributions in this setting historically preceded the graphon-theoretic terminology; see \cite{shlyakhtenko1996random}.
Graphon theory was originally developed in \cite{lovasz2006limits} as limiting objects for sequences of dense graphs.

Systems of self-consistent equations as in \Cref{thm: SingValN} frequently occur in the theory of random matrices with variance profiles \cite{far2008slow,shlyakhtenko1996random,zhu2020graphon}.
However, solving these equations to determine an explicit expression for the Stieltjes transform $s(z)$ is rarely possible.  
A numerical method based on an iteration of contraction maps has been developed in the field of operator-valued free probability theory \cite{helton2007operator}.

\subsection{Poisson limit theorems for Markov chains}

The variance profiles in our block Markovian setting follow from a Poisson limit theorem; see \Cref{thm: Poisson}. This is in turn deduced from a nonasymptotic Poisson approximation theorem; see \Cref{thm: PoissonNonAss}. 
Poisson limit theorems for Markov chains are a topic of study in their own right. 
We refer to \cite[Section 2.3]{schbath2000overview} and \cite[Section 5]{haydn2013entry} for an overview of the literature. 

Distinct from the literature, the emphasis of \Cref{thm: PoissonNonAss} lies on the fact that the state space is growing. 
Compare this e.g.\ to \cite{pitskel1991poisson} which concerns the number of visits to an increasingly rare cylindrical set in the sample path of a Markov chain on a \emph{fixed} finite state space.
\Cref{thm: PoissonNonAss}'s proof relies on a general Poisson approximation result for sums of dependent random variables from \cite{arratia1989two} which in turn relies on a method from \cite{chen1975poisson}.

\subsection{Random transition matrices}
There has been recent interest in the spectral properties of random walks in a random environment. 
The setting of \cite{bordenave2012circular} is to first sample a random $n\times n$ matrix $M$ of independent and identically distributed nonnegative real random variables of finite variance, and to then construct the random transition matrix $P := D^{-1} M$ with $D$ the diagonal matrix containing the row sums of $M$. 
The results then concern limiting laws for the singular value distribution and eigenvalue distribution of $\sqrt{n}P$. 
Different models for random transition matrices have been considered in \cite{zyczkowski2003random, bordenave2008spectrum,horvat2009ensemble,chafai2009aspects,chafai2010dirichlet,bordenave2011spectrum, kuhn2015random,chi2016random,bordenave2017spectrum,innocentini2018time,oliveira2019random,mosam2021breakdown}.

Our study of $\hat{P}_X$ differs from the study of $P$ in the source of randomness: the randomness in $\hat{P}_X$ is due to the observation noise in the sampled Markov chain in a deterministic environment whereas the randomness in $P$ is due to a perfect observation of a random environment. 
Our situation involves the additional subtlety that one has to deal with the dependence intrinsic to the sampling noise of a Markov chain. 
It should however be mentioned that many of the aforementioned results also concern the distribution of eigenvalues, while we study singular values.

\section{Notation and preliminaries}
\label{sec: Notation}

\subsection{Block Markov chains}
\label{sec: NotBMC}
Block Markov chains were defined in \Cref{sec: results}. 
Denote $\sigma : \mathcal{V}\to \{1,\ldots,K \}$ for the map which sends any $v\in \mathcal{V}_k$ to the cluster index $k$. 
Note that the definition of a block Markov chain implies that $\Sigma_X := (\sigma(X_t))_{t=0}^\ell$ is a Markov chain on the space of clusters $\{1,\ldots,K \}$ with transition matrix $p$.

Recall that it is assumed that $\lvert \#\mathcal{V}_k  - \alpha_k n\rvert = o(n)$ with $\alpha_k >0$ for all $k\in \{1,\ldots,K\}$.
Further, recall that the $\mathcal{V}_k$ are assumed to be nonempty for all $n$.
Hence, there exists some $\alpha_{min}\in \mathbb{R}_{>0}$, independent of $n$, such that $ \# \mathcal{V}_k > \alpha_{min}n$ for all $k\in \{1,\ldots,K\}$.

We denote $E_X := (E_{X,t})_{t=1}^\ell$ for the chain of edges $E_{X,t}:= (X_{t-1},X_t)$ associated with $X$.

\subsection{Graphs}
\label{sec: NotG}

All graphs in this paper are assumed to be finite and are allowed to have self-loops or multiple edges. 
We use the term \emph{simple} to refer to the case where self-loops or multiple edges are not allowed.  
An \emph{ordered tree} is a simple rooted tree such that every vertex is equipped with a total order on its descendants. 
The collection of all ordered trees on $k+1$ vertices is denoted by $\mathcal{T}_k$.

For any $n\in \mathbb{Z}_{\geq 1}$ we denote $\vec{E}_n$ for the set of directed edges $\{1,\ldots,n \}^2$ corresponding to the state space $\mathcal{V}:= \{1,\ldots,n \}$. 
Given a directed edge $e = (i,j)$ and an $n\times n$ matrix $M$ we denote $M_{e}:= M_{ij}$. 
For two vectors of integers $m,m'\in \mathbb{Z}^R$ we denote $m \leq m'$ if $m_i \leq m_i'$ for all $i \in \{1,\ldots,R\}$.

\subsection{Graphon theory}

A \emph{graphon} is an integrable map $W:[0,1]^2 \to \mathbb{R}$ which is symmetric, meaning that $W(x,y) = W(y,x)$ for all $x,y\in [0,1]$. 
We denote $\mathcal{W}_{0}$ for the collection of graphons $W$ such that $0\leq W(x,y) \leq 1$ for all $x,y \in [0,1]$. 

The \emph{cut norm} on a graphon $W$ is defined by 
$
    \Vert W \Vert_\square 
    := 
    \sup_{S,T\subseteq [0,1]}\vert\int_{S\times T} W(x,\allowbreak y) \allowbreak \rm{d}x \rm{d}y \vert
$
where the supremum runs over all measurable subsets $S,T$ of $[0,1]$. 
The \emph{cut metric} $\delta_\square$ on the space of graphons is defined by 
\begin{align}
    \delta_\square(W,W') := \inf_{\phi }\Vert W^\phi - W' \Vert_{\square} \label{eq: CutMet}
\end{align}
where the infinimum runs over all measure preserving bijections $\phi: [0,1]\to [0,1]$ and $W^\phi(x,y) := W(\phi(x),\phi(y))$.  

Given a symmetric matrix $M\in \mathbb{R}^{n\times n}$ one can define a graphon $W^M$ by setting 
\begin{align}
    W^M(x,y) 
    := 
    M_{ij} 
    \quad 
    \text{if } 
    (x,y) 
    \in 
    \Bigl[ \frac{i-1}{n}, \frac{i}{n} \Bigr) 
    \times \Bigl[ \frac{j-1}{n}, \frac{j}{n} \Bigr)
    .
    \label{eq: WMd}
\end{align}
The graphon $W^M$ can be assigned values on the boundaries $x=1$ and $y=1$ by extending continuously.

\subsection{Measure theory}
\label{sec: SpecM}
Recall that the \emph{Stieltjes transform} of a finite measure on $\mathbb{R}$, \emph{weak convergence in probability} of random measures, and the \emph{symmetrization} of a measure on $\mathbb{R}_{\geq 0}$ were defined in \Cref{sec: results}.
For two probability measures $\mu$ and $\nu$ on a countable space $\mathcal{V}$, we can define the \emph{total variation distance} as 
\begin{align}
    d_{\mathrm{TV}}(\mu,\nu):= \frac{1}{2}\sum_{x\in \mathcal{V}}\lvert \mu(x) - \nu(x)\rvert.
    \label{eq: defTV} 
\end{align}

Whenever the $k$th moment of a measure $\mu$ on $\mathbb{R}$ exists, $k\in \mathbb{Z}_{\geq 0}$, the $k$th moment will be denoted as $\mathfrak{m}_k(\mu):= \int x^k d\mu(x)$.  
Note that the definition of the \emph{empirical singular value distribution} of an $n\times n$ real matrix $M$ with singular values $s_1(M) \geq \ldots \geq s_n(M)$ from \cref{eq: Def_SingvalDistribution} can be rephrased as stating that $\nu_M$ is the measure on $\mathbb{R}_{\geq 0}$ given by $\nu_M := n^{-1}\sum_{i=1}^n \delta_{s_i(M)}$ where $\delta_{s_i(M)}$ denotes a point mass at $s_i(M)$.
The \emph{empirical eigenvalue distribution} $\mu_A$ of an $n\times n$ symmetric matrix $A$ with eigenvalues $\lambda_1(A)\geq \ldots \geq \lambda_n(A)$ is similarly defined as the measure on $\mathbb{R}$ given by $\mu_{A} := n^{-1}\sum_{i=1}^n \delta_{\lambda_i(A)}$. 

Denote the \emph{Hermitian dilation of $M$} by
\begin{align}
    H(M)
    := 
    \begin{pmatrix}
        0 & M \\ 
        M^{\mathrm{T}} & 0 \\
    \end{pmatrix}
    .
    \label{eq: Hdil}
\end{align}
Note that the eigenvalues of $H(M)$ are $s_1(M)$, $-s_{1}(M)$, $\ldots$, $s_n(M)$, $-s_n(M)$. 
Hence, $\mu_{H(M)} = \operatorname{sym}(\nu_{M})$ where $\operatorname{sym}(\nu_M)$ denotes the symmetrization of the measure $\nu_M$. 
\subsection{Compressed notation for conditional probability} \label{sec: Notation_TwoLineABCD}
Let $A,B,C,D\in \mathcal{F}$ be events in some probability space $(\Omega,\mathcal{F},\mathbb{P})$ with $\mathbb{P}(C\cap D)\neq 0$. 
We will on a few occasions encounter long expressions involving the associated conditional probability. 
To preserve readability we may then also use the following compressed notation 
\begin{align}
    \mathbb{P}(A,B \mid C,D) =: \mathbb{P}\bigg(\ \begin{aligned}
        A\\ B
     \end{aligned}\ \ \bigg\vert \ \ \begin{aligned}
       C\\ D
    \end{aligned}\ \bigg). 
\end{align}
Similar notation may be used for conditional expectation and unconditional probability. 
\subsection{Asymptotic notation}

We employ the usual conventions for big-$O$ notation:
Let $(x_n)_{n=0}^\infty$ and $(y_n)_{n=0}^\infty$ be two sequences of real numbers. Then $x_n = O(y_n)$ if and only if there exist $C,n_0>0$ such that $\lvert x_n \rvert \leq C\lvert y_n \rvert$ for all $n\geq n_0$. 
Similarly, $x_n = o(y_n)$ if and only if for every $C>0$ there exists some $n_0$ such that $\lvert x_n \rvert \leq C\lvert y_n \rvert$ for all $n \geq n_0$; 
and $x_n = \Omega(y_n)$ if and only if there exist $C,n_0>0$ such that $\lvert x_n \rvert \geq C\lvert y_n \rvert$ for all $n\geq n_0$. 
Finally, $x_n = \Theta(y_n)$ if and only if $x_n = O(y_n)$ as well as $x_n = \Omega(y_n)$. 

If $(x_n)_{n=0}^\infty$ depends on some parameters $a,b$, then the possible dependence of the constants on the parameters is expressed in the notation. 
For example, $x_n = O_a(y_n)$ means that there exist $C,n_0>0$ possibly dependent on $a$ but not on $b$ such that $\lvert x_n \rvert \leq C\lvert y_n \rvert$ for all $n\geq n_0$. 
When emphasizing like such, the parameters $a,b$ are assumed not to depend on $n$.

\section{Proof outline}
\label{sec: SketchOfProofs}

Our proof of \Cref{thm: SingValN} and \Cref{thm: SingValP} consists of two parts: a general-purpose universality result and a reduction argument. 

The first part of our proof is given in \Cref{sec: ApUn} where we generalize the results concerning eigenvalues of approximately uncorrelated random matrices from \cite{hochstattler2016semicircle} to admit the possibility of a variance profile. 
The corresponding general-purpose universality results are \Cref{thm: WeakProbAlUncor} and \Cref{cor: Stieltjes}.
Similar results for matrices with variance profile have previously also appeared in \cite[Theorem 3.2 and Theorem 3.4]{zhu2020graphon} under the assumption that the entries are independent.  
The combination of variance profiles with dependence, i.e., the notion of approximately uncorrelated random matrices with a variance profile, is new.

The second part of our proof is given in \Cref{sec: SketchBMC} and consists of a reduction to \Cref{cor: Stieltjes}.  
This involves two key difficulties. 
First, we need to determine the variance profiles associated with the block Markovian random matrices $\hat{N}_X$ and $\hat{P}_X$. 
These variance profiles are established using \Cref{thm: Poisson}, which states that $\hat{N}_{X,ij}$ is asymptotically Poisson distributed with a rate that depends only on the clusters to which $i$ and $j$ belong.
Second, we need to establish that the block Markovian random matrices are in the approximately uncorrelated universality regime. 
To do so, we develop a coupling argument which shows that the covariance between the number of traversals of different edges decays sufficiently quickly; see \Cref{prop: ApproxUncorM} for the corresponding result.

\subsection{Eigenvalue distributions of approximately uncorrelated random matrices with variance profile}
\label{sec: ApUn}

The results in this section concern the eigenvalues of a symmetric matrix whereas we are interested in the singular values of the matrices $\hat{N}_X$ and $\hat{P}_X$. 
To this end, let us remind the reader of the fact that the study of singular values of any matrix can be reduced to the study of the eigenvalues of a symmetric matrix by a Hermitian dilation (recall \cref{eq: Hdil}).  

\begin{definition}
    \label{def: ApproxUncorr}
    A family of symmetric random matrices $(A_n)_{n=1}^\infty$ is said to be \emph{approximately uncorrelated with variance profile} if, for any non-negative integers $0 \leq r \leq  R$ and $m_1,\ldots,m_R\in \mathbb{Z}_{\geq 0}$ with $m_i = 1$ for $i= 1,\ldots,r$, it holds that 
    \begin{align}
        \max_{\forall k\neq l : \{i_k,j_k\} \neq \{i_l,j_l \} }\left\lvert \mathbb{E}\left[A_{n,i_1j_1}^{m_1} \cdots A_{n,i_Rj_R}^{m_{R}}\right] \right\rvert = O_{m,R}(n^{-r/2})\label{eq: defAUGenMom}
    \end{align}
    and 
    \begin{align}
        \max_{\forall k\neq l : \{i_k,j_k \}\neq  \{i_l,j_l \} }\big\lvert \mathbb{E}\big[  A_{n,i_1j_1}^{2} \cdots  A_{n,i_Rj_R}^{2}\big] - {\textstyle\prod_{k=1}^R}\mathbb{E}\big[A_{n,i_kj_k}^2\big] \big\rvert = o_{R}(1)\label{eq: defAUVar}
    \end{align}
    where the maxima run over all values of $(i_1,j_1),\ldots,(i_R,j_R) \in \{1,\ldots,n \}^2$ with $\{i_k,j_k \}\neq \{i_l,j_l\}$ for all $k\neq l$.
\end{definition} 

In order to identify a limit of the empirical eigenvalue distribution $\mu_{A_n/\sqrt{n}}$ it is necessary to assume that the variance profile has tractable asymptotic behavior. 
We follow the approach taken in \cite[Theorem 3.2]{zhu2020graphon} and employ the homomorphism density.
The \emph{homomorphism density} from a simple graph $F = (V,E)$ on $V = \{1,\ldots,R\}$ to a symmetric matrix $M\in \mathbb{R}^{n\times n}$ is defined by 
\begin{align}
    t(F,M) 
    := 
    \frac{1}{n^{R}}\sum_{i_1,\ldots,i_{R} =1}^n  \prod_{\{v,w \} \in E}M_{i_v,i_w}.
    \label{eq: HomDens}
\end{align}
The name homomorphism density may be explained by the fact that if $A$ is the adjacency matrix of a graph $G$, then $t(F,A)$ counts the number of graph homomorphisms of $F$ to $G$. 
A detailed proof for the following result may be found in \Cref{sec: ProofProbAlUncor}. 

\begin{theorem}
    \label{thm: WeakProbAlUncor}
    Let $(A_n)_{n=1}^\infty$ be a family of symmetric random matrices which are approximately uncorrelated with variance profile $S_n:=(\operatorname{Var}[A_{n,ij}])_{i,j=1}^n$. 
    Assume that for every ordered tree $T\in \cup_{m=0}^\infty\mathcal{T}_m$ it holds that $t(T,S_n)$ has a limit as $n\to \infty$. 
    Then, the empirical eigenvalue distribution of $\mu_{A_n/\sqrt{n}}$ converges weakly in probability to the unique probability measure $\mu$ whose moments are given by 
    \begin{align}
        \mathfrak{m}_{2m}(\mu) 
        = 
        \sum_{T\in \mathcal{T}_m} \lim_{n\to\infty} t(T,S_n),
        \quad 
        \mathfrak{m}_{2m + 1}(\mu) = 0,
    \end{align}
    for every $m\in \mathbb{Z}_{\geq 0}$. Moreover, $\mu$ is compactly supported. 
\end{theorem} 

\begin{proof}[Proof sketch]    
    Just as in the classical moment method, the key step is to show that $\mathbb{E}[\mathfrak{m}_k(\mu_{A_n/\sqrt{n}})] = \mathfrak{m}_k(\mu) + o_k(1)$ for every $k \in \mathbb{Z}_{\geq 0}$. 
    Observe that  
    \begin{align}
        \mathbb{E}[\mathfrak{m}_{k}(\mu_{A_n/\sqrt{n}})] 
        &
        = n^{-1 - k/2}\mathbb{E}[\operatorname{Tr}(A_n^{k})] 
        \\ 
        &
        = n^{-1 - k/2}\sum_{i_1,\ldots,i_k = 1}^n \mathbb{E}[A_{n,i_1i_2} A_{n,i_2i_3} \cdots A_{n,i_{k-1}i_k} A_{n,i_k i_1}]
        \label{eq: MomentExpansion}
    \end{align}
    for every $k \in \mathbb{Z}_{\geq 0}$. 
    Given a sequence of indices $i := (i_1,\ldots,i_k,i_1)$, which occurs on the right-hand side of \cref{eq: MomentExpansion}, let $G_i := (V(i), E(i))$ denote the induced undirected graph with vertex set $V(i):=\{i_1,\ldots,i_k\}$ and edge set $E(i):= \{\{i_1,i_2 \},\{i_2,i_3 \},\ldots,\allowbreak \{i_k,i_1 \} \}$.
    Viewing $i$ as a cycle on $G_i$, let $r_1(i)$ be the number of edges which are traversed exactly once and let $r_2(i)$ be the number of edges which are traversed exactly twice.
    Note that we could in principle also define $r_3(i), r_4(i), \ldots$, but these quantities will not be relevant in the proof.
    
    Consider $P(i):= \mathbb{E}[A_{n,i_1i_2} \cdots A_{n,i_k i_1}]$.    
    In a classical application of the moment method, one has assumed that all entries $A_{n,ij}$ with $i\leq j$ are independent and centered. 
    Under such assumptions, it immediately follows that $P(i)=0$ whenever $r_1(i) >0$. 
    The entries of $A_{n}$ are however not independent in our case. 
    It may thus be that $P(i) \neq 0$. 
    Instead, one has to rely on part \cref{eq: defAUGenMom} in the definition of an approximately uncorrelated random matrix to deduce that $P(i)$ is small whenever $r_1(i)$ is large. 
    Combined with a bound on the number of terms with $r_1(i) = r$, which is stated in \Cref{lem: Combinatorics} and established in \cite{hochstattler2016semicircle}, this is still sufficient to argue that the contribution of the terms with $r_1(i)>0$ is asymptotically negligible.

    When $k = 2m+1$ is odd, the number of terms with $r_1(i)=0$ in \cref{eq: MomentExpansion} is of a smaller order than the normalizing factor $n^{-1 - k/2}$. 
    This yields that $n^{-1 - k/2}\mathbb{E}[\operatorname{Tr}(A_n^k)] = o_k(1)$ for all odd values of $k$.
    When $k=2m$ is even, the asymptotics are dominated by the contribution of those $P(i)$ for which $G_i$ is a tree and $r_2(i) = k/2$.
    This leads to the conclusion that $n^{-1 - k/2}\mathbb{E}[\operatorname{Tr}(A_n^k)] = \sum_{T\in \mathcal{T}_m} t(T,S_n) + o_k(1)$.  
\end{proof}

Note that \Cref{thm: WeakProbAlUncor} also applies to random matrices with independent entries. 
Correspondingly, the limit $\mu$ may be explicitly identified whenever it is known for independent random matrices with the same variance profile. 
The following corollary is an instance of this principle and uses the description provided in \cite[Theorem 3.4]{zhu2020graphon} for eigenvalues of matrices with independent entries. 
The details for the reduction argument are provided in \Cref{sec: ProofCorSystemA}. 

\begin{corollary}
    \label{cor: Stieltjes}
    Let $(A_n)_{n=1}^\infty$ be a family of symmetric random matrices which are approximately uncorrelated with variance profile $S_n:= (\operatorname{Var}[A_{n,ij}])_{i,j=1}^n$.
    Assume that there exists some graphon $W\in \mathcal{W}_0$ such that $\delta_{\square}(W^{S_n},W)\to 0$. 
    Then, the empirical eigenvalue distribution $\mu_{A_n/\sqrt{n}}$ converges weakly in probability to the probability measure $\mu$ whose Stieltjes transform $s(z)$ is given by 
    \begin{align}
        s(z) = \int_{0}^1 a(z,x) dx \label{eq: stiela}
    \end{align}
    where $a(z,x)$ is the unique analytic function from $\mathbb{C}^+\times [0,1]$ to $\mathbb{C}^-$ satisfying the following self-consistent equation 
    \begin{align}
        a(z,x)^{-1} = z - \int_0^1 W(x,y) a(z,y) dy.\label{eq: aselfconsistent}
    \end{align} 
\end{corollary} 

\subsection{Block Markovian random matrices are approximately uncorrelated with variance profile} 
\label{sec: SketchBMC}

The proof of \Cref{thm: SingValN} and \Cref{thm: SingValP} now amounts to a reduction to \Cref{cor: Stieltjes} which is done in three steps. 
First, in \Cref{sec: Reduction} we argue that we may recenter the matrices and we may assume that $X$ starts from its equilibrium distribution. 
Second, in \Cref{sec: Poisson} we determine the variance profiles by means of a Poisson limit theorem. 
Finally, in \Cref{sec: ApproxU} we show that the random matrices are in the approximately uncorrelated regime which is done using a coupling argument.  
The key ideas for the coupling argument are demonstrated in \Cref{sec: Examp}.

\subsubsection{Reduction to centered random matrices when starting in equilibrium}
\label{sec: Reduction} 

A sample path $(Z_t)_{t=0}^\ell$ of a Markov chain $Z$ on the state space $\mathcal{V}$ is said to have initial distribution $\iota:\mathcal{V}\to [0,1]$ if $\mathbb{P}(Z_0 = v) = \iota(v)$ for all $v\in \mathcal{V}$. 
Assume that $Z$ is irreducible and acyclic so that it has an equilibrium distribution $\Pi_Z$. 
Then, $Z$ is said to \emph{start in equilibrium} if it has initial distribution $\Pi_Z$.

Let $\hat{D}_X$ denote the $n\times n$ diagonal matrix whose $i$th diagonal value is the sum of the values on the $i$th row of $\hat{N}_X$: 
\begin{align}
    \hat{D}_{X,ii} := \sum_{j=1}^n \hat{N}_{X,ij}. \label{eq: defDx}
\end{align}
Observe that we can write the definition of the empirical transition matrix in \cref{eq: EmpTransition} as $\hat{P}_X = \hat{D}_X^{-1} \hat{N}_X$. 

The following lemma, whose proof is provided in \Cref{sec: MXQX} based on perturbative arguments, allows us to make the following two reductions. First, we may recenter $\hat{N}_X$ and pretend as if $\hat{D}_X$ is a deterministic matrix. Second, we may assume that $X$ starts in equilibrium. 

\begin{lemma}
    \label{lem: MXQX}
    Let $X:= (X_t)_{t=0}^{\ell}$ and $Y:= (Y_t)_{t=0}^\ell$ be sample paths from the block Markov chain where $X$ starts in equilibrium and $Y$ has an arbitrary initial distribution. 
    Denote $M_X := \hat{N}_X - \mathbb{E}[\hat{N}_X]$ and $Q_X := \operatorname{diag}((\ell+1)\Pi_X)^{-1} M_X$. 
    \begin{enumerate}[label = \rm{(\roman*)}]
        \item Assume that $\nu_{M_X/\sqrt{n}}$ converges weakly in probability to a probability measure $\nu$. Then, $\nu_{\hat{N}_Y/\sqrt{n}}$ converges weakly in probability to $\nu$. 
        \item Assume that $\nu_{\sqrt{n}Q_X}$ converges weakly in probability to a probability measure $\nu$. Then, $\nu_{\sqrt{n}\hat{P}_Y}$ converges weakly in probability to $\nu$. 
    \end{enumerate}  
\end{lemma}

The strategy is now to apply \Cref{cor: Stieltjes} with $A_n= \sqrt{2} H(M_X)$ or $A_n = \sqrt{2}H(nQ_X)$ where $X$ is a block Markov chain which starts in equilibrium.

\subsubsection{Determination of the variance profile}
\label{sec: Poisson}

The limiting variance profile of $\hat{N}_{X}$ can be established by a direct calculation which shows that the covariance between the different terms of $\hat{N}_{X,e} = \sum_{t = 1}^\ell \1_{E_{X,t} = e}$ is asymptotically negligible.
We instead take a different route: one that yields the stronger claim that $\hat{N}_{X,e}$ satisfies a Poisson limit theorem and is conceptually more satisfying.

\begin{theorem}
    \label{thm: Poisson} 
    Assume that $X$ starts in equilibrium. 
    Fix some $k_1, k_2 \in \{1,\ldots,K\}$ with $p_{k_1,k_2} > 0$ and let $(e_n)_{n=1}^\infty$ be a sequence of directed edges with $e_n \in \mathcal{V}_{k_1}\times \mathcal{V}_{k_2}$ for all $n$. 
    Then $\hat{N}_{X,e_n}$ converges in distribution to a Poisson distribution with rate $\lambda \pi(k_1)  \alpha_{k_1}^{-1} \alpha_{k_2}^{-1} p_{k_1,k_2}$. 
\end{theorem}

A proof is provided in \Cref{sec: PoissonLimit} where one can also find a nonasymptotic \Cref{thm: PoissonNonAss} which gives a precise upper bound on the total variation distance of $\hat{N}_{X,e_n}$ to a Poisson distribution.   
The proof relies on a reduction to a general Poisson approximation theorem for sums of dependent random variables from \cite{arratia1989two}.

Let us remark that the proof of \Cref{thm: Poisson} may also be used to derive a Poisson limit theorem in different scaling regimes than the running assumption $\ell = \Theta(n^2)$ and $\# \mathcal{V}_{k} = \Theta(n)$.
More precisely, a Poisson limit theorem holds whenever $\ell$ and $\#\mathcal{V}_{k_1}\times \#\mathcal{V}_{k_2}$ tend to infinity in such a fashion that $(\# \mathcal{V}_{k_1} \# \mathcal{V}_{k_2})^{-1} \ell$ converges to a nonzero constant. 
For example, $\hat{N}_{X,e_n}$ also satisfies a Poisson limit theorem in a block Markov chain with two clusters of size $\# \mathcal{V}_1 = \Theta(1)$ and $\# \mathcal{V}_2 = \ell =\omega(1)$ respectively.     

The variance profile of $\hat{N}_X$ now follows by a tightness argument which is provided in \Cref{sec: VarProfile}. 

\begin{corollary}
    \label{cor: VarianceProfile}
    Let $e_n$ be as in \Cref{thm: Poisson} and assume that $X$ starts in equilibrium. 
    Then, as $n$ tends to infinity, it holds that $\operatorname{Var}[\hat{N}_{X,e_n}] = \lambda \pi(k_1) \alpha_{k_1}^{-1}\alpha_{k_2}^{-1} p_{k_1,k_2} \allowbreak + o_{k_1,k_2}(1)$. 
\end{corollary}

Graphon limits for the variance profiles of $\sqrt{2}H(M_X)$ and $\sqrt{2}H(n Q_X)$ are immediate from \Cref{cor: VarianceProfile}. 
To be precise, by \cref{eq: WMd} and \cref{eq: Hdil} the variance profile of $\sqrt{2}H(M_X)$ converges to the graphon
\begin{align}
    W_{M}(x,y) 
    = 
    \begin{cases}
        2\lambda \pi(i)\alpha_{i}^{-1}\alpha_j^{-1}  p_{i,j} & \text{if } (2x,2y-1)\in [c_i,c_{i+1})\times [c_j,c_{j+1}), \\      
        2\lambda \pi(j)\alpha_{i}^{-1}\alpha_j^{-1}  p_{j,i} & \text{if } (2x-1,2y)\in [c_i,c_{i+1})\times [c_j,c_{j+1}), \\ 
        0 & \text{otherwise}
    \end{cases}   
    \label{eq: WM}
\end{align} 
with respect to the cut metric \cref{eq: CutMet}. 
Here, $c_i := \sum_{k=1}^{i-1} \alpha_k$ and $i,j \in \{1,\ldots,K \}$.
Similarly, note that $\Pi_X(v) = \pi(\sigma(v)) /\allowbreak(n\alpha_{\sigma(v)}) + o(1)$ so that the variance profile of $\sqrt{2}H(nQ_X)$ converges to the graphon $W_Q$ specified by 
\begin{align}
    W_{Q}(x,y) 
    &
    = 
    \begin{cases}
        2\lambda^{-1} \pi(i)^{-1}\alpha_i \alpha_j^{-1}  p_{i,j} & \text{if } (2x,2y-1)\in [c_i,c_{i+1})\times [c_j,c_{j+1}),\\      
        2\lambda^{-1} \pi(j)^{-1}  \alpha_{i}^{-1}\alpha_{j} p_{j,i} & \text{if } (2x-1,2y)\in [c_i,c_{i+1})\times [c_j,c_{j+1}),\\ 
        0 & \text{otherwise}.
    \end{cases} 
    \label{eq: WQ}
\end{align}

\subsubsection{Approximately uncorrelated}
\label{sec: ApproxU}

It remains to show that $H(M_X)$ and $H(nQ_X)$ are approximately uncorrelated with variance profiles. 
In fact, since $nQ_X$ is derived from $M_X$ by rescaling with $n\operatorname{diag}((\ell + 1)\Pi_X)^{-1}$, which is a deterministic diagonal matrix with entries of size $\Theta(1)$, it is sufficient to establish that $H(M_X)$ is approximately uncorrelated with variance profile.

\begin{proposition}
    \label{prop: ApproxUncorM}
    Assume that $X$ starts in equilibrium. Then the sequence of self-adjoint random matrices $H(M_X)$ is approximately uncorrelated with variance profile. 
\end{proposition}

Recall that \Cref{def: ApproxUncorr} of approximately uncorrelated random matrices with a variance profile consists of two properties, namely \cref{eq: defAUGenMom} and \cref{eq: defAUVar}. 
The proof for \Cref{prop: ApproxUncorM} given in \Cref{sec: propApproxUncorM} thus comes down to a verification of these two properties: \cref{eq: defAUGenMom} is verified in \Cref{prop: MomentsPower1Sharper} and \cref{eq: defAUVar} is verified in \Cref{prop: HigherMoments}.

The proof of \Cref{thm: SingValN} and \Cref{thm: SingValP} is then complete. 
Indeed, by using the Hermitian dilation in \cref{eq: Hdil} and the preliminary reduction from \Cref{lem: MXQX}, it is sufficient to establish limiting laws for the eigenvalues of $\sqrt{2}H(M_X)/\sqrt{2n}$ and $\sqrt{2}H(nQ_X)/\sqrt{2n}$ when $X$ starts in equilibrium. 
This case follows from \Cref{cor: Stieltjes} with the limiting variance profiles in \cref{eq: WM} and \cref{eq: WQ}.

\subsubsection{Demonstration of the coupling argument}
\label{sec: Examp}

\Cref{prop: ApproxUncorM} is the most important ingredient for our results. 
Let us provide an example for the method of proof by establishing a special case of \cref{eq: defAUGenMom}: the covariance between two entries decays at an appropriate rate.

\begin{proposition}
    \label{prop: MomentSketch}
    Assume that $X$ starts in equilibrium. Then, 
    \begin{align}
        \max_{e_1 \neq e_2} \lvert \mathbb{E}[M_{X,e_1} M_{X,e_2}]\rvert = O(n^{-1})
    \end{align}
    where the maximum runs over all pairs of distinct edges $e_1,e_2 \in \vec{E}_n$.
\end{proposition}

\begin{proof}    
    The proof is split into parts. 
    The main ideas are contained in Part \ref{prt: RewriteM} and Part \ref{prt: Couple}. 
    In Part \ref{prt: RewriteM} we observe that it is sufficient to understand how much the expectation of $\hat{N}_{X,e_2}$ changes when it is conditioned on a traversal of $e_1$ at some predetermined time. 
    This effect of conditioning on a traversal is then understood by a coupling argument in Part \ref{prt: Couple}. 
    \proofpart{Preliminary reduction to $K\geq 5$\label{prt: ReductionK5}} 

    \noindent
    We claim that there is no loss in generality in assuming that $K \geq 5$.
    The idea is to split a cluster into pieces. 

    Since $\# \mathcal{V}_K = \alpha_K n + o(n)$ it may be assumed that $\#\mathcal{V}_K \geq 5$. 
    Define a new partition of the state space $\mathcal{V}$ into nonempty sets $(\mathcal{V}_i')_{i=1}^{K+4}$ by $\mathcal{V}_i' := \mathcal{V}_i$ for $i<K$ and by taking $(\mathcal{V}_{i}')_{i=K}^{K+4}$ to be a partition of $\mathcal{V}_K$ into nonempty sets. 
    It can here be ensured that we remain in the asymptotic regime where the clusters have size $\Theta(n)$. 
    Indeed, this for instance follows if the subdivision of $\mathcal{V}_K$ is taken to be into clusters of roughly equal size so that the ratio $\# \mathcal{V}_{i}' / \#\mathcal{V}_{K}$ tends to $1/5$ for every $i\in \{K,\ldots, K+4 \}$. 
    Further define a $(K+4)\times (K+4)$ stochastic matrix $p'$ by $p_{i,j}' := (\#\mathcal{V}_{j}'/\# \mathcal{V}_{\min\{K,j\}}) p_{\min\{K,i\},\min\{K,j\}}$ for all $i,j = 1,\ldots,K+4$. 
    
    The reduction to $K\geq 5$ now follows by observing that the clusters $(\mathcal{V}_{i}')_{i=1}^{K+4}$ and cluster transition matrix $p'$ define exactly the same block Markov chain as we started with.
    Indeed, if $P_{i,j}'$ denotes the transition probabilities of the `new' block Markov chain then for any $i \in \mathcal{V}_x'$ and $j \in \mathcal{V}_{y}'$ it holds that 
    \begin{align}
        P_{i,j}' = \frac{1}{\# \mathcal{V}_{y}'}p_{x,y}' =   \frac{1}{\# \mathcal{V}_{y}'} \frac{\# \mathcal{V}_{y}'}{\#\mathcal{V}_{\min\{K,y\}}} p_{\min\{K,x\},\min\{K,y\}} = P_{i,j}. 
    \end{align} 
    
    \proofpart{Rewriting $\mathbb{E}[M_{X,e_1}M_{X,e_2}]$\label{prt: RewriteM}}
    
    \noindent
    Pick two distinct edges $e_1,e_2 \in \vec{E}_n$.
    Recall from \Cref{sec: NotBMC} that $E_{X}= (E_{X,t})_{t=1}^\ell$ denotes the induced Markov chain of edges $E_{X,t} = (X_{t-1},X_t)$. 
    It may be assumed that $e_1$ is such that $\mathbb{P}(E_{X,1} = e_1) \neq 0$, otherwise $M_{X,e_1} = 0$ and there is nothing to prove.

    Recall that $M_X = \hat{N}_X - \mathbb{E}[\hat{N}_X]$ and write $\hat{N}_{X,e_1} = \sum_{t_1 = 1}^{\ell} \1_{E_{X,t_1} = e_1}$ to find that  
    \begin{align}
        \mathbb{E}[M_{X,e_1} M_{X,e_2}] &= \mathbb{E}[\hat{N}_{X,e_1}M_{X,e_2}] - \mathbb{E}[\hat{N}_{X,e_1}]\mathbb{E}[M_{X,e_2}]\\ 
        &= \sum_{t_1= 1}^{\ell} \mathbb{P}(E_{X,t_1} = e_1) \big(\mathbb{E}[\hat{N}_{X,e_2}\mid E_{X,t_1} = e_1] - \mathbb{E}[\hat{N}_{X,e_2}]\big).
        \label{eq: SketchSum}
    \end{align}
    Note that all edges whose starting point and ending point have the same clusters as the starting point and ending point of $e_1$ respectively are equally likely to be traversed at time $t_1$. 
    There are at least $\alpha_{min}^2n^2$ such edges. 
    Hence, $\mathbb{P}(E_{X,t_1} = e_1) \leq \alpha_{min}^{-2}n^{-2}$. 
    Considering that there are $\ell = \Theta(n^2)$ terms on the right-hand side of \cref{eq: SketchSum}, it remains to be shown that $\mathbb{E}[\hat{N}_{X,e_2}\mid E_{X,t_1} = e_1] - \mathbb{E}[\hat{N}_{X,e_2}] = O(n^{-1})$ uniformly in $t_1$.
    
    \proofpart{Construction of coupled chains $(X,Y)$\label{prt: Couple}}

    \noindent
    Recall that we ensured that $K\geq 5$. 
    In particular there exists some $k\in \{1,\ldots,K \}$ such that $\mathcal{V}_k$ does not contain any endpoint of $e_1$ and $e_2$. 
    To study the difference $\mathbb{E}[\hat{N}_{X,e_2}\mid E_{X,t_1} = e_1] - \mathbb{E}[\hat{N}_{X,e_2}]$ we construct a pair of chains $(X,Y)$; see \Cref{fig: Sketch} for a visualization.

    \begin{enumerate}[label = \rm{(\roman*)}]

        \item 
        Sample an infinitely long path $\widetilde{X}:= (\widetilde{X}_t)_{t=-\infty}^\infty$ from the block Markov chain and independently sample an infinitely long path $\widetilde{Y}:= (\widetilde{Y}_t)_{t=-\infty}^\infty$ from the block Markov chain conditioned on $E_{\widetilde{Y}, t_1} = e_1$.
        Note that it is possible to sample at negative times by means of a time reversal of the Markov chain. 
        Such time reversal exists by the assumption that the Markov chain associated with $p$ is irreducible and acyclic. 
        
        \item 
        Define 
        \begin{align}
            T^- 
            &
            := 
            t_1 - \sup\{t \in \mathbb{Z}_{<t_1}: \widetilde{X}_t \in \mathcal{V}_k, \widetilde{Y}_t \in \mathcal{V}_k\}, 
            \\ 
            T^+ 
            &
            := 
            \inf\{t\in \mathbb{Z}_{>t_1}:\widetilde{X}_t\in \mathcal{V}_k, \widetilde{Y}_t \in \mathcal{V}_k\} - t_1 
        \end{align}
        and note that $T^-$ and $T^+$ are finite with probability one due to the assumption that the Markov chain associated with $p$ is irreducible and acyclic. 
        Let $L^- := \max\{0, t_1 - T^-\}$ and $L^+ := \min\{\ell , t_1 + T^+\}$.
        
        \item 
        Let $X:= (\widetilde{X}_t)_{t=0}^\ell$. Define $Y_t:= (Y_t)_{t=0}^\ell$ by $Y_t = \widetilde{Y_t}$ for $t\in \{L^-,\ldots,L^+\}$ and $Y_t = \widetilde{X}_t$ otherwise.    

    \end{enumerate}

    \begin{figure}[tb]
        \centering
        \includegraphics[width= 0.7 \textwidth]{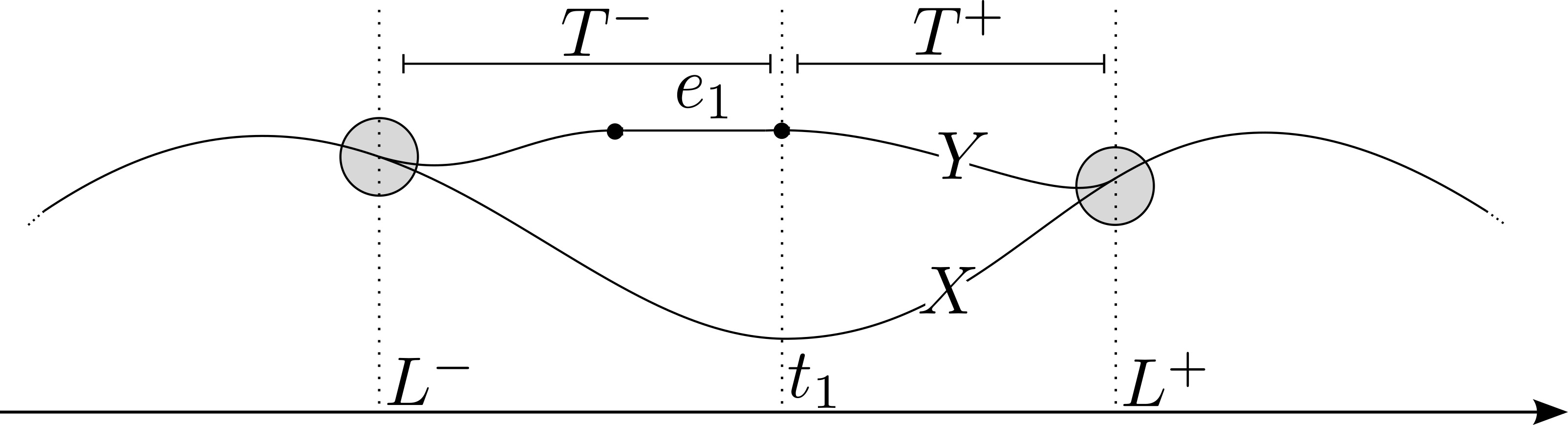}
        \caption{Visualization of the construction of the chain $Y$ in the proof of \Cref{prop: MomentSketch} by gluing an edge $e_1$ onto $X$ at time $t_1$. 
        The clusters of the block Markov chain allow us to ensure that $\mathbb{E}[L^+ - L^-]$ is $O(1)$ which causes the expected difference $\mathbb{E}[\hat{N}_{Y,e_2} - \hat{N}_{X,e_2}]$ to be small because a traversal of $e_2$ in this small period of time is unlikely.   
        }
        \label{fig: Sketch}
    \end{figure}

    By construction, $X$ is a sample path from the block Markov chain whereas $Y$ is a sample path of the block Markov chain conditioned on the event $E_{Y,t_1} = e_1$. 
    Now, by the law of total expectation 
    \begin{align}
        &
        \mathbb{E}[\hat{N}_{X,e_2}\mid E_{X,t_1} = e_1] - \mathbb{E}[\hat{N}_{X,e_2}] 
        = \mathbb{E}[ \hat{N}_{Y,e_2} - \hat{N}_{X,e_2}] 
        \\ &
        = 
        \sum_{\ell^-, \ell^+ =1}^\ell 
        \mathbb{P}(
            L^- = \ell^-,
            L^+ = \ell^+
        ) 
        \mathbb{E}[  
            \hat{N}_{Y,e_2} - \hat{N}_{X,e_2} 
            \mid 
            L^- = \ell^-,  L^+ = \ell^+
        ]
        . 
        \label{eq: DiffConditioned}
    \end{align} 

    \proofpart{$\lvert\mathbb{E}[  \hat{N}_{Y,e_2} - \hat{N}_{X,e_2}\mid L^{-} = \ell^-, L^+ = \ell^+]\rvert \leq 2\alpha_{min}^{-1}n^{-1}(\ell^+ - \ell^-)$}

    \noindent
    Let $\Delta_X := \sum_{t = L^{-}+1}^{L^+} \1_{E_{X,t} = e_2}$  denote the number of times $e_2$ was traversed by $(X_t)_{t = L^-}^{L^+}$. 
    Similarly, let $\Delta_{Y}:= \sum_{t = L^{-}+1}^{L^+} \1_{E_{Y,t} = e_2}$. 
    Since $\mathcal{V}_k$ does not contain any endpoint of $e_2$, it holds that $\hat{N}_{Y,e_2} - \hat{N}_{X,e_2} = \Delta_Y - \Delta_X$. 
    Therefore  
    \begin{align}
        \lvert\mathbb{E}[\hat{N}_{Y,e_2}& - \hat{N}_{X,e_2}\mid
            L^- = \ell^-,  L^+ = \ell^+
        ]\rvert 
        \label{eq: DiffNYNX}
        \\ 
        &
        \leq  \mathbb{E}[ \Delta_X \mid L^- = \ell^-,  L^+ = \ell^+] 
        + \mathbb{E}[ \Delta_Y \mid
            L^- = \ell^-,  L^+ = \ell^+
        ]. 
        \nonumber
    \end{align}
    We will establish a bound on the conditional expectation of $\Delta_Y$ by using its definition in terms of $\1_{E_{Y,t} = e_2}$. 
    To this end we claim that 
    \begin{align}
        \mathbb{P}(E_{Y,t}= e_2\mid L^- = \ell^-,L^+ = \ell^+) \leq \alpha_{min}^{-1}n^{-1}
        \label{eq: Pe2}
    \end{align}
    for any $t\in \{1,\ldots,\ell \}$.
    In case $t = t_1$, the left-hand side of \cref{eq: Pe2} is zero and there is nothing to prove. 
    Now consider the case where $t \neq t_1$. 
    For any edge $e$ whose starting point is equal to the starting point of $e_2$ and whose ending point is in the same cluster as the ending point of $e_2$ it holds that $\mathbb{P}(E_{Y,t} = e\mid L^- = \ell^- , L^+ = \ell^+) = \mathbb{P}(E_{Y,t} = e_2\mid L^- = \ell^- , L^+ = \ell^+)$.  
    This implies \cref{eq: Pe2} whenever $t>t_1$ since there are at least $\alpha_{min}n$ such edges $e$.
    The case $t<t_1$ may be deduced similarly by reversing the roles of the ending point and the starting point of $e$. 
    
    Combine \cref{eq: Pe2} with the fact that $\Delta_Y = \sum_{t= L^-+1}^{L^+}\1_{E_{Y,t}= e_2}$ to conclude that 
    \begin{align}
        \mathbb{E}[\Delta_{Y}\mid L^- = \ell^-, L^+ = \ell^+]
        &= \sum_{t= \ell^-+1}^{\ell^+}\mathbb{P}(E_{Y,t} = e_2\mid L^- = \ell^- , L^+ = \ell^+)\\ 
        &\leq \alpha_{min}^{-1}n^{-1}(\ell^+ - \ell^-).
    \end{align}
    The same conclusion applies to $\Delta_X$. 
    Combine finally with \cref{eq: DiffConditioned} and \cref{eq: DiffNYNX} to deduce that  
    \begin{align}
        \lvert  \mathbb{E}[\hat{N}_{X,e_2}\mid E_{X,t_1} &= e_1] - \mathbb{E}[\hat{N}_{X,e_2}] \rvert \leq 2\alpha_{min}^{-1}n^{-1} \mathbb{E}[L^+ - L^-].
        \label{eq: DiffConditioned2}
    \end{align}

    \proofpart{$\mathbb{E}[L^+ - L^-] = O(1)$}

    \noindent
    Note that $\mathbb{E}[L^+ - L^-] \leq \mathbb{E}[T^+] + \mathbb{E}[T^-]$. 
    Here, $\mathbb{E}[T^+]$ and $\mathbb{E}[T^-]$ are finite constants which do not depend on $n$. 
    Indeed, consider the product Markov chain $\Sigma_{(X,Y)}^+ := (\sigma(X_{t_1 + t}), \sigma(Y_{t_1 + t}))_{t=0}^{\infty}$ on the space $\{1,\ldots,K \}\times \{1,\ldots,K \}$. 
    Then $T^+$ is the first strictly positive time $\Sigma_{(X,Y)}^+$ is in $(k,k)$. 
    Recall that the transition dynamics $p$ for $\sigma(X_t)$ and $\sigma(Y_t)$ are assumed to be acyclic and irreducible. 
    It follows that $\mathbb{P}(T^+>t)$ shows exponential decay in $t$. 
    Because there are only $K^2$ possibilities for the initial state $\Sigma_{(X,Y),0}^+$ it follows that $\mathbb{E}[T^+] \leq  B^+ $ for some constant $B^+\in \mathbb{R}_{>0}$ which does not depend on $e_1,e_2$ or $n$. 
    A similar argument shows that $\mathbb{E}[T^-]\leq B^-$ for some $B^- \in \mathbb{R}_{>0}$.
    
    From \Cref{eq: SketchSum} and \Cref{eq: DiffConditioned2} it now follows that 
    \begin{align}
        \mathbb{E}[M_{X,e_1}M_{X,e_2}] \leq 2\ell (B^+ + B^-)\alpha_{min}^{-3} n^{-3}.
        \label{eqn:Upper_bound_to_the_expectation_of_product_MXe1MXe2}
    \end{align}
    Observe that the right-hand side of \Cref{eqn:Upper_bound_to_the_expectation_of_product_MXe1MXe2} is independent of $e_1, e_2$. 
    Since $\ell = \Theta(n^2)$ this concludes the proof.  
\end{proof}

\section{Numerical experiment on Manhattan taxi trips}
\label{sec: Taxi}

We will now demonstrate that \Cref{thm: SingValN} and \Cref{thm: SingValP} can give nontrivial predictions for the singular value distributions on an actual dataset. 
Specifically, we will analyze the first six months of 2016 in the New York City yellow cab dataset \cite{taxidata_2016}.
Each datapoint contains the pick-up and drop-off location of one trip. 
Here, pick-up locations are typically close to drop-off locations so that the dataset may be modelled as a fragmented sample path of a Markov chain.  

Spectral clusterings using the Markovian structure of this dataset have previously been analyzed in \cite{zhang2020spectral}.
Our preprocessing is similar to what was done in \cite{zhang2020spectral}, and is as follows. 
The map is subdivided into a fine grid and we trim all states which have been visited fewer than 200 times. 
We further remove all self-transitions. 
This results in a state space of size $n = 4486$ with $\ell = 55\times 10^6$ transitions.

Note that $\ell / n^2 \approx 2.7$.
This empirical observation allows us to make a relevant remark concerning our theoretical assumption $\ell = \Theta(n^2)$.
Some of our readers may namely be familiar with the literature on random graphs. 
It is an empirical observation that most real-world graphs are sparse. 
This sparsity is correspondingly a key difficulty which one should interact with in the setting of random graphs. 
Sparsity is not irrelevant in the setting of sequential data but it has a different meaning; it relates to the amount of time that the process was observed. 
As opposed to random graphs it is not unusual to encounter dense sequential data in the real world.   
The key novel difficulty is rather that sequential data has dependence. 
This is precisely the difficulty which our proofs interact with; recall the coupling argument in the proof of \Cref{prop: MomentSketch}.

A clustering $(\hat{\mathcal{V}}_k)_{k=1}^K$ is found by applying both steps from the algorithm in \cite{sanders2020clustering} with $K=4$ clusters; the result is displayed on the left-hand side of \Cref{fig: Taxi}.
Having obtained these clusters we may estimate the parameters of the block Markov chain as   
\begin{align*}
    \hat{\lambda} =\frac{\ell}{n^2},\ \hat{\alpha}_k = \frac{\# \hat{\mathcal{V}}_k}{n}, \ \hat{\pi}_k = \frac{\sum_{i \in \mathcal{V}}\sum_{j\in \hat{\mathcal{V}}_k}\hat{N}_{X,ij}}{\sum_{i \in \mathcal{V}}\sum_{j\in \mathcal{V}} \hat{N}_{X,ij}},\ \hat{p}_{k_1, k_2} = \frac{\sum_{i \in \hat{\mathcal{V}}_{k_1}}\sum_{j\in \hat{\mathcal{V}}_{k_2}} \hat{N}_{X,ij}}{\sum_{i \in \hat{\mathcal{V}}_{k_1}} \sum_{j\in \mathcal{V}}\hat{N}_{X,ij}}.  
\end{align*}
These parameters may be substituted in \Cref{thm: SingValN} and \Cref{thm: SingValP} to yield predictions for the singular value distributions of $\hat{N}_{X}$ and $\hat{P}_X$. 
The theoretical predictions and the empirical observations are displayed in \Cref{fig: Taxi}.
For comparison, we have also displayed the quarter circle law which is the universal law for the singular values of a random matrix with independent entries and identical variance. 
In other words, the quarter circle law is the prediction corresponding to $K = 1$.  

Taking into account the fact that we used just $K = 4$ clusters, we conclude that the predictions match the shape of the singular value distributions fairly well. 
Observe that the quarter circle law does not even match the general shape of the distributions: the quarter circle law has a concave density whereas the observed empirical distributions have convex densities.  

\begin{figure}[tb]
    \centering
    \begin{subfigure}[b]{0.22\textwidth}
        \centering
        \resizebox{\linewidth}{!}{
            \includegraphics{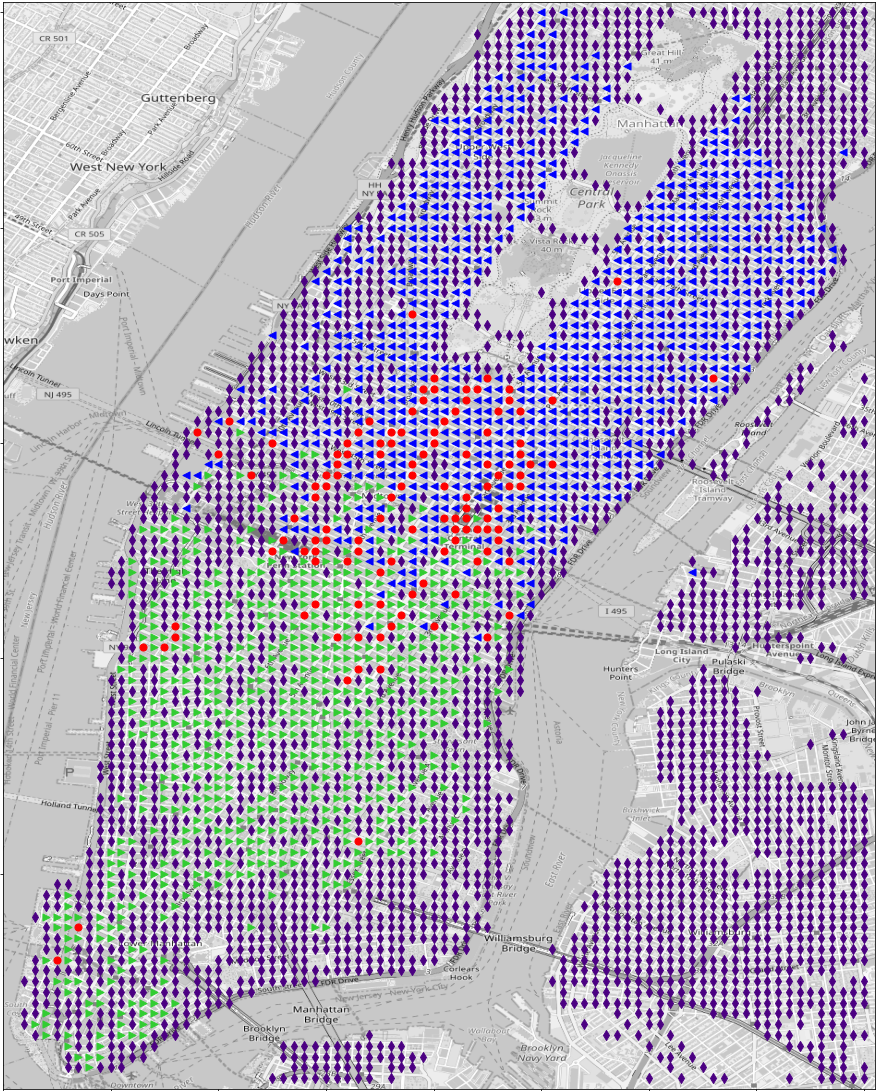}    
        } 
    \end{subfigure}
    \hfill 
    \begin{subfigure}[b]{0.38\textwidth}
        \centering
        \resizebox{.95\linewidth}{!}{\input{figures/N_Taxi.tex}}  
    \end{subfigure}
    \hfill 
    \begin{subfigure}[b]{0.38\textwidth}
        \centering
        \resizebox{.95\linewidth}{!}{\input{figures/P_Taxi.tex}}  
    \end{subfigure}
    \caption{
    On the left: Manhattan and the four clusters in purple, blue, green and red. 
    In the middle: $\hat{N}_X/\sqrt{n}$ and a frequency-based histogram of the singular values 
    compared to our theoretical predictions 
    and the quarter-circle law with density $(\pi \hat{\lambda})^{-1}(4\hat{\lambda} - x^2)^{1/2}\1_{x\in [0, 2\sqrt{\lambda}]}$. 
    On the right: $\sqrt{n}\hat{P}_X$ and its singular values compared to our theoretical predictions and the quarter-circle law with density $\pi^{-1}\hat{\lambda}(4/\hat{\lambda} - x^2)^{1/2}\1_{x\in [0,2/\sqrt{\lambda}]}$. 
    The systems of equations in \Cref{thm: SingValN} and \Cref{thm: SingValP} were solved using the algorithm in \cite[Proposition 4.1]{helton2007operator} after which we used the Stieltjes inversion forumula \cref{eq: StieltjesInv} to recover the measures. 
    }
    \label{fig: Taxi}
\end{figure}

\section{Proofs}
\label{sec: RemainingProofs}

\subsection{{Proof of Theorem \ref{thm: WeakProbAlUncor}}}
\label{sec: ProofProbAlUncor}

We work under the assumptions of \Cref{thm: WeakProbAlUncor}.
This is to say that $(A_n)_{n=1}^\infty$ is a family of symmetric $n\times n$ random matrices which are approximately uncorrelated with variance profile. 
Recall that for any fixed ordered tree $T\in \mathcal{T}_k$ it is assumed that $t(T,S_n)$ has a limit as $n\to \infty$ where $S_n= (\operatorname{Var}(A_{n,ij}))_{i,j=1}^n$ denotes the variance profile of $A_n$. 

The proof of \Cref{thm: WeakProbAlUncor} comes down to a modification of the proof in \cite{hochstattler2016semicircle} to include the variance profile.
Let us start by including some background on the moment method. 
These results are well-known and included for the reader's convenience. 

The following lemma is implicit in the proof of the Wigner semicircle law in \cite[Section 2.1]{anderson2010introduction}. 
See particularly the remarks following Lemma 2.1.7 in \cite[Section 2.1.2]{anderson2010introduction} and the application of Chebyshev's inequality in the first sentence of \cite[Section 2.1.4]{anderson2010introduction}.
\begin{lemma}
    \label{lem: ConvergenceMomentsSuffices}
    Let $(\mu_n)_{n=1}^\infty$ be a sequence of random probability measures on $\mathbb{R}$ and let $\mu$ be a deterministic and compactly supported probability measure on $\mathbb{R}$. If, for every $k\in \mathbb{Z}_{\geq 0}$, 
    \begin{align}
        \mathbb{E}[\mathfrak{m}_k(\mu_n)] = \mathfrak{m}_k(\mu) + o_k(1); 
        \quad 
        \operatorname{Var}[\mathfrak{m}_k(\mu_n)] 
        = 
        o_k(1)
        ,
    \end{align}
    then $\mu_n$ converges weakly in probability to $\mu$. 
\end{lemma}

The following result is moreover a direct consequence of the Stone--Weierstrass theorem and the Riesz representation theorem \cite[Theorem 2.14]{rudin1987real}. 

\begin{lemma}
    \label{lem: MomentProblem}

    For any compactly supported probability measure $\mu$ it holds that the sequence of moments $(\mathfrak{m}_{k}(\mu))_{k=0}^\infty$ satisfies the following properties: 
    \begin{enumerate}[label = \rm{(\roman*)}]
        \item It holds that $\mathfrak{m}_0(\mu) = 1$. 
        \item For any $k\in \mathbb{Z}_{\geq 0}$ the Hankel matrix $(\mathfrak{m}_{i+j}(\mu))_{i,j=0}^k$ is positive semi-definite.
        \item There exists a constant $c\in \mathbb{R}_{>0}$ such that $\lvert \mathfrak{m}_k(\mu) \rvert \leq c^k$ for all $k \in \mathbb{Z}_{\geq 0}$.    
    \end{enumerate}  
    Moreover, for any sequence of real number $(m_k)_{k=0}^\infty$ satisfying these properties there exists a unique probability measure $\mu$ with $(\mathfrak{m}_k(\mu))_{k=0}^\infty = (m_k)_{k=0}^\infty$ and this measure $\mu$ is compactly supported. 
\end{lemma}

\begin{lemma}\label{lem: UniqueMu}
    There exists a unique probability measure $\mu$ whose moments are given by 
    \begin{align}
        \mathfrak{m}_{2m}(\mu) = \sum_{T\in \mathcal{T}_m} \lim_{n\to\infty} t(T,S_n)
        ;
        \quad 
        \mathfrak{m}_{2m + 1}(\mu) 
        = 
        0
        ,
    \end{align}
    for every $m\in \mathbb{Z}_{\geq 0}$. Moreover, $\mu$ is compactly supported.
\end{lemma}
\begin{proof}
    The existence of such a probability measure $\mu$ is known \cite[Theorem 3.2]{zhu2020graphon}. 
    However, it is not explicitly stated in \cite{zhu2020graphon} that $\mu$ is compactly supported so, for the sake of completeness, let us provide an argument based on \Cref{lem: MomentProblem}. 

    The condition that $\mathfrak{m}_0(\mu)=1$ is satisfied since $\mu$ is a probability measure. 
    Further, the positive semi-definiteness of the Hankel matrix $(\mathfrak{m}_{i+j}(\mu))_{i,j=0}^k$ for every $k \in \mathbb{Z}_{\geq 0}$ is equivalent to the trivial statement that $\int p(x) \overline{p(x)}d\mu(x) \geq 0$ for every polynomial $p(x)\in \mathbb{C}[x]$.
    It remains to show that the moments of $\mu$ are exponentially bounded. 

    It follows from \cref{eq: defAUGenMom} in the definition of an approximately uncorrelated random matrix with a variance profile that $\max_{i,j = 1,\ldots,n}S_{n,ij} \leq \newconstant{bound}$.
    Therefore, using the definition of homomorphism densities in \cref{eq: HomDens} it holds that for every $m\in \mathbb{Z}_{\geq 0}$ and any ordered tree $T \in \mathcal{T}_m$,
    \begin{align}
        t(T,S_n) \leq  \cte{bound}^{m}.
    \end{align}
    It holds that $\# \mathcal{T}_m = C_m$ where $C_m$ is the $m$th Catalan number. 
    It is known that there exists some $\newconstant{Catalan}\in \mathbb{R}_{>0}$ such that $C_m \leq \cte{Catalan}^m$ for all $m\in \mathbb{Z}_{\geq 0}$.  
    Hence, 
    \begin{align}
        \sum_{T\in \mathcal{T}_m} t(T,S_n) \leq \cte{Catalan}^m \cte{bound}^{m}\leq \cte{CS}^{2m}
        \label{eq: CSeven}
    \end{align} 
    for some constant $\newconstant{CS}\in \mathbb{R}_{>0}$ and all $m \geq 1$.
    Equation \cref{eq: CSeven} provides an exponential bound on the rate of growth of the even moments $\mathfrak{m}_{2m}(\mu)$. 
    Further, recall that we already know that $\mathfrak{m}_0(\mu) = 1$ and $\mathfrak{m}_{2m + 1} =0$.  
    It follows that $\mathfrak{m}_k(\mu) \leq \cte{CS}^k$ for all $k\in \mathbb{Z}_{\geq 0}$.  
    Conclude by \Cref{lem: MomentProblem} that $\mu$ is compactly supported and unique. 
\end{proof}

We adapt the same notation as in the sketch of \Cref{thm: WeakProbAlUncor}. 
This is to say that given integers $i := (i_1,\ldots,i_k, i_1)$ with $i_j \in \{1,\ldots,n\}$ for every $j=1,\ldots,k$ we denote the induced undirected graph with vertex set $V(i):=\{i_1,\ldots,i_k\}$ and edge set $E(i):= \{\{i_1,i_2 \},\{i_2,i_3 \},\ldots,\{i_k,i_1 \} \}$ by $G_i = (V(i), E(i))$. 
Viewing $i$ as a cycle on $G_i$ we let $r_1(i)$ be the number of edges which are traversed exactly once and $r_2(i)$ be the number of edges which are traversed exactly twice.
Finally, we define $P(i):= \mathbb{E}[A_{n,i_1i_2} \cdots A_{n,i_k i_1}]$.

The following combinatorial result will be essential. 

\begin{lemma}[{\cite[Corollary 10]{hochstattler2016semicircle}}]
    \label{lem: Combinatorics}
    With $i$ as above it holds that $\# V(i) \leq (k + r_1(i))/2 + 1$ with strict inequality whenever $r_1(i) > 0$. 
\end{lemma}

\begin{lemma}
    \label{lem: EA}
    For any $m\in \mathbb{Z}_{\geq 0}$ it holds that
    \begin{align}
        \mathbb{E}[\mathfrak{m}_{2m + 1}(\mu_{A_n/\sqrt{n}})]&=   
            o_{m}(1),\\
        \mathbb{E}[\mathfrak{m}_{2m}(\mu_{A_n/\sqrt{n}})]&= \sum_{T\in \mathcal{T}_m} t(T,S_n) + o_{m}(1).    
    \end{align}   
\end{lemma}

\begin{proof}
    Let $k\in \mathbb{Z}_{\geq 0}$ be a positive integer and recall the expansion for $\mathbb{E}[\mathfrak{m}_k(\allowbreak \mu_{A_n/\sqrt{n}})]$ in \cref{eq: MomentExpansion}.
    Let $i:= (i_1,\ldots,i_k,i_1)$ be a sequence of integers as occurs in the right-hand side of \cref{eq: MomentExpansion}.  
    By property \Cref{eq: defAUGenMom} in the definition of an approximately uncorrelated random matrix it holds that $P(i) = O_{k}(n^{-r/2})$ whenever $r_1(i) = r$.
    By the part concerning the strict inequality in \Cref{lem: Combinatorics} we have that for any $r\in \mathbb{Z}_{>0}$ there are $o_k(n^{(k + r)/2 + 1})$ terms in the right-hand side of \cref{eq: MomentExpansion} with $r_1(i)=r$. 
    This means that only the terms with $r_1(i)=0$ survive the normalization by $n^{-1-k/2}$:
    \begin{align}
        \mathbb{E}[\mathfrak{m}_k(\mu_{A_n/\sqrt{n}})] 
        = 
        n^{-1 - k/2}\sum_{i: r_1(i)=0}^n P(i) + o_k(1)
        .
        \label{eq: MER1}
    \end{align}
    Note that $G_i$ is connected. 
    Hence, for any $i$ with $r_1(i) = 0$ it has to hold that $\# V(i) \leq k/2 + 1$ with equality if and only if $k$ is even, $G_i$ is a tree and $r_2(i) = k/2$.    
    In particular, for $k$ odd there are $o_k(n^{1 + k/2})$ terms on the right-hand side of \cref{eq: MER1}. 
    Hence, for any $m\in \mathbb{Z}_{\geq 0}$ taking $k = 2m+1$ yields that 
    \begin{align}
        \mathbb{E}[\mathfrak{m}_{2m + 1}(\mu_{A_n/\sqrt{n}})]&= o_m(1).
        \label{eq: OddM}
    \end{align} 
    Now let $k = 2m$ be even.
    As above, the contribution of the terms for which $G_i$ is not a tree or $r_2(i) \neq m$ is asymptotically negligible. 
    Consider some sequence of indices $i= (i_1,\ldots,i_{k},i_1)$ for which $G_i$ is a tree and $r_2(i) = m$. 
    An example of such a sequence $i$ is depicted in \Cref{fig: Gi}.     
    Equip $G_i$ with the unique order such that $i$ is the path traversed by depth-first search.

    \begin{figure}[tb]
        \centering
        \includegraphics[width = 0.7\textwidth]{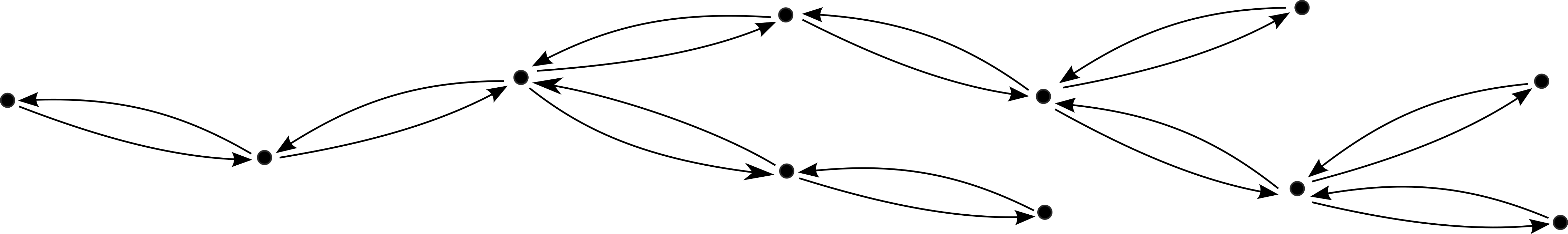}
        \caption{Visualization of a sequence of integers $i$ for which $P(i)$ has an asymptotically relevant contribution to \cref{eq: MER1}. 
        The corresponding undirected graph $G_i$ is the tree found by identifying the doubled edges.   
         } 
         \label{fig: Gi}
    \end{figure}
    
    Now, 
    \begin{align}
        \mathbb{E}[\mathfrak{m}_{2m}(\mu_{A_n/\sqrt{n}})] &= n^{-1 - (2m)/2}\sum_{T\in \mathcal{T}_m} \sum_{i: G_i \cong T} P(i)  + o_m(1)\label{eq: OddMM}
    \end{align}
    where the isomorphism condition is to be considered in the space of ordered trees.
    
    The ordering on any $T\in \mathcal{T}_m$ induces a canonical numbering of the vertex set corresponding to the order of visits in depth-first search. 
    This is to say that it can be assumed that $V(T) = \{1,\ldots,m+1\}$. 
    Denote $E(T)$ for the set of edges.
    Note that the collection of sequences of indices $i$ with $G_i \cong T$ is in bijection with the collection of injective labelings $l^i: V(T) \to \{1,\ldots,n\}$. 
    The assumption that $A_n$ is approximately uncorrelated with variance profile $S_n$ then implies that $P(i) = \prod_{\{v,w \} \in E(T)} S_{n,l_v^i l_w^i} + o_m(1)$. 
    Note that there are $O_m(n^{m+1})$ injective labelings of $V(T)$ and only $O_m(n^{m})$ labelings which are not injective. 
    Hence, the sum over all $i$ with $G_i \cong T$ is asymptotically equivalent to a sum over all (not necessarily injective) labelings $l:V(T) \to \{1,\ldots,n \}$:
    \begin{align}
        n^{-1 - (2m)/2} \sum_{i:G_i \cong T} P(i) 
        &
        = 
        n^{-1-(2m)/2}\sum_{l_1,\ldots,l_{m+1} = 1}^n \prod_{\{v,w \}\in E(T)} S_{n,l_v l_w} + o_m(1)
        \\ 
        &
        = 
        t(T, S_n) + o_m(1) 
        .
        \label{eq: TrToHomd}
    \end{align} 
    Note that \cref{eq: HomDens} was used in the second step. 
    By \cref{eq: OddMM} it now follows that for every $m\in \mathbb{Z}_{\geq 0}$
    \begin{align}
        \mathbb{E}[\mathfrak{m}_{2m}(\mu_{A_n/\sqrt{n}})] &= \sum_{T\in \mathcal{T}_m} t(T,S_n) + o_m(1).
        \label{eq: Even}
    \end{align}
    Combine \cref{eq: OddM} and \cref{eq: Even} to conclude the proof.
\end{proof}

\begin{lemma}
    \label{lem: Var}
    For any $k\in \mathbb{Z}_{\geq 0}$ it holds that
    $\operatorname{Var}[\mathfrak{m}_k(\mu_{A_n/\sqrt{n}})]= o_{k}(1).$
\end{lemma}

\begin{proof}
    Recall that
    \begin{align}
        \operatorname{Var}[\mathfrak{m}_k(\mu_{A_n/\sqrt{n}})] &= \mathbb{E}[\mathfrak{m}_k(\mu_{A_n/\sqrt{n}})^2] - \mathbb{E}[\mathfrak{m}_k(\mu_{A_n/\sqrt{n}})]^2. 
    \end{align}
    The limit of $\mathbb{E}[\mathfrak{m}_k(\mu_{A_n/\sqrt{n}})]^2$ is known by \Cref{lem: EA}. Hence, it suffices to show that $\mathbb{E}[\mathfrak{m}_k(\mu_{A_n/\sqrt{n}})^2]$ converges to the same limit.
    For any $k\in \mathbb{Z}_{\geq 0}$, 
    \begin{align}
        &
        \mathbb{E}[\mathfrak{m}_k(\mu_{A_n/\sqrt{n}})^2] 
        =
        n^{-2 - k}\mathbb{E}[(\operatorname{Tr} A_n^k)^2] 
        \\ &
        = n^{-2 - k}\hspace{-0.8em}\sum_{i_1,\ldots,i_k=1}^n \sum_{j_1,\ldots,j_k = 1}^n\hspace{-0.3em}\mathbb{E}[A_{n,i_{1}i_2}A_{n,i_2i_3}\cdots A_{n,i_k i_1} A_{n,j_1 j_2} A_{n,j_2,j_3}\cdots A_{n,j_k j_1}].
        \label{eq: ij}
    \end{align}  
    Let $i:= (i_1,\ldots,i_k,i_1)$ and $j:= (j_1,\ldots,j_k,j_1)$ be sequences of integers which occur in \cref{eq: ij}.
    Define $G_{i,j} := (V(i) \cup V(j), E(i) \cup E(j))$. Let $P(i,j):=\mathbb{E}[A_{n,i_1i_2}\cdots A_{n,j_k j_1}]$ be the term occurring in the right-hand side of \cref{eq: ij}.

    First, consider those sequences of indices $i,j$ for which $V(i) \cap V(j) \neq \emptyset$.
    Then $i,j$ can be merged: let $l_1 \in \{1,\ldots,k\}$ be the minimal index such that $i_{l_1} \in \{ j_1,\ldots,j_r\}$ and let $l_2\in \{1,\ldots,k\}$ be the minimal index such that $j_{l_2}  = i_{l_1}$ and set  
    \begin{equation}
        \ell_{i,j} 
        := 
        (i_1,i_2,\ldots,i_{l_1},j_{l_2 + 1},\ldots,j_{k},j_1,j_2,\ldots,j_{l_2}, i_{l_1 + 1},\ldots,i_k,i_1)
        .
    \end{equation}
    A resulting sequence is schematically depicted in \Cref{fig: lij}.

    \begin{figure}[tb]
        \centering
        \includegraphics[width = 0.7\textwidth]{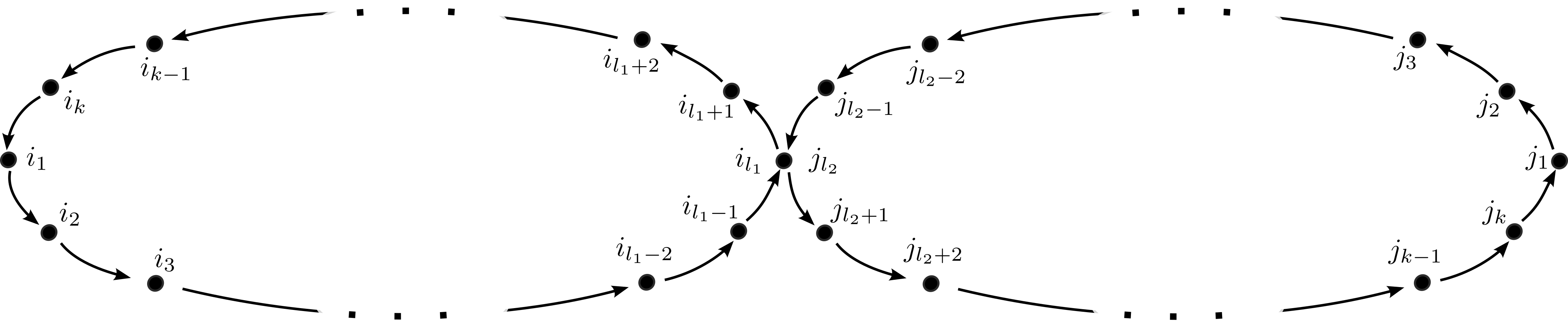}
        \caption{Visualization of a pair of sequences of integers $i,j$ with $V(i) \cap V(j)\neq \emptyset$ as occur in the proof of \Cref{lem: Var}.
        While not depicted it is possible that $i$ and $j$ are also equal at other vertices than $i_{\ell_1} = j_{\ell_2}$. 
         } 
         \label{fig: lij}
    \end{figure}
    
    By definition of approximately uncorrelated it holds that $P(i,j) = O_k(n^{-r/2})$ whenever $r_1(\ell_{i,j}) = r$. 
    Further, it follows from \Cref{lem: Combinatorics} that $\# (V(i) \cup V(j)) \leq (2k + r_1(\ell_{i,j}))/2 + 1$.  
    Hence, for any $r\in \mathbb{Z}_{\geq 0}$ there are $O_k((2k + r)/2 + 1)$ terms with $V(i) \cap V(j) \neq \emptyset$ and $r_1(\ell_{i,j}) = r$. 
    This implies that the contribution is asymptotically negligible: 
    \begin{align}
        n^{-2 - k}\sum_{i,j:V(i)\cap V(j) \neq \emptyset} P(i,j) = O_k(n^{-1}). \label{eq: ViVjnotempty}
    \end{align}

    The argument to deal with the terms with $V(i) \cap V(j)=\emptyset$ are identical to those used in the proof of \Cref{lem: EA}. 
    In particular, it can be established that 
    \begin{align}
        n^{-2 - k}\sum_{i,j : V(i) \cap V(j) = \emptyset} P(i,j) =  
        \begin{cases}
            o_k(1) \qquad & \text{ if }k \text{ is odd},\\ 
            \big(\sum_{T\in \mathcal{T}_{k/2}} t(T,S_n)\big)^2 + o_k(1)\qquad &\text{ if }k\text{ is even},
        \end{cases}
        \label{eq: ViVjempt}
    \end{align}
    where it is used that the contribution of those pairs of trees $G_i, G_j$ which were removed by restricting our attention to the case of $V(i) \cap V(j) = \emptyset$ is asymptotically negligible due to the fact that there are only $O_k(n^{k + 1})$ values of $i,j$ with $V(i) \cap V(j) \neq \emptyset$ and $r_2(i) = r_2(j) = k/2$. 

    Combine \cref{eq: ViVjnotempty}, \cref{eq: ViVjempt}, and the limit established in \Cref{lem: EA} for $\mathbb{E}[\mathfrak{m}_k(\mu_{A_n/\sqrt{n}})]^2$ to deduce that $\operatorname{Var}[\mathfrak{m}_k(\mu_{A_n/\sqrt{n}})] = o_k(1)$. 
\end{proof}
\begin{proof}[{Proof of \Cref{thm: WeakProbAlUncor}}]
    By \Cref{lem: UniqueMu} there exists a compactly supported probability measure $\mu$ with the specified moments. 
    The result is now immediate by \Cref{lem: ConvergenceMomentsSuffices} whose assumptions were verified in \Cref{lem: EA} and \Cref{lem: Var}. 
\end{proof}

\subsection{{Proof of Corollary \ref{cor: Stieltjes}}}\label{sec: ProofCorSystemA}
\begin{proof}
    Recall that it is assumed that $\delta_{\square}(W^{S_n},W)\to 0$. 
    Then, by \cite[Theorem 11.5]{lovasz2012large}, it holds that $t(F,W^{S_n})$ converges as $n\to \infty$ for every fixed tree. 
    Consequently, by \Cref{thm: WeakProbAlUncor}, it holds that $\mu_{A_n/\sqrt{n}}$ converges weakly in probability to some limit $\mu$ and it remains to show that the Stieltjes transform of $\mu$ is given by \cref{eq: aselfconsistent}. 
    
    To this end remark that \Cref{thm: WeakProbAlUncor} also applies to random matrices with independent entries.
    Consequently, if we consider a sequence of random matrices $B_n$ with independent entries and the same variance profile as $A_n$, then also $\mu_{B_n/\sqrt{n}}$ converges weakly in probability to $\mu$. 
    The desired shape for the Stieltjes transform now follows from \Cref{thm: WeakProbAlUncor} by \cite[Theorem 3.4]{zhu2020graphon} which provides the same description for the Stieltjes transform in the case of a random matrix with independent entries.
    The assumption that $\max_{ij= 1,\ldots,n}S_{n,ij} = O(1)$ in \cite{zhu2020graphon} is satisfied since \cref{eq: defAUGenMom} in the definition of an approximately uncorrelated random matrix implies that $\max_{i,j = 1,\ldots,n}\mathbb{E}[A_{n,ij}^2] = O(1)$.

    Let us remark that it may not be clear from the statement of \cite[Theorem 3.4.]{zhu2020graphon} that the codomain of $a(z,x)$ should be $\mathbb{C}^-$. 
    This specification of codomain is however important to guarantee uniqueness of the solution to the self-consistent equation \cref{eq: aselfconsistent}; see \cite[Theorem 2.1 and subsequent remarks]{ajanki2019quadratic}. 
    The fact that the codomain should be $\mathbb{C}^-$ is implicit in the final sentence of the proof of \cite[Theorem 3.4]{zhu2020graphon} which makes reference to \cite[Theorem 2.1]{ajanki2019quadratic}.
    Let us provide some additional details for this final sentence of the proof in \cite{zhu2020graphon} so as to make the codomain explicit.

    The proof of \cite[Theorem 3.4]{zhu2020graphon} begins by specifying the function by means of a Laurent series, namely $a(z,x):=\sum_{k=0}^\infty \beta_{2k}(x)z^{-2k-1}$ for certain explicit coefficients $\beta_{2k}(x)$, in \cite[Equation (3.4)]{zhu2020graphon} after which it is deduced that \cref{eq: stiela} and \cref{eq: aselfconsistent} are satisfied by the provided Laurent series.
    On the other hand, by using a geometric series at $z=\infty$ in \cite[(2.8)]{ajanki2019quadratic}, one can see that the unique solution $\widetilde{a}(z,x)$ to \cref{eq: aselfconsistent} with codomain $\mathbb{C}^-$ also admits a Laurent series, namely $\widetilde{a}(z,x) = \sum_{k=1}^{\infty} b_{k}(x)z^{-k}$.  
    Let us warn here that the convention for the Stieltjes transform which was used in \cite{ajanki2019quadratic} differs by a minus sign from the convention which was used here and in \cite{zhu2020graphon}. 
    One should correspondingly take $m_x(z) = -a(z,x)$ when applying \cite[Theorem 2.1]{ajanki2019quadratic}. 
    This difference by a sign also explains the codomain in \cite{ajanki2019quadratic} is $\mathbb{C}^+$ whereas we claim that the appropriate codomain is $\mathbb{C}^-$. 
    
    Now observe that a Laurent series of the form $\sum_{k=1}^\infty c_{k}(x) z^{-k}$ satisfies \cref{eq: aselfconsistent} if and only if $c_1(x)= 1$, $c_2(x) = 0$, and 
    \begin{align}
        c_k(x) = \int_{0}^1 W(x,y) \sum_{j=1}^{k-2} c_{k-j-1}(x)c_{j}(y) {\mathop{}\!\mathrm{d}}y\ \text{ for all }\ k\geq 3.
    \end{align}     
    This implies that the coefficients of the Laurent series of $a(z,x)$ and $\widetilde{a}(z,x)$ have to be equal since both functions satisfy \cref{eq: aselfconsistent}. 
    It follows that $a(z,x)$ and $\widetilde{a}(z,x)$ are equal in a neighborhood of $z=\infty$ and consequently everywhere on $\mathbb{C}^+$ by the identity theorem for analytic functions. 
    This shows that the function $a(z,x)$ provided in \cite[Equation (3.4)]{zhu2020graphon} indeed has codomain $\mathbb{C}^-$. 
\end{proof}

\subsection{Proofs for {Theorems \ref{thm: SingValN} and \ref{thm: SingValP}}}

It remains to establish the claims which were announced in the proof outline in \Cref{sec: SketchBMC}. 
This means that we have to prove the following statements: (i) The preliminary reduction to centered random matrices starting in equilibrium in \Cref{lem: MXQX}. (ii) The Poisson limit \Cref{thm: Poisson} and it's \Cref{cor: VarianceProfile} concerning the variance profile of $\hat{N}_X$. (iii) The statement that $H(M_X)$ is approximately uncorrelated with variance profile from \Cref{prop: ApproxUncorM}.

\subsubsection{Proof of preliminary reductions in Lemma {\ref{lem: MXQX}}}
\label{sec: MXQX} 

Recall from \Cref{sec: SpecM} that $\nu_{A}$ denotes the empirical singular value distribution of a matrix $A$. 
\begin{lemma}
    \label{lem: perturbative}
    Let $A_n$ be a sequence of random $n\times n$ matrices such that $\nu_{A_n}$ converges weakly in probability to some probability measure $\nu$. 
    \begin{enumerate}[label = \rm{(\roman*)}]
        \item Let $B_n$ be a sequence of random $n\times n$ matrices such that $\frac{1}{n}\Vert B_n \Vert_\mathrm{F}^2$ converges to $0$ in probability. Then $\nu_{A_n + B_n}$ converges weakly in probability to $\nu$. 
        \item Let $C_n$ be a sequence of random symmetric $n\times n$ matrices such that $\frac{1}{n}\operatorname{rank}(C_n)$ converges to $0$ in probability. Then $\nu_{A_n + C_n}$ converges weakly in probability to $\nu$. 
        \item Let $D_n$ be a sequence of random diagonal $n\times n$ matrices such that $\Vert D_n - \operatorname{Id}\Vert_{\mathrm{op}}$ converges to $0$ in probability. Then $\nu_{D_nA_n}$ converges weakly in probability to $\nu$. \label{item: D}
    \end{enumerate}
\end{lemma} 

\begin{proof}
    Statement (i) follows after a Hermitian dilation \cref{eq: Hdil} from \cite[Corollary A.41]{bai2010spectral}, see also \cite[Exercise 2.4.3]{tao2011topics}. 
    Statement (ii) follows after a Hermitian dilation from \cite[Theorem A.43]{bai2010spectral}, see also \cite[Exercise 2.4.4]{tao2011topics}. 
    Let us now provide a proof for the statement (iii).

    Pick some arbitrarily small $\varepsilon>0$.
    It has to be shown that 
    \begin{align}
        \mathbb{P}\big(\lvert {\textstyle \int} f(x){\mathop{}\!\mathrm{d}}\nu_{D_nA_n}(x) - {\textstyle \int} f(x){\mathop{}\!\mathrm{d}}\nu(x) \rvert > \varepsilon\big) 
        =
        o(1)
        \label{eq: desiref}
    \end{align} 
    for every continuous bounded function $f\in \mathcal{C}_b(\mathbb{R})$.
    Since $\nu$ and $\nu_{D_nA_n}$ are probability measures we may further restrict ourselves to the case where $f$ is compactly supported. 
    This follows by consideration of $fg$ with $g$ a bump function; see e.g.\ \cite[Proof of Lemma 6.21]{kallenberg1997foundations}.
    Let $c_f \in \mathbb{R}_{>0}$ be a sufficiently large constant so that $\operatorname{supp}(f) \subseteq [-c_f,c_f]$. 
    By $f$ being compactly supported and continuous it follows that $f$ is uniformly continuous. 
    There therefore exists some $\delta >0$ such that $\lvert f(x) - f(y) \rvert < \varepsilon/2$ whenever $\lvert x-y \rvert < \delta$.
    
    Observe that $\Vert D_n \Vert_{\mathrm{op}} \leq 1 + \Vert D_n - \operatorname{Id} \Vert_{\mathrm{op}}$ and $\Vert D_n^{-1} \Vert_{\mathrm{op}}^{-1} \geq 1 - \Vert D_n - \operatorname{Id} \Vert_{\mathrm{op}}$. 
    Further, it follows from \cite[Theorem 3.3.16(d)]{roger1994topics} that 
    $
        \sigma_i(D_nA_n)
        \leq 
        \Vert D_n \Vert_{\mathrm{op}} \sigma_i(A_n)
    $
    and 
    $
        \sigma_i(D_nA_n)
        \geq 
        \Vert D_n^{-1} \Vert_{\mathrm{op}}^{-1} 
        \sigma_i(A_n) 
    $.
    When combined this yields that 
    $ 
        \sigma_i(D_nA_n) - \sigma_i(A_n)
        \leq 
        \Vert D_n - \operatorname{Id} \Vert_{\mathrm{op}} \sigma_i(A_n)
    $ 
    and 
    $ 
        \sigma_i(D_nA_n) - \sigma_i(A_n)
        \geq 
        -
        \Vert D_n - \operatorname{Id} \Vert_{\mathrm{op}} \sigma_i(A_n)
    $,
    respectively.
    This is equivalent to saying that
    \begin{align}
        \lvert \sigma_i(D_nA_n)-\sigma_i(A_n) \rvert 
        \leq 
        \Vert D_n - \operatorname{Id} \Vert_{\mathrm{op}} \sigma_i(A_n)
        .
        \label{eq: approx} 
    \end{align} 

    Since $\Vert D_n - \operatorname{Id} \Vert_{\mathrm{op}}$ converges to zero in probability it holds that 
    \begin{align}
        \mathbb{P}\big(\Vert D_n - \operatorname{Id} \Vert_{\mathrm{op}} < \min\{\delta/2c_f, 1/2\}\big) 
        =
        1 - o(1)
        . 
        \label{eq: sig}
    \end{align}
    Deduce from \cref{eq: approx} that whenever the event contained in the left-hand side of \cref{eq: sig} holds, it follows that $\sigma_i(D_nA_n) > c_f$ for any $i$ with $\sigma_i(A_n) > 2c_f$ and that $\lvert \sigma_i(D_nA_n) - \sigma_i(A_n)\rvert < \delta$ for any $i$ with $\sigma_i(A_n) \leq 2c_f$.    
    Therefore, whenever the event contained in the left-hand side of \cref{eq: sig} holds, it follows that 
    \begin{align}
        \big\lvert {\textstyle \int} f(x) {\mathop{}\!\mathrm{d}}\nu_{D_nA_n}(x) - {\textstyle \int} f(x) {\mathop{}\!\mathrm{d}}\nu_{A_n}(x)\big\rvert &\leq \frac{1}{n}\sum_{i: \sigma_i(A_n)\leq 2c_f}\lvert f(\sigma_i(D_nA_n)) - f(\sigma_i(A_n))  \rvert \nonumber \\ 
        &
        < 
        \frac{\varepsilon}{2}
        . 
        \label{eq: eps}
    \end{align}
    We used here that $\operatorname{supp}(f)\subseteq [-c_f,c_f]$. 
    Combine \cref{eq: sig} with \cref{eq: eps} and the assumption that $\nu_{A_n}$ converges to $\nu$ weakly in probability to conclude that \cref{eq: desiref} holds.
    This concludes the proof.  
\end{proof}

We will ultimately use \Cref{lem: perturbative}\ref{item: D} with $D_n := \hat{D}_X^{-1} \operatorname{diag}((\ell + 1)\Pi_X)$ to replace $\hat{D}_X$ by a deterministic matrix. 
This requires that $\hat{D}_X \approx \operatorname{diag}((\ell + 1)\Pi_X)$, which will be shown by means of a concentration inequality that follows from the short mixing times of block Markov chains.

Let $Z$ be a Markov chain on the state space $\mathcal{V}$ which is irreducible and acyclic. 
For any $\varepsilon \in [0,1)$ the \emph{$\varepsilon$-mixing time} of $Z$ is defined as $t_{\mathrm{mix}}^Z(\varepsilon) := \min\{t\in \mathbb{Z}_{\geq 0}: d(t) \leq \varepsilon\}$ where  
\begin{align}
    d(t) := \sup_{z\in \mathcal{V}}d_{\mathrm{TV}}(\mathbb{P}(Z_t = -\mid Z_0 = z),\Pi_Z)\label{eq: ddef}
\end{align}
and $d_{\mathrm{TV}}$ denotes the total variation distance defined in \cref{eq: defTV}. 
Set $t_{\mathrm{mix}}^Z:= t_{\mathrm{mix}}^Z(1/4)$. 
Observe that since $X$ is a block Markov chain it holds that $t_{\mathrm{mix}}^X \leq \max\{t_{\mathrm{mix}}^{\Sigma_X},1\}$ where $\Sigma_{X,t}:= \sigma(X_t)$ is the induced Markov chain on the clusters $\{1,\ldots,K\}$. 
Observe furthermore that the dynamics of $\Sigma_X$ are independent of $n$ by definition of a block Markov chain, so that $t_{\mathrm{mix}}^{\Sigma_X}$ is a constant.
Thus $t_{\mathrm{mix}}^X = O(1)$.

We will rely on a concentration inequality from \cite{paulin2015concentration} which we reproduce here for the reader's convenience. 
Similar proofs for concentration in block Markov chains using this concentration inequality may be found in \cite{sanders2020clustering,sanders2021spectral}.

The concentration inequality is provided in terms of an invariant $\gamma_{ps}^Z$ called the \emph{pseudo-spectral gap}: 
\begin{align}
    \gamma_{ps}^Z:= \max_{i\in \mathbb{Z}_{\geq 1}} \frac{1 - \lambda_{2}((P^*_Z)^i P_Z^i)}{i}\text{ where } P_Z^*  := \operatorname{diag}(\Pi_Z)^{-1} P_Z^{\mathrm{T}} \operatorname{diag}(\Pi_Z). 
\end{align}
Here $P_Z$ denotes the transition matrix of the Markov chain $Z$ and $P_Z^{\rm{T}}$ is the transpose of this matrix. 
The pseudo-spectral gap is closely related to the mixing time:  
\begin{align}
    \frac{1}{2t_{\mathrm{mix}}^Z}\leq \gamma_{ps}^Z \leq \frac{1 + 2\ln(2) + \ln(1/\min_{i\in \mathcal{V}} \Pi_{Z}(i))}{t_{\mathrm{mix}}^Z}
    \label{eq: PsTmix}
\end{align}
by \cite[Proposition 3.4]{paulin2015concentration} whose assumptions are satisfied because $Z$ is trivially uniformly ergodic as an irreducible and acyclic chain on a finite state space. 

For any function $f:\mathcal{V} \to \mathbb{R}$ denote $\mathbb{E}_{\Pi_Z}[f]$ and $\operatorname{Var}_{\Pi_Z}[f]$ for the expectation and variance of the random variable $f(Z_0)$, respectively, where $Z_0$ has distribution $\Pi_Z$.
The following proposition occurs in \cite{paulin2015concentration} in a more general setting with possibly infinite state spaces. 
We state the result here for the reader's convenience since it will be used multiple times. 
\begin{proposition}[{\cite[Theorem 3.4]{paulin2015concentration}}]
    \label{prop: PaulinPseudo}
    Let $Z$ be an irreducible and acyclic Markov chain on $\mathcal{V}$ starting in equilibrium. Consider a function $f:\mathcal{V} \to \mathbb{R}$ on the state space such that $\lvert f(z) - \mathbb{E}_{\Pi_Z}[f] \rvert \leq C$ for all $z\in \mathcal{V}$. 
    Then, with $S := \sum_{t = 0}^\ell f(Z_t)$, it holds that    
    \begin{align}
        \mathbb{P}(\lvert S  - \mathbb{E}[S]\rvert>r) \leq 2 \exp\left(-\frac{r^2\gamma_{\rm{ps}}^Z}{8(\ell + 1 + 1/\gamma_{\rm{ps}}^Z) \operatorname{Var}_{\Pi_Z}[f] + 20rC}\right)
    \end{align}
    for all $r\in \mathbb{R}_{\geq 0}$. 
\end{proposition} 

\begin{lemma}

    Let $X := (X_t)_{t=0}^\ell$ be a sample path of the block Markov chain with an arbitrary initial distribution. 
    Then, there exist constants $\newconstant{subexp}, \newconstant{subexp2} \in \mathbb{R}_{>0}$ not depending on the initial distribution such that 
    \begin{align}
        \mathbb{P}(\lvert \hat{D}_{X,ii} - (\ell+1)\Pi_{X}(i) \rvert > \sqrt{n}t) \leq \cte{subexp}\exp(-\cte{subexp2}t)
    \end{align}
    for all $i=1,\ldots,n$. 
\end{lemma}

\begin{proof}
    Let $\kappa\in \{1,\ldots,K\}$ be a random variable following $\pi$ which is independent of $X$. 
    Define 
    \begin{equation}
        T
        := 
        \inf
        \{
            t\in \mathbb{Z}_{\geq 1}
            : 
            \sigma(X_t) = \kappa
        \}
        .
    \end{equation}
    Observe that since the Markov chain associated with $p$ is assumed to be irreducible and acyclic, it follows that $\mathbb{P}(T > t) \leq \newconstant{ec1}\exp(-\newconstant{ec} t) $ for some $\cte{ec1},\cte{ec}\in \mathbb{R}_{>0}$ which are independent of $n$ and the initial state $X_0$. 
    
    Write $S_1 := \sum_{t=0}^{T-1} f(X_t) - \sum_{t= \ell +1}^{\ell + T} f(X_t)$ and $S_2 := \sum_{t=T}^{\ell + T} f(X_t)$ where $f(x) = \1_{x = i}$. 
    We may apply \Cref{prop: PaulinPseudo} to derive a concentration inequality for $S_2$. 
    Observe that 
    \begin{align}
        \lvert f(x) - \mathbb{E}_{\Pi_X}[f]\rvert
        \leq 
        1
        \quad
        \textnormal{and}
        \quad 
        \operatorname{Var}_{\Pi_X}[f] 
        = 
        \Pi_X(i)(1-\Pi_X(i)) \leq \cte{cpi}n^{-1}
    \end{align} 
    for some $\newconstant{cpi}\in \mathbb{R}_{>0}$. 
    Recall that $t_{\mathrm{mix}}^X = O(1)$. Correspondingly, by \cref{eq: PsTmix}, it follows that there exists some constant $\newconstant{ps}\in \mathbb{R}_{>0}$ such that $\gamma_{\rm{ps}}^X\geq \cte{ps}$. 
    By \Cref{prop: PaulinPseudo} and the fact that $\ell = \Theta (n^2)$, we find that for any $r\in \mathbb{R}_{>0}$
    \begin{align}
        \mathbb{P}(\lvert S_2 - (\ell+1)&\Pi_X(i) \rvert > \sqrt{n}r)\nonumber 
        \\ &
        \leq 
        2\exp\left(-\frac{n r^2 \gamma_{\rm{ps}}^X}{8(\ell + 1 + 1/\gamma_{\rm{ps}}^X) \operatorname{Var}_{\Pi_X}[f] + 20 \sqrt{n}r}\right)
        \\ 
        &
        \leq 
        \cte{newr2} \exp(-\cte{newr}r).
    \end{align}  
    The constant $\cte{newr}\in \mathbb{R}_{> 0}$ is chosen so that $\cte{newr} \leq n\gamma_{\rm{ps}}^X/(8 (\ell + 1 + 1/\gamma_{\rm{ps}}^X)\operatorname{Var}_{\Pi_X}[f] + 20 \sqrt{n})$ for all $n$ and the constant $\newconstant{newr2}\geq 2$ is chosen so that $\cte{newr2} \exp(-\newconstant{newr}r) \geq 1$ whenever $r\leq 1$. 
    
    Further, since $T$ has exponential decay and $f\leq 1$, there exist constants $\newconstant{newT}, \newconstant{newT2}\in \mathbb{R}_{>0}$ such that 
    \begin{align}
        \mathbb{P}(S_1 >\sqrt{n} r) &\leq \mathbb{P}(T \geq \sqrt{n}r/2)\leq\cte{newT}\exp(-\cte{newT2}\sqrt{n}r).
    \end{align}        
    The desired result now follows by the triangle inequality and the fact that $\lvert \hat{D}_{X,ii} - \sum_{i=0}^\ell f(X_t) \rvert \leq 1$.  
\end{proof}

The following corollary is immediate by the union bound and the fact that $(\ell + 1)\Pi_X$ has entries of size $\Theta(n)$. 

\begin{corollary}
    \label{cor: Pi1Exists}

    Let $X := (X_t)_{t=0}^\ell$ be a sample path from the block Markov chain with an arbitrary initial distribution. 
    Then $\Vert \operatorname{diag}((\ell+1)\Pi_X)^{-1}\hat{D}_X  - \operatorname{Id}\Vert_{\mathrm{op}}$ 
    converges to zero almost surely as $n$ tends to infinity. 
    In particular, asymptotically, $\hat{D}_{X}$ is invertible almost surely and $\hat{P}_X$ thus well-defined.  
\end{corollary}

Let us now proceed to the proof of \Cref{lem: MXQX}. Recall that \Cref{lem: MXQX} concerns a reduction to centered random matrices when starting the chain from equilibrium. 

\begin{proof}[Proof of \Cref{lem: MXQX}]
    Let $\widetilde{X}:= (\widetilde{X}_t)_{t=0}^\infty$ and $\widetilde{Y}:= (\widetilde{Y}_t)_{t=0}^\infty$ denote infinitely long sample paths of the block Markov chain where $\widetilde{X}$ starts in equilibrium and $\widetilde{Y}$ has some arbitrary initial distribution.
    Define 
    \begin{align}
        T
        := 
        \inf
        \{
            t\in \mathbb{Z}_{\geq 1}
            : 
            \sigma(\widetilde{X}_t) = \sigma(\widetilde{Y}_t) 
        \}
    \end{align}
    and observe that since the Markov chain associated with $p$ is irreducible and acyclic, it holds that $\mathbb{P}(T>t)\leq \newconstant{gam}\exp(-\newconstant{ex}t)$ for some constants $\cte{gam},\cte{ex}\in \mathbb{R}_{>0}$ which are independent of $n$.  
    Let $X := (X_t)_{t=0}^\ell$ be the path defined by $X_t = \widetilde{X}_t$ for all $t < T$ and $X_t = \widetilde{Y}_t$ for all $t \geq T$.
    Further, let $Y := (\widetilde{Y}_t)_{t=0}^\ell$ be the truncation of $\widetilde{Y}$ to length $\ell + 1$.         
    Observe that $X$ is a sample path from the block Markov chain starting in equilibrium whereas $Y$ is a sample path from the block Markov chain with an arbitrary initial distribution.
    The desired results concern the singular value distributions associated with $Y$. 

    Set 
    \begin{align}
        B_n 
        := 
        \frac{ \hat{N}_Y - \hat{N}_X }{\sqrt{n}}
        \quad 
        \textnormal{and}
        \quad
        C_{n}
        := 
        \frac{ \mathbb{E}[\hat{N}_{X}] }{\sqrt{n}}
        ,
        \label{eq: BC}
    \end{align}
    and observe that $\hat{N}_{Y}/\sqrt{n} = M_{X}/\sqrt{n} + B_n + C_n$ almost surely.
    The first claim now follows by \Cref{lem: perturbative} since $\Vert B_n \Vert_\mathrm{F}^2/n\leq (2T/\sqrt{n})^2/n$ converges to zero in probability and $\operatorname{rank}(C_n) \leq K$.   

    For the second claim set 
    $D_n := \hat{D}_{Y}^{-1} \operatorname{diag}((\ell+1) \Pi_X)$
    and observe that $\sqrt{n}\hat{P}_{Y} = D_n\sqrt{n} \operatorname{diag}((\ell+1)\Pi_X)^{-1} \hat{N}_{Y}$.
    By \Cref{cor: Pi1Exists} and the continuity of $M\mapsto M^{-1}$ in the neighborhood of $\operatorname{Id}$ it holds that $\Vert D_n - \operatorname{Id} \Vert_{\mathrm{op}}$ converges to zero in probability.  
    Hence, it is sufficient to establish that the singular value distribution of $\sqrt{n} \operatorname{diag}((\ell+1)\Pi_X)^{-1} \hat{N}_{Y}$ has the desired weak limit in probability by \Cref{lem: perturbative}\ref{item: D}.
    
    In this regard observe that, with notation as in \cref{eq: BC},  
    \begin{align}
        \sqrt{n} \operatorname{diag}((&\ell+1)\Pi_X)^{-1} \hat{N}_{Y}
        \label{eq: eq} \\ 
        &=\sqrt{n}Q_{X} + n \operatorname{diag}((\ell+1)\Pi_X)^{-1}B_n +  n\operatorname{diag}((\ell+1)\Pi_X)^{-1}C_n.   \nonumber
    \end{align}
    Here, it holds that 
    \begin{align}
        \Vert n\operatorname{diag}((\ell+1)\Pi_X)^{-1}B_n\Vert_\mathrm{F}^2/n \leq \Vert n\operatorname{diag}((\ell+1)\Pi_X)^{-1} \Vert_{\mathrm{op}}^2 \Vert B_n \Vert_{\mathrm{F}}^2/n.
        \label{eq: b}
    \end{align}
    Hence, since we already know that $\Vert B_n \Vert_{\mathrm{F}}^2/n$ converges to zero in probability and since $\Vert n\operatorname{diag}((\ell+1)\Pi_X)^{-1} \Vert_{\mathrm{op}} = \Theta(1)$ it follows that $\Vert n\operatorname{diag}((\ell+1)\Pi_X)^{-1}B_n\Vert_\mathrm{F}^2/n$ converges to zero in probability. 
    Furthermore, 
    \begin{align}
        \operatorname{rank}(n\operatorname{diag}((\ell+1)\Pi_X)^{-1}C_n) \leq \operatorname{rank}(C_n) \leq K.
        \label{eq: c}
    \end{align} 
    Apply \Cref{lem: perturbative} to \cref{eq: eq} to conclude the proof.
\end{proof}

\subsubsection{Proof of the Poisson limit Theorem {\ref{thm: Poisson}}}
\label{sec: PoissonLimit}

We will extract \Cref{thm: Poisson} from a nonasymptotic result. 
For such a nonasymptotic result one has to precisely quantify the mixing behavior of the chain of clusters $\Sigma_{X}= (\sigma(X_t))_{t=0}^\ell$. 
For our purposes this is most naturally done in terms of the relative pointwise distance. 

Let $Z$ be a Markov chain on the state space $\mathcal{V}$ which is irreducible and acyclic. 
The \emph{relative pointwise distance} $\Delta_{Z}(r)$ after $r\in \mathbb{Z}_{\geq 1}$ steps is given by  
\begin{align}
    \Delta_{Z}(r)
    := 
    \max_{x,y \in \mathcal{V}} \frac{\lvert \mathbb{P}(Z_{r} = y \mid Z_0 = x) - \Pi_Z(y) \rvert}{\Pi_Z(y)}.
    \label{eq: defRelPointwise}
\end{align}
Note that $\Delta_{Z}(r)$ is related to the quantity $d(r)$ from \cref{eq: ddef} which was used to define the mixing time. 
Indeed, a direct calculation with the definitions shows that 
\begin{align}
    0 
    \leq 
    \Delta_{Z}(r) 
    \leq 
    \frac{\max_{x,y \in \mathcal{V}} \lvert \mathbb{P}(Z_{r} = y \mid Z_0 = x) - \Pi_Z(y) \rvert}{\min_{y\in \mathcal{V}}\Pi_Z(y)} 
    = 
    (\min_{y\in \mathcal{V}}\Pi_Z(y))^{-1} d(r)
    .
    \label{eq: relativevsd}
\end{align} 

\begin{theorem}
    \label{thm: PoissonNonAss}  
    Let $X = (X_{t})_{t=0}^\ell$ be a sample path from a block Markov chain which starts in equilibrium. 
    Pick some $\varepsilon \in [0,1/2]$ and $r_0\in \mathbb{Z}_{\geq 1}$ such that $\Delta_{\Sigma_X}(r) \leq \varepsilon $ for all $r \geq r_0$.  
    Then, for any $k_1, k_2 \in \{1,\ldots,K \}$ and $e \in \mathcal{V}_{k_1} \times \mathcal{V}_{k_2}$  which is not a self-loop 
    \begin{align}
         d_{\mathrm{TV}}\bigg(\mathbb{P}(\hat{N}_{X,e} = -),& \operatorname{Poisson}\Big(\frac{\ell \pi(k_1)  p_{k_1,k_2}}{\# \mathcal{V}_{k_1} \# \mathcal{V}_{k_2}}\Big)\bigg)\\ 
         &\leq \ell \pi(k_1)p_{k_1,k_2} \bigg(\frac{4r_0 - 1}{(\#\mathcal{V}_{k_1})^2(\#\mathcal{V}_{k_2})^2} +  \frac{12\varepsilon}{\# \mathcal{V}_{k_1} \#\mathcal{V}_{k_2}}\bigg),\nonumber
    \end{align}
    and for any self-loop $e\in \mathcal{V}_{k_1} \times \mathcal{V}_{k_1}$
    \begin{align}
        d_{\mathrm{TV}}\bigg(\mathbb{P}(\hat{N}_{X,e} = -),& \operatorname{Poisson}\Big(\frac{\ell \pi(k_1)  p_{k_1,k_1}}{(\# \mathcal{V}_{k_1})^2}\Big)\bigg)\\ 
        &\leq \ell \pi(k_1)p_{k_1,k_1} \bigg(\frac{4r_0 - 1}{(\#\mathcal{V}_{k_1})^4} +  \frac{2}{(\# \mathcal{V}_{k_1})^3}+  \frac{12\varepsilon}{(\# \mathcal{V}_{k_1})^2}\bigg).
        \nonumber
    \end{align}
\end{theorem}

\begin{proof} 
    The proof consists of the following parts. 
    Part \ref{prt: notation} introduces notation which is used in \cite{arratia1989two} to quantify local and long-range dependence. The parameters quantifying the local dependence are estimated in Part \ref{prt: paramslocal}.
    Finally, the parameter quantifying long-range dependence is estimated in Part \ref{prt: paramslong}.    
    \proofpart{Notation for local and long-range dependence\label{prt: notation}}

    \noindent
    For every $t \in \{ 1,\ldots, \ell\}$ let $B_{r_0}(t):= \{t'\in \{1,\ldots, \ell \}: \lvert t' - t \rvert \leq  r_0 \}$ and consider the following parameters which quantify local and long-range dependencies:
    \begin{align}
        \mathfrak{p}_t 
        &
        := 
        \mathbb{P}(E_{X,t} = e)
        ,
        \nonumber \\ 
        \mathfrak{p}_{t,t'}
        &
        := 
        \mathbb{P}(E_{X,t} = e ,E_{X,t'} = e)
        ,
        \nonumber \\ 
        \mathfrak{s}_t
        &
        := 
        \mathbb{E}\big[\big\lvert \mathbb{P}\big(E_{X,t} = e\mid (\1_{E_{X,t' = e}})_{t' \in \{1,\ldots,\ell \}\setminus B_{r_0}(t)}\big) - \mathfrak{p}_t \big\rvert\big]
        .
    \end{align}
    Applying \cite[Theorem 1]{arratia1989two} to the sum of dependent random variables $\hat{N}_{X,e} = \sum_{t= 1}^\ell\1_{E_{X,t} = e}$ yields that 
    \begin{align}
        d_{\mathrm{TV}}&\bigg(\mathbb{P}(\hat{N}_{X,e} = -), \operatorname{Poisson}\Big(\frac{\ell \pi(k_1)  p_{k_1,k_2}}{\# \mathcal{V}_{k_1} \# \mathcal{V}_{k_2}}\Big)\bigg) \leq b_1 + b_2 + b_3.
        \label{eq: TVbs}
    \end{align} 
    where
    \begin{align}
        b_1 
        := 
        \sum_{t=1}^\ell \sum_{t'\in B_{r_0}(t)} \mathfrak{p}_t \mathfrak{p}_{t'}
        ,
        \quad
        b_2 
        := 
        \sum_{t=1}^\ell \sum_{t' \in B_{r_0}(t) \setminus \{t\}} \mathfrak{p}_{t,t'}
        ,
        \quad
        b_3 
        := 
        \sum_{t=1}^\ell \mathfrak{s}_t
        .
    \end{align}
    Here we used that the definition in \cite{arratia1989two} for total variation distance differs from our definition \cref{eq: defTV} by a factor two. 
    Indeed, observe that $\sup_{A\subseteq \mathbb{Z}_{\geq 0}} \lvert \mu(A) - \nu(A) \rvert = \frac{1}{2}\sum_{x\in \mathbb{Z}_{\geq 0}} \lvert \mu(x) - \nu(x)  \rvert$ for any two probability measures $\mu$ and $\nu$ using that the supremum is realized by $A = \{x\in \mathbb{Z}_{\geq 0}:\mu(x)\geq \nu(x) \}$.

    \proofpart{Bounding the local dependence contributions ($b_1$ and $b_2$)\label{prt: paramslocal}}

    \noindent
    By definition of $X$ as a block Markov chain it holds that for any $t\in \{1,\ldots,\ell\}$
    \begin{align}
        \mathfrak{p}_t &= \frac{\mathbb{P}(\Sigma_{X,t-1} = k_1, \Sigma_{X,t} = k_2)}{\# \mathcal{V}_{k_1}\#\mathcal{V}_{k_2}}=   \frac{\pi(k_1) p_{k_1,k_2}}{\# \mathcal{V}_{k_1} \# \mathcal{V}_{k_2}} 
        .
        \label{eq: explicitpt}
    \end{align} 
    From \cref{eq: explicitpt} it follows that 
    \begin{align}
        b_1 \leq \ell (2r_0 + 1)  \frac{\pi(k_1)^2 p_{k_1,k_2}^2 }{(\# \mathcal{V}_{k_1})^2 (\# \mathcal{V}_{k_2})^2} \leq  \ell \pi(k_1)p_{k_1,k_2} \frac{2r_0 + 1}{(\# \mathcal{V}_{k_1})^2 (\# \mathcal{V}_{k_2})^2}
        \label{eq: b1}
    \end{align}
    where it was used that $\pi(k_1) p_{k_1,k_2} \leq 1$.
    It holds for any $t,t' \in \{1,\ldots,\ell \}$ with $\lvert t - t' \rvert > 1$ that 
    \begin{align}
        \mathfrak{p}_{t,t'} &= \frac{\mathbb{P}(\Sigma_{X,t-1} = k_1, \Sigma_{X,t} = k_2, \Sigma_{X,t' - 1} = k_1, \Sigma_{X,t'} = k_2)}{(\# \mathcal{V}_{k_1})^2(\#\mathcal{V}_{k_2})^2}\\ 
        &\leq  \frac{\pi(k_1) p_{k_1,k_2}}{(\# \mathcal{V}_{k_1})^2(\#\mathcal{V}_{k_2})^2}.
    \end{align} 
    For $t' \in  \{t +1, t-1\}$ it may similarly be deduced that if $e$ is a self-loop then $\mathfrak{p}_{t,t'}  \leq \pi(k_1)p_{k_1,k_1} (\# \mathcal{V}_{k_1})^{-3}$ and if $e$ is not a self-loop then $\mathfrak{p}_{t,t'} = 0$.
    Hence, 
    \begin{align}
        b_2 \leq 
        \begin{cases}
            \ell\pi(k_1)p_{k_1,k_1} \bigl( \frac{2}{(\# \mathcal{V}_{k_1})^{3}} + \frac{2r_0 - 2}{(\# \mathcal{V}_{k_1})^4}\bigr) \quad &\text{ if }e\text{ is a self-loop,}\\ 
            \ell\pi(k_1)p_{k_1,k_2}\frac{2r_0 - 2}{(\# \mathcal{V}_{k_1})^2( \# \mathcal{V}_{k_2})^2}   &\text{ otherwise.}
        \end{cases}
        \label{eq: b2}
    \end{align}

    \proofpart{Bounding the long-range dependence contribution ($b_3$)\label{prt: paramslong}} 
    
    \noindent
    The first step in this part will be to control $\mathfrak{s}_t$ in terms of a simpler quantity; this step is achieved in \cref{eq: st}.
    Thereafter, control over \cref{eq: st} is achieved in \cref{eq: delshapeestimate} and we deduce the desired bound for $b_3$ in \cref{eq: b3}. 
    
    Recall that $X$ is uniformly distributed in $\prod_{t = 0}^\ell \mathcal{V}_{s_{X,t}}$ given that $\Sigma_{X} = s_X$. 
    It follows that $\1_{E_{X,t}= e}$ is conditionally independent of $(\1_{E_{X,t} = e})_{t'\in \{1,\ldots,\ell \}\setminus B_{r_0}(t)}$ given $\Sigma_{X,t - r_0-1}$ and $\Sigma_{X,t + r_0}$. 
    (If $t-r_0-1<0$ one can use a time reversal to make sense of $\Sigma_{X,t - r_0-1}$; a time reversal exists since $\Sigma_X$ is assumed to be irreducible and acyclic.)
    Hence, by the law of total probability 
    \begin{align}
        &
        \mathbb{P}\bigl(E_{X,t} = e\mid(\1_{E_{X,t} = e})_{t'\in \{1,\ldots,\ell \}\setminus B_{r_0}(t)}\bigr)\label{eq: totalpb}\\ 
        &
        =  \sum_{s_1,s_2 \in \{1,\ldots,K\}} \mathbb{P}\bigl(\Sigma_{X,t-r_0 -1} = s_1,\Sigma_{X,t + r_0} =s_2 \mid (\1_{E_{X,t} = e})_{t'\in \{1,\ldots,\ell \}\setminus B_{r_0}(t)}\bigr)\nonumber \\ 
        &
        \hphantom{\sum_{s_1,s_2 \in \{1,\ldots,K\}}\mathbb{P} } \times\mathbb{P}(E_{X,t} = e\mid \Sigma_{X,t - r_0-1} = s_1, \Sigma_{X,t + r_0} =s_2 ).
        \nonumber   
    \end{align}
    On the other hand, we may write 
    \begin{align}
        &\mathfrak{p}_t 
        \\ 
        &= 
        \mathfrak{p}_t \sum_{s_1,s_2 \in \{1,\ldots,K\}} \mathbb{P}\bigl(\Sigma_{X,t-r_0 -1} = s_1,\Sigma_{X,t + r_0} =s_2 \mid (\1_{E_{X,t} = e})_{t'\in \{1,\ldots,\ell \}\setminus B_{r_0}(t)}\bigr).\nonumber 
    \end{align}
    Observe that 
    \begin{align}
        \mathbb{E}\bigl[\mathbb{P}(\Sigma_{X,t-r_0 -1} = s_1,\Sigma_{X,t + r_0} =s_2 \mid &(\1_{E_{X,t} = e})_{t'\in \{1,\ldots,\ell \}\setminus B_{r_0}(t)})\bigr]\label{eq: e}    \\ 
        &= \mathbb{P}(\Sigma_{X,t-r_0 -1} = s_1,\Sigma_{X,t + r_0} =s_2). \nonumber
    \end{align}
    Therefore, by using the triangle inequality on $\mathfrak{s}_t$'s definition together with \cref{eq: totalpb}--\cref{eq: e}, it follows that  
    \begin{align}
        \mathfrak{s}_t\leq \sum_{s_1,s_2 \in \{1,\ldots,K\}}&\mathbb{P}(\Sigma_{X,t-r_0 -1} = s_1,\Sigma_{X,t + r_0} =s_2)\label{eq: st}\\ 
        &\times \bigl\lvert \mathbb{P}(E_{X,t} = e\mid \Sigma_{X,t - r_0-1} = s_1, \Sigma_{X,t + r_0} =s_2 ) -\mathfrak{p}_t\bigr\rvert.
        \nonumber 
    \end{align}  
    Observe that for any $s_1,s_2 \in \{1,\ldots,K \}$ with $\mathbb{P}( \Sigma_{X,t + r_0 } = s_2 \mid \Sigma_{X,t - r_0-1} = s_1)> 0$, the definition of conditional probability yields that 
    \begin{align}
        \mathbb{P}(E_{X,t} = e\mid \Sigma_{X,t - r_0-1} = s_1,& \Sigma_{X,t + r_0 } = s_2 )\label{eq: rr} \\ 
        &= \frac{\mathbb{P}(E_{X,t} = e, \Sigma_{X,t + r_0 
        } = s_2\mid \Sigma_{X,t - r_0-1} = s_1)}{\mathbb{P}( \Sigma_{X,t + r_0 } = s_2 \mid \Sigma_{X,t - r_0-1} = s_1)}.
        \nonumber
    \end{align} 
    Here, by definition of a block Markov chain  
    \begin{align}
        \mathbb{P}(E_{X,t} = e,& \Sigma_{X,t + r_0 } = s_2\mid \Sigma_{X,t - r_0-1} = s_1)\label{eq: 2rr}\\ 
        &=\frac{\mathbb{P}(\Sigma_{X,t-1} = k_1, \Sigma_{X,t} = k_2, \Sigma_{X,t+r_0} = s_2 \mid \Sigma_{X,t-r_0-1} = s_1)}{\#\mathcal{V}_{k_1}\# \mathcal{V}_{k_2}}.
        \nonumber
    \end{align}
    Moreover, the Markov property applied to $\Sigma_{X}$ yields that 
    \begin{align}
        \mathbb{P}(\Sigma_{X,t-1}& = k_1, \Sigma_{X,t} = k_2, \Sigma_{X,t+r_0} = s_2 \mid \Sigma_{X,t-r_0-1} = s_1)\label{eq: MarkovProp}\\ 
        &= \mathbb{P}(\Sigma_{X_{t-1} }= k_1 \mid \Sigma_{X,t-r_0-1}=s_1)p_{k_1,k_2}\mathbb{P}(\Sigma_{X,t+r_0}=s_2\mid \Sigma_{X,t} = k_2).
        \nonumber 
    \end{align}
    Combining \cref{eq: rr}, \cref{eq: 2rr} and \cref{eq: MarkovProp} yields that 
    \begin{align}
        &
        \mathbb{P}(
            E_{X,t} = e\mid \Sigma_{X,t - r_0-1} = s_1,\Sigma_{X,t + r_0 } = s_2 
        )
        \label{eq: predelta}
        \\ &
        = \frac{p_{k_1,k_2}}{\# \mathcal{V}_{k_1}\# \mathcal{V}_{k_2}} \frac{\mathbb{P}(\Sigma_{X,t-1} = k_1 \mid \Sigma_{X,t-r_0-1}=s_1)\mathbb{P}(\Sigma_{X,t+r_0}=s_2\mid \Sigma_{X,t} = k_2)}{\mathbb{P}( \Sigma_{X,t + r_0 } = s_2 \mid \Sigma_{X,t - r_0-1} = s_1)}.
        \nonumber
    \end{align}

    Let us introduce the following notation: 
    \begin{align}
        \delta_{s_1, s_2}^{(a)} 
        &
        := 
        \frac{\mathbb{P}(\Sigma_{X,t + r_0 } = s_2\mid \Sigma_{X,t - r_0-1} = s_1) - \pi(s_2)
        }{\pi(s_2)},
        \label{eqn:delta_a}
        \\ 
        \delta_{s_1,k_1}^{(b)} 
        &
        := 
        \frac{\mathbb{P}(\Sigma_{X,t-1} = k_1 \mid \Sigma_{X,t-r_0-1} = s_1) - \pi(k_1)}{\pi(k_1)},
        \label{eqn:delta_b}
        \\ 
        \delta_{k_2,s_2}^{(c)} 
        &
        := 
        \frac{\mathbb{P}(\Sigma_{X,t+r_0} = s_2 \mid \Sigma_{X,t} = k_2) - \pi(s_2)}{\pi(s_2)}
        \label{eqn:delta_c}
        .  
    \end{align}
    Substituting \cref{eqn:delta_a}--\cref{eqn:delta_c} into \cref{eq: predelta} yields 
    \begin{align}
        \mathbb{P}(E_{X,t} = e\mid \Sigma_{X,t - r_0-1} =& s_1, \Sigma_{X,t + r_0 } = s_2 ) \\ 
        &= 
       \frac{\pi(k_1)p_{k_1,k_2}}{\# \mathcal{V}_{k_1} \# \mathcal{V}_{k_2}}\frac{(1 + \delta_{s_1,k_1}^{(b)}) (1 + \delta_{k_2,s_2}^{(c)} )}{1 + \delta_{s_1, s_2}^{(a)}}.  \label{eq: delshapeestimate}
    \end{align}
    Note that by our assumption and the definition of relative pointwise distance, $\lvert \delta_{s_1,s_2}^{(a)} \rvert \leq \varepsilon \leq 1/2$.
    Hence by Taylor's theorem 
    \begin{align}
        \bigg\lvert 1 - \frac{1}{1 + \delta_{s_1, s_2}^{(a)}}  \bigg\rvert \leq \varepsilon \max_{x \in [-1/2,1/2]} \bigg(\frac{1}{1 + x}\bigg)^2 \leq 4 \varepsilon.
        \label{eq: es1}
    \end{align}
    Similarly, it holds that $\lvert \delta_{s_1,k_1}^{(b)} \rvert \leq \varepsilon \leq 1/2$ and $\lvert \delta_{k_2,s_2}^{(c)}  \rvert\leq \varepsilon \leq 1/2$. 
    In particular $\lvert (1 + \delta_{s_1,k_1}^{(b)}) (1 + \delta_{k_2,s_2}^{(c)})\rvert\leq 9/4$ and 
    \begin{align}
        \lvert 1 - (1 + \delta_{s_1,k_1}^{(b)})(1 + \delta_{k_2,s_2}^{(c)}) \rvert \leq 3 \varepsilon. \label{eq: es2} 
    \end{align} 
    Bound \cref{eq: delshapeestimate} using \cref{eq: explicitpt}, \cref{eq: es1}, \cref{eq: es2} to find that 
    \begin{align}
        \lvert \mathbb{P}(E_{X,t} = e\mid \Sigma_{X,t - r_0-1} = s_1, \Sigma_{X,t + r_0} = s_2 ) -\mathfrak{p}_t\rvert \leq  \frac{\pi(k_1)p_{k_1,k_2}}{\#\mathcal{V}_{k_1} \# \mathcal{V}_{k_2}} \cdot 12 \varepsilon.
    \end{align}
    Hence by \cref{eq: st} and the definition of $b_3$,
    \begin{align}
        b_3 \leq \ell\pi(k_1)p_{k_1,k_2}\frac{12  \varepsilon}{\#\mathcal{V}_{k_1} \# \mathcal{V}_{k_2}}.
        \label{eq: b3}
    \end{align}
    This concludes the proof: the estimates for $b_1, b_2$ and $b_3$ from \cref{eq: b1}, \cref{eq: b2} and \cref{eq: b3} can be substituted in \cref{eq: TVbs} to yield the desired result. 
\end{proof}

\begin{proof}[Proof of Theorem \ref{thm: Poisson}]
    Note that the total variation distance in \cref{eq: defTV} metrizes convergence in distribution. 
    This is to say that $\hat{N}_{X,e_n}$ converges in distribution if and only if $\mathbb{P}(\hat{N}_{X,e_n} = -)$ converges with respect to the metric $\operatorname{d}_{\mathrm{TV}}$.

    Recall that $\Sigma_{X}$ is assumed to be irreducible and acyclic. 
    Hence, by \cref{eq: relativevsd} it holds that $\Delta_{\Sigma_X}(r)$ decays exponentially in $r$. 
    Recall that $\ell = \lambda n^2 + o(n^2)$ and $\# \mathcal{V}_{k}  = \alpha_kn + o(n)$. 
    In particular $(\# \mathcal{V}_{k_1} \# \mathcal{V}_{k_2})^{-1}\ell$ converges to $\alpha_{k_1}^{-1} \alpha_{k_2}^{-1}\lambda$ as $n$ tends to infinity.   
    The result now follows by taking $r_0 = \log(\# \mathcal{V}_{k_1} \# \mathcal{V}_{k_2})$ in \Cref{thm: PoissonNonAss}. 
\end{proof}

\subsubsection{Proof of {Corollary \ref{cor: VarianceProfile}} on the variance profile}
\label{sec: VarProfile}  

Due to the decomposition $\operatorname{Var}[\hat{N}_{X,e_n}] = \mathbb{E}[\hat{N}_{X,e_n}^2] - \mathbb{E}[\hat{N}_{X,e_n}]^2$ the following result implies \Cref{cor: VarianceProfile}. 
\begin{lemma}
    \label{lem: GeneralMomentP}
    Assume that $X$ starts in equilibrium. Let $e_n$ be as in \Cref{thm: Poisson}, and $Y$ a Poisson distributed random variable with rate $\lambda \pi(k_1)\alpha_{k_1}^{-1} \alpha_{k_2}^{-1} p_{k_1,k_2}$. 
    Then for every $m\in \mathbb{Z}_{\geq 0}$ it holds that, as $n$ tends to infinity, 
    \begin{align}
        \mathbb{E}\left[\hat{N}_{X,e_n}^m\right] = \mathbb{E}\left[Y^m\right] + o_m(1).
    \end{align}
\end{lemma}
\begin{proof}
    We have already established in \Cref{thm: Poisson} that $\hat{N}_{X,e_n}$ converges in distribution to $Y$. 
    Hence, to derive convergence of the moments it suffices to verify that for every $m\in \mathbb{Z}_{\geq 0}$ it holds that the sequence of random variables $\hat{N}_{X,e_n}^{m+1}$ is uniformly integrable \cite[7.10. (15)]{grimmett2020probability}. 
     
    We will apply \Cref{prop: PaulinPseudo} to the Markov chain of edges $E_X := (E_{X,t})_{t=1}^\ell$ with the function $f(e) = \1_{e = e_n}$.
    Observe that 
    \begin{align}
        \lvert f - \mathbb{E}_{\Pi_{E_X}}[f] \rvert\leq 1; \qquad \operatorname{Var}_{\Pi_{E_X}}[f]= \Pi_{E_X}(e_n)(1 - \Pi_{E_X}(e_n)) \leq \alpha_{min}^{-2}n^{-2}.
    \end{align}
    Note that $E_X$ induces a chain $\Sigma_{E_X}$ on the reduced space $\{1,\ldots,K \}\times \{1,\ldots,K \}$ and that $t_{\mathrm{mix}}^{E_X} \leq \max\{2, t_{\mathrm{mix}}^{\Sigma_{E_X}} \}$ by $X$ being a block Markov chain.
    Since $X$ is a block Markov chain it follows that the dynamics of $\Sigma_{E_X}$ do not depend on $n$. 
    It follows that $t_{\mathrm{mix}}^{\Sigma_{E_X}}$ is a constant. 
    Correspondingly, by \cref{eq: PsTmix} it holds that $1/\gamma_{\rm{ps}}^{E_X} = O(1)$. 
    Hence, since $\ell = \Theta(n^2)$, \Cref{prop: PaulinPseudo} yields that
    \begin{align}
        \mathbb{P}(\lvert \hat{N}_{X,e_n} -  \ell \pi(k_1)&(\#\mathcal{V}_{k_1})^{-1} (\# \mathcal{V}_{k_2})^{-1} p_{k_1,k_2}  \rvert > r)\nonumber \\ 
        & \leq 2\exp\left(-\frac{r^2 \gamma_{\rm{ps}}^{E_X}}{8(\ell  + 1/\gamma_{\rm{ps}}^{E_X})\operatorname{Var}_{\Pi_{E_X}}[f] + 20r}\right)\\ 
        & \leq \cte{r2} \exp(-\cte{r}r)
    \end{align}
    for some $\newconstant{r}, \newconstant{r2}\in \mathbb{R}_{>0}$ which do not depend on $n$. 
    The constant $\cte{r}\in \mathbb{R}_{> 0}$ is chosen so that $\cte{r} \leq \gamma_{\rm{ps}}^{E_X}/(8 (\ell + 1/\gamma_{\rm{ps}}^{E_X})\operatorname{Var}_{\Pi_{E_X}}[f] + 20 )$ for all $n$ and the constant $\cte{r2}\geq 2$ is chosen so that $\cte{r2} \exp(-\cte{r}r) \geq 1$ whenever $r\leq 1$. 
    Note that $\ell \pi(k_1)(\#\mathcal{V}_{k_1})^{-1} (\# \mathcal{V}_{k_2})^{-1} p_{k_1,k_2}$ converges to a finite constant as $n$ tends to infinity. 
    Hence, for some sufficiently large constant $\newconstant{sube}\in \mathbb{R}_{>0}$
    \begin{align}
        \mathbb{P}(\hat{N}_{X,e_n} > r) \leq \cte{sube} \exp(-\cte{r}r).
    \end{align}
    In particular, 
    \begin{align}
        \mathbb{E}[\hat{N}_{X,e_n}^{m+1}] &= \sum_{r=1}^\infty r^{m+1}\mathbb{P}(\hat{N}_{X,e_n} = r)\\ 
        &\leq \cte{sube}\sum_{r=1}^\infty r^{m+1}\exp(-\cte{r}(r-1))\label{eq: unint} 
    \end{align}    
    The right-hand side of \cref{eq: unint} is finite and independent of $n$.  
    This shows that $\hat{N}_{X,e_n}^{m+1}$ is uniformly integrable and concludes the proof. 
\end{proof}
\begin{corollary}
    \label{cor: SingleHighMoment}
    Assume that $X$ starts in equilibrium. Then, for all $m\in \mathbb{Z}_{\geq 0}$ it holds that as $n$ tends to infinity $\max_{e\in\vec{E}_n}\vert \mathbb{E}[\hat{N}_{X,e}^m] \vert = O_{m}(1)$.
\end{corollary}
\begin{proof}
    Observe that $\mathbb{E}[\hat{N}_{X,e}^m]$ can only take $K^2$ possible values depending on the clusters of the endpoints of $e$. 
    Hence, the result follows from \Cref{lem: GeneralMomentP}. 
\end{proof}
\subsubsection{Proof of Proposition {\ref{prop: ApproxUncorM}} on $H(M_X)$ being approximately uncorrelated}
\label{sec: propApproxUncorM}
Recall \Cref{def: ApproxUncorr} of an approximately uncorrelated random matrix with a variance profile. 
The definition consists of two statements: \cref{eq: defAUGenMom} and \cref{eq: defAUVar}.  
These two parts are established in \Cref{prop: MomentsPower1Sharper} and \Cref{prop: HigherMoments} respectively. 
\begin{proposition}
    \label{prop: HigherMoments}
    Assume that $X$ starts in equilibrium and let $R\in \mathbb{Z}_{\geq 1}$. 
    Then, for all $m\in \mathbb{Z}_{\geq 0}^R$ it holds that as $n$ tends to infinity    
    \begin{align}
        \max_{e_1,\ldots,e_R \in \vec{E}_n} \left\lvert \mathbb{E}\left[M_{X,e_{1}}^{m_1} \cdots M_{X,e_{R}}^{m_R}\right] - \mathbb{E}\left[M_{X,e_{1}}^{m_1}\right]\cdots\mathbb{E}\left[M_{X,e_{R}}^{m_R}\right] \right\rvert =   o_{m,R}(1).
    \end{align}
\end{proposition}
\begin{proof}
    We proceed by induction on $R$. The base-case $R=1$ is trivial. Now assume that $R>1$ and that the claim is known to hold for any product with strictly less than $R$ edges $e_{j}$.

    The proof is subdivided in three main steps. 
    First, we construct a pair of Markov chains $X,Y$ which are equal most of the time but are sometimes allowed to diverge. 
    The goal of this process is to ensure two properties: (i) It holds that $M_{X,e_j}\approx M_{Y,e_j}$ for every $j \in \{ 2,\ldots,R\}$. (ii) It holds that $X$ is independent of $M_{Y,e_1}$.
    The approximation from the first item may be used to show that  
    \begin{align}
        \mathbb{E}[M_{Y,e_1}^{m_1} M_{Y,e_2}^{m_2}\cdots M_{Y,e_R}^{m_R}] \approx \mathbb{E}[M_{Y,e_1}^{m_1} M_{X,e_2}^{m_2}\cdots M_{X,e_R}^{m_R}].
        \label{eq: sk1}    
    \end{align}
    The independence from the second item allows us to factorize 
    \begin{align}
        \mathbb{E}[M_{Y,e_1}^{m_1} M_{X,e_2}^{m_2}\cdots M_{X,e_R}^{m_R}] = \mathbb{E}[M_{Y,e_1}^{m_1}]\mathbb{E}[M_{X,e_2}^{m_2}\cdots M_{X,e_R}^{m_R}].
        \label{eq: sk2}    
    \end{align}
    Combined, \cref{eq: sk1} and \cref{eq: sk2} yield that $\mathbb{E}[M_{Y,e_1}^{m_1} \cdots M_{Y,e_R}^{m_R}]$ equals $\mathbb{E}[M_{Y,e_1}^{m_1}]\mathbb{E}[M_{X,e_2}^{m_2}\allowbreak \cdots M_{X,e_R}^{m_R}]$ up to some error term. 
    The induction hypothesis then yields the desired result provided we show that the error term is small. 
    A bound on the size of the error terms is established in the third step.\\

    Some preliminary reductions are applicable.
    Firstly, precisely as in Part \ref{prt: ReductionK5} of the proof of \Cref{prop: MomentSketch}, it can be assumed that $K\geq 2R + 1$ by splitting a cluster into pieces of asymptotically equal proportions.
    Further, the value of $\mathbb{E}[M_{X,e_{1}}^{m_1} \cdots M_{X,e_{R}}^{m_R}] - \mathbb{E}[M_{X,e_{1}}^{m_1}]\cdots\mathbb{E}[M_{X,e_{R}}^{m_R}]$ can only depend on the isomorphism type of the labeled directed subgraph $G$ induced by $\{e_1,\ldots,e_R\}$ with vertex-labels given by the clusters and edge-labels by the index $j$ of $e_j$. 
    As $n$ tends to infinity the number of possible isomorphism types for $G$ remains bounded. 
    Hence, we can fix some isomorphism type for $G$. 
    Informally, this allows us to pretend that the edges $e_1,\ldots, e_R\in \vec{E}_n$ stay fixed as $n$ tends to infinity.
    Due to the induction hypothesis it may also be assumed that the $e_j$ are distinct. 
    \proofpart{Construction of Markov chains $(X,Y)$}
 
    \noindent
    Recall that it was ensured that $K\geq 2R+ 1$.
    Hence, there exists some $k \in \{1,\ldots,K \}$ such that $\mathcal{V}_k$ does not contain any endpoint of $e_j$ for $j=1,\ldots,R$.

    \begin{figure}[tb]
        \centering
        \includegraphics[width = 0.7\textwidth]{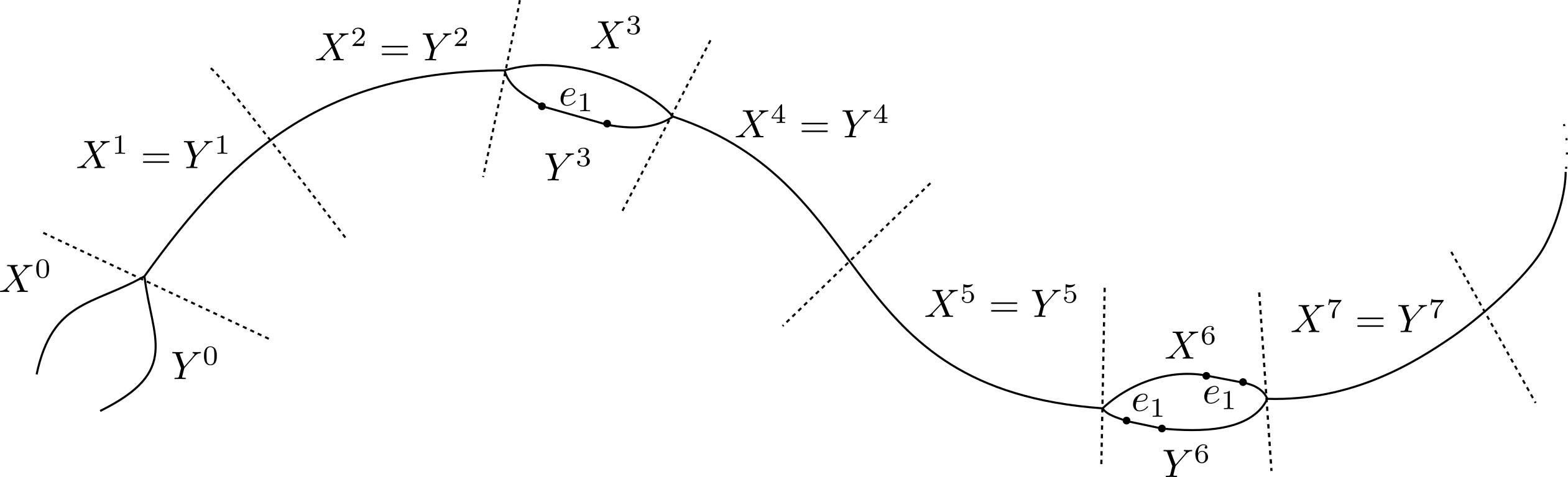}
        \caption{Visualization of the merging process in the construction of $X$ and $Y$ during the proof of \Cref{prop: HigherMoments}.  
        The fragments $X^i$ and $Y^i$ are allowed to diverge whenever either one uses $e_1$.  
        This ensures that all information about $Y$ using $e_1$ is erased from $X$.
        Otherwise, the fragments are merged by taking $X^i = Y^i$. 
        This ensures that the two chains are often equal because using $e_1$ is a rare event. 
        The cluster structure is exploited to ensure that the endpoints of the short fragments $X^i$ and $Y^i$ can be glued after diverging. 
        } 
        \label{fig: Squigle}
    \end{figure}
    
    Use the following procedure to construct a triple of sample paths $(\widetilde{X},\widetilde{Y},\widetilde{Z})$ with random length $\geq \ell+1$. 
    These will be trimmed to paths $X,Y$ and $Z$ of length exactly equal to $\ell+1$ afterwards. 
    See \Cref{fig: Squigle} for a visualization of the construction. 
    \begin{enumerate}[label = \rm{(\roman*)}]
        \item
        Sample two independent infinite paths $(\widetilde{X}_t^{0})_{t=0}^{\infty}$ and $(\widetilde{Z}_t^{0})_{t=0}^\infty$ from the block Markov chain starting in equilibrium. 
        Due to the assumption that the Markov chain associated with $p$ is irreducible and acyclic it holds that $T_0:= \inf\{t\geq 1: \widetilde{X}^{0}_t\in \mathcal{V}_k ,\widetilde{Z}^{0}_t \in \mathcal{V}_k\}$ is finite with probability one.\\ 

        \noindent
        Set $(\widetilde{X}_t)_{t=0}^{T_0-1} := (\widetilde{X}^{0}_{t})_{t=0}^{T_0-1}$ and  $(\widetilde{Y}_t)_{t=0}^{T_0-1} = (\widetilde{Z}_t)_{t=0}^{T_0-1} := (\widetilde{Z}_t^0)_{t=0}^{T_0-1}$. 
        \item For $i=1,\ldots,\ell$ extend $\widetilde{X},\widetilde{Y}$ and $\widetilde{Z}$ by using the following procedure.
        \begin{enumerate}
            \item Sample two independent infinite paths $(\widetilde{X}^{i}_t)_{t=0}^{\infty}$ and $(\widetilde{Z}^{i}_t)_{t=0}^{\infty}$ from the block Markov chain with $\widetilde{X}^{i}_0 = \widetilde{Z}^{i}_0$ chosen uniformly at random in $\mathcal{V}_k$. 
            \item Let $T_i:= \inf\{t\in \mathbb{Z}_{\geq 1}: \widetilde{X}^{i}_{t}\in \mathcal{V}_k, \widetilde{Z}^{i}_{t} \in \mathcal{V}_k\}$. 
            Due to the assumption that the Markov chain associated with $p$ is irreducible and acyclic it will hold that $T_i$ is finite with probability one. 
            \item \label{item: iii} If $e_1$ was traversed by $(\widetilde{X}^{i}_t)_{t=0}^{T_i-1}$ or $(\widetilde{Z}^{i}_t)_{t=0}^{T_i-1}$ then let $(\widetilde{Y}^{i}_t)_{t=0}^{T_i-1} := (\widetilde{Z}^{i}_t)_{t=0}^{T_i-1}$. 
            Else, let $(\widetilde{Y}^{i}_t)_{t=0}^{T_i-1} := (\widetilde{X}^{i}_t)_{t=0}^{T_i-1}$.
            \item \label{item: iv} Append $(\widetilde{X}^{i}_t)_{t=0}^{T_i-1}$, $(\widetilde{Y}^{i}_{t})_{t=0}^{T_i-1}$ and $(\widetilde{Z}^i_t)_{t=0}^{T_{i}-1}$ to the previously constructed parts of $\widetilde{X},\widetilde{Y}$ and $\widetilde{Z}$ respectively. This is to say that we define $\widetilde{X}_{t+\sum_{j =0}^{i-1} T_j} := \widetilde{X}^{i}_t$, $\widetilde{Y}_{t+\sum_{j =0}^{i-1} T_j} := \widetilde{Y}^{i}_t$ and $\widetilde{Z}_{t+\sum_{j =0}^{i-1} T_j} := \widetilde{Z}^{i}_t$ for $t=0,\ldots,T_i-1$. 
        \end{enumerate}
    \end{enumerate}
    
    The sampled paths $X^i := (\widetilde{X}^{i}_t)_{t=0}^{T_i-1}$, $Y^i := (\widetilde{Y}^{i}_t)_{t=0}^{T_i-1}$ and $Z^i := (\widetilde{Z}^{i}_t)_{t=0}^{T_i-1}$ used in the construction of $\widetilde{X},\widetilde{Y}$ and $\widetilde{Z}$ will be called fragments. 
    For any $i = 1,\ldots,\ell$ the $i$th fragments are said to be merged if $(\widetilde{X}_t^i)_{t=0}^{T_i-1}$ and $(\widetilde{Z}_t^i)_{t=0}^{T_i-1}$ did not traverse $e_1$. 
    Denote $X := (\widetilde{X}_t)_{t=0}^\ell, Y := (\widetilde{Y}_t)_{t=0}^\ell$ and $Z :=(\widetilde{Z}_t)_{t=0}^\ell$  for the truncations to length $\ell+1$. 
    Note that these are block Markov chains from the specified model. 

    We further claim that $X$ is independent of $M_{Y,e_1}$.
    Let us remark that this is due to the special role of $e_1$ in the above construction; it will typically not be true that $X$ is independent of $M_{Y,e'}$ for any other edge $e' \neq e_1$. 
    Indeed, let $\widetilde{W}$ be the path found from $\widetilde{Z}$ by replacing $\widetilde{Z}_{\sum_{j=0}^i T_i}$ with a uniformly random node in $\mathcal{V}_k$ for every $i=0,\ldots,\ell-1$.
    Then $X$ is independent of $W:= (\widetilde{W}_t)_{t=0}^\ell$.
    In particular $X$ will be independent of $M_{W,e_1}$ since this is a function of $W$. 
    Now note that $M_{Y,e_1}$ is equal to $M_{W,e_1}$ by definition of $Y$ and the fact that $\mathcal{V}_k$ was chosen not to contain any of the endpoints of $e_1$. 
    This establishes that $X$ is independent of $M_{Y,e_1}$ as desired. 
    \proofpart{Approximate factorization of $\mathbb{E}[M_{Y,e_1}^{m_1}\cdots M_{Y,e_R}^{m_R}]$}

    \noindent
    For any $j=2,\ldots,R$ define $\Delta_j$ to be the difference between the number of times $e_j$ was used in $Y$ and $X$, that is, 
    \begin{align}
        \Delta_j := \hat{N}_{Y,e_j} - \hat{N}_{X,e_j}.
    \end{align} 
    Observe that $M_{Y,e_j} = M_{X,e_j} + \Delta_j$ for every $j\geq 2$ since $\mathbb{E}[\hat{N}_{Y,e_j}] = \mathbb{E}[\hat{N}_{X,e_j}]$; after all $X$ and $Y$ have the same distribution. 
    Substitute the binomial expansion for every factor $M_{Y,e_j}^{m_{j}} = (M_{X,e_j} + \Delta_j)^{m_j}$ in $M_{Y,e_2}^{m_2}\cdots M_{Y,e_R}^{m_R}$ and use the fact that $M_{Y,e_1}$ is independent of the $M_{X,e_j}$ with $j\geq 2$ to find that  
    \begin{align}
        \mathbb{E}[M_{Y,e_1}^{m_1} M_{Y,e_2}^{m_2} \cdots M_{Y,e_R}^{m_R}] &- \mathbb{E}[M_{Y,e_1}^{m_1}] \mathbb{E}[M_{X,e_2}^{m_2}\cdots M_{X,e_R}^{m_R}] \\ 
        &= \sum_{0\leq m' \leq m}  c_{m'} \mathbb{E}\bigg[ M_{Y,e_1}^{m_1} \prod_{j=2}^R \Delta_{j}^{m_j'} \prod_{j=2}^R M_{X,e_j}^{m_j - m_j'}\bigg]\nonumber
    \end{align}    
    where the summation runs over vectors of integers of length $R$ and the $c_{m'}$ are absolute constants with $c_{m'}= 0$ whenever $m_j'=0$ for all $j\geq 2$. 
    Note that $M_{X,e_2}^{m_2}\cdots M_{X,e_R}^{m_R}$ is a product with $R-1$ factors.  
    It follows that the induction hypothesis is applicable to $\mathbb{E}[M_{X,e_2}^{m_2}\cdots M_{X,e_R}^{m_R}]$ so that it remains to show that $\mathbb{E}[ M_{Y,e_1}^{m_1} \prod_{j=2}^R \Delta_{j}^{m_j'} \prod_{j=2}^R M_{X,e_j}^{m_j - m_j'}] =o_{m,R}(1)$ whenever $m_j'\geq 1$ for some $j\geq 2$. 

    By the Cauchy--Schwarz inequality it holds that 
    \begin{align}
        \bigg\vert \mathbb{E}\bigg[ M_{Y,e_1}^{m_1}\prod_{j=2}^R \Delta_{j}^{m_j'}  \prod_{j=2}^R M_{X,e_j}^{m_j - m_j'}\bigg] \bigg\vert
        \leq
        \sqrt{
        \mathbb{E}\bigg[\prod_{j=2}^R \Delta_j^{2m_j'}\bigg] \mathbb{E}\bigg[M_{Y,e_1}^{2m_1}\prod_{j=2}^R M_{X,e_j}^{2(m_j - m_j')}\bigg]}.
        \label{eq: CauchySchwarz}
    \end{align}
    Due to the independence of $X$ on $M_{Y,e_1}$ we can write 
    \begin{equation}
        \mathbb{E}\bigg[M_{Y,e_1}^{2m_1}\prod_{j=2}^R M_{X,e_j}^{2(m_j - m_j')}\bigg] = \mathbb{E}\bigg[M_{Y,e_1}^{2m_1}\bigg] \mathbb{E}\bigg[\prod_{j=2}^R M_{X,e_j}^{2(m_j - m_j')}\bigg]     
    \end{equation}
    which is $O_{m}(1)$ by \Cref{cor: SingleHighMoment} and the induction hypothesis.
    It remains to show that $\mathbb{E}[\prod_{j=2}^R \Delta_j^{2m_j'}]=o_{m'}(1)$. 

    Define $\mathcal{M}_{i}$ to be the event where $X^i$ and $Y^i$ were merged and denote $\neg \mathcal{M}_i$ for the complement of this event. 
    For any fixed $j\geq 2$ let $\Delta_{X,j} := \hat{N}_{X^0,e_j} + \sum_{i=1}^\ell \1_{\neg \mathcal{M}_i}\hat{N}_{X^i,e_j}$ denote the number of times when $\widetilde{X}$ used edge $e_j$ in a fragment where $X$ and $Y$ were not merged and define $\Delta_{Y,j}$ similarly. 
    Then
    \begin{align}
        \lvert \Delta_j \rvert &= \lvert \hat{N}_{Y,e_j} - \hat{N}_{X,e_j} \rvert
        \leq \Delta_{Y,j} + \Delta_{X,j}.    
    \end{align}
    Substitute the definitions of $\Delta_{X,j}$ and $\Delta_{Y,j}$ in the resulting bound $\prod_{j=2}^R\Delta_{j}^{2m_j'} \leq \prod_{j=2}^R(\Delta_{Y,j} + \Delta_{X,j})^{2m_j'}$ and expand the monomial expression.
    Then, the Cauchy--Schwarz inequality reduces us to the statement that $\mathbb{E}[\hat{N}_{X^0, e_j}^q] = o_{q}(1)$ and $\mathbb{E}[( \sum_{i=1}^\ell\allowbreak \1_{\neg\mathcal{M}_i}\hat{N}_{X^i,e_j})^q] = o_{q}(1)$ for any $q\in \mathbb{Z}_{\geq 1}$ where it was used that $Y$ has the same distribution as $X$. 
    These claims are established in the next part of the proof. In fact, we will establish the stronger claims that $\mathbb{E}[\hat{N}_{X^i,e_j}^{q}]=O_{q}(n^{-2})$ for any $(i,j)\in \{ 0,1\}\times \{1,\ldots,R\}$ and $\mathbb{E}[( \sum_{i=1}^\ell \1_{\neg \mathcal{M}_i}\hat{N}_{X^i,e_j})^q] = O_q(n^{-1})$ for any $i\geq 1$.

    \proofpart{Bounding $\mathbb{E}[N^{q}_{X^i,e_j}]$ and $\mathbb{E}[( \sum_{i=1}^\ell \1_{\neg \mathcal{M}_i}\hat{N}_{X^i,e_j})^q]$}
     \subproofpart{$\mathbb{E}[N^{q}_{X^i,e_j}] = O_{q}(n^{-2})$\label{prt: Nq}}

     \noindent
     Fix some $(i,j) \in \{0,1\}\times \{1,\ldots,R\}$ and $q\in \mathbb{Z}_{\geq 1}$.
     By the law of total expectation
     \begin{align}
        \mathbb{E}[\hat{N}_{X^i,e_j}^{q}] = \sum_{t =1}^\infty \mathbb{P}(T_i = t) \mathbb{E}[\hat{N}_{X^i,e_j}^{q}\mid T_i =t] .
        \label{eq: NiTotalExpectationA}
     \end{align}
     Use $\hat{N}_{X^i, e_j} = \sum_{t=1}^{T_i} \1_{E_{X^i,t} = e_j}$ to derive the following bound
     \begin{align}
        \mathbb{E}[\hat{N}_{X^i,e_j}^{q}\mid T_i = t ] &=  \sum_{t_1,\ldots,t_{q} = 1}^t \mathbb{P}(E_{X^i,t_1}= e_j,\ldots, E_{X^i,t_{q}}=e_j\mid T_i = t )\\ 
        &\leq  \sum_{t_1,\ldots,t_{q}=1}^{t}  \mathbb{P}( E_{X^i,t_1} = e_j \mid T_i = t).
        \label{eq: NiTotalExpectation}
     \end{align}
     For any edge $e$ whose starting and ending point are in the same clusters as the starting and ending point of $e_j$ it holds that $\mathbb{P}(E_{X^i,t_1} = e \mid T_i = t) =\mathbb{P}(E_{X^i,t_1} = e_j \mid T_i = t)$. 
     There are at least $\alpha_{min}^2 n^2$ such edges so it follows that $\mathbb{P}( E_{X^i,t_1} = e_j \mid T_i = t) \leq \alpha_{min}^{-2}n^{-2}$. 

     Equation \cref{eq: NiTotalExpectation} now yields that
     \begin{align}
        \mathbb{E}[\hat{N}_{X^i,e_j}^{q}\mid T_i =t] &\leq \alpha^{-2}_{min}n^{-2}t^{q}. \label{eq: TauBoundConditionalExNi}
     \end{align}
     By \cref{eq: NiTotalExpectationA} and \cref{eq: TauBoundConditionalExNi} it follows that 
    \begin{align}
        \mathbb{E}[\hat{N}_{X^i,e_j}^{q}] & \leq \alpha_{min}^{-2} n^{-2} \mathbb{E}[T_i^{q}].
    \end{align}
    Here, $\mathbb{E}[T_i^{q}]$ is a finite constant which does not depend on $n$.
    Indeed, consider the product chain $\Sigma_{(\widetilde{X}^i, \widetilde{Z}^i)} := (\sigma(\widetilde{X}_t^i), \sigma(\widetilde{Z}_t^i))_{t=0}^{\infty}$ on the space of clusters $\{1,\ldots,K \}\times \{1,\ldots,K\}$.
    Then $T_i$ is the first strictly positive time $\Sigma_{(\widetilde{X}^i,\widetilde{Z}^i)}$ is in $(k,k)$.
    In particular, $\mathbb{E}[T_i^{q}]$ is independent of $n$. 
    The fact that $\mathbb{E}[T_i^{q}]$ is finite is immediate from the fact that $\mathbb{P}(T_i >t)$ shows exponential decay in $t$ since the Markov chain associated with $p$ is assumed to be irreducible and acyclic. 
    This concludes the proof of the statement that $\mathbb{E}[N^{q}_{X^i,e_j}] = O_{q}(n^{-2})$.
    \subproofpart{ $\mathbb{E}[( \sum_{i=1}^\ell \1_{\neg \mathcal{M}_i}\hat{N}_{X^i,e_j})^q] = O_{q}(n^{-1})$}

    \noindent
    Observe that for any $q\in \mathbb{Z}_{\geq 1}$ 
    \begin{align}
        \mathbb{E}\bigg[\bigg(\sum_{i=1}^\ell \1_{\neg \mathcal{M}_i}\hat{N}_{X^i,e_j}\bigg)^q\bigg] &= \sum_{i_1,\ldots,i_q = 1}^\ell \mathbb{E}\bigg[ \prod_{l=1}^{q} \1_{\neg \mathcal{M}_{i_l}} \hat{N}_{X^{i_l},e_j}\bigg]\\ 
        &= \sum_{d = 1}^{q} \sum_{\# \{i_1,\ldots,i_q \} = d} \mathbb{E}\bigg[ \prod_{l=1}^{q} \1_{\neg \mathcal{M}_{i_l}} \hat{N}_{X^{i_l},e_j}\bigg] 
        \label{eq: SumReduction}
    \end{align}
    where the second sum in \cref{eq: SumReduction} runs over all values of $i_1,\ldots,i_q \in \{1,\ldots,\ell\}$ with $\#\{i_1,\ldots,i_q \} = d$.  

    Note that the random variables $\1_{\neg\mathcal{M}_{i}}\hat{N}_{X^{i},e_j}$ with $i\geq 1$ are independent and identically distributed due to the fact that $\mathcal{V}_k$ was chosen not to contain any endpoint of $e_1$ or $e_j$. 
    Correspondingly, a term with $d$ distinct $i_{l}$ values in the right-hand side of \cref{eq: SumReduction} may be factorized as $\prod_{l=1}^d \mathbb{E}[ \1_{\neg \mathcal{M}_1}\hat{N}_{X^1,e_j}^{q_l}]$ for some $q_1,\ldots,q_l\in \{1,\ldots,q\}$.
    Since $\hat{N}_{X^1,e_j}\geq 1$ it holds that $\mathbb{E}[ \1_{\neg \mathcal{M}_1}\hat{N}_{X^1,e_j}^{q_l}] \leq \mathbb{E}[ \1_{\neg \mathcal{M}_1}\hat{N}_{X^1,e_j}^{q}]$. 
    Consequently, we may bound \cref{eq: SumReduction} as 
    \begin{align}
        \mathbb{E}\bigg[\bigg(\sum_{i=1}^\ell \1_{\neg \mathcal{M}_i}\hat{N}_{X^i,e_j}\bigg)^q\bigg] \leq \sum_{d = 1}^{q} c_{q,d} \ell^d  \mathbb{E}\bigg[\1_{\neg \mathcal{M}_{1}} \hat{N}_{X^{1},e_j}^q\bigg]^d  \label{eq: SumRed2}
    \end{align}
    for some absolute constants $c_{q,d}\in\mathbb{R}_{>0}$.
    Since $\ell = \Theta(n^2)$ it now suffices to show that $\mathbb{E}[\1_{\neg \mathcal{M}_{1}} \hat{N}_{X^{1},e_j}^q] = O_q(n^{-3})$. 
    This will be shown by an argument resembling Part \ref{prt: Nq}.

    By the law of total expectation 
    \begin{align}
        \mathbb{E}&[\1_{\neg\mathcal{M}_1}\hat{N}_{X^1,e_j}^q]\label{eq: TotalExpec}\\ &= \sum_{t = 1}^\infty \mathbb{P}(T_1 = t) \mathbb{P}(\neg \mathcal{M}_1\mid T_1 = t)\mathbb{E}[\hat{N}_{X^1,e_j}^q \mid T_1 = t,\neg \mathcal{M}_1].
        \nonumber
    \end{align}
    The union bound implies that $\mathbb{P}(\mathcal{A}\mid \cup_i\mathcal{B}_i) \leq \sum_i \mathbb{P}(\mathcal{A}\mid \mathcal{B}_i)$ for any countable collection of events $\mathcal{A},\mathcal{B}_i$ with $\mathbb{P}(\mathcal{B}_i)\neq 0$. 
    Observe that 
    \begin{align}
        \neg\mathcal{M}_1 = \cup_{t_0\in \mathbb{Z}_{\geq 1}}\{E_{X^1, t_0}  = e_1 \text{ and } t_0 \leq T_1\}\cup \{E_{Z^1,t_0} = e_1 \text{ and } t_0 \leq T_1\}.    
    \end{align}
    Therefore, since $\hat{N}_{X^1,e_j}^q = \sum_{t_1 = 1}^{T_1} \1_{E_{X^1, t_1} = e_j}$ 
    \begin{align}
        \mathbb{E}&[\hat{N}_{X^1,e_j}^q \mid T_1= t, \neg \mathcal{M}_1]\nonumber \\ 
        &= \sum_{t_1, \ldots,t_q=1}^t \mathbb{P}(E_{X^1,t_1} = e_j, \ldots, E_{X^1, t_q} = e_j\mid T_1 = t, \neg \mathcal{M}_1)\\ 
         &\leq \sum_{t_1, \ldots,t_q=1}^t \big(\sum_{t_0}\mathbb{P}(E_{X^1,t_1} = e_j\mid T_1 = t, E_{X^1,t_0}= e_1) \label{eq: UnionBoundN1}\\ 
         &\hphantom{\leq \sum_{t_1, \ldots,t_q=1}^t} \qquad + \sum_{t_0}\mathbb{P}(E_{X^1,t_1} = e_j\mid T_1 = t, E_{Z^1,t_0}= e_1)\big)\label{eq: sumZ}
    \end{align}
    where the second sum in \cref{eq: UnionBoundN1} runs over all $t_0\in \{1,\ldots,t \}$ with $\mathbb{P}(T_1 = t,E_{X^1,t_0} = e_1 ) \neq 0$ and the sum in \cref{eq: sumZ} runs over all $t_0 \in \{1,\ldots,t \}$ with $\mathbb{P}(T_1 = t,E_{Z^1,t_0} = e_1)\neq 0$. 
    First consider the terms with $E_{Z^1,t_0} = e_1$. 
    Note that there are at least $\alpha_{min}^{2}n^2$ edges $e$ whose starting point and ending point have the same clusters as the starting point and ending point of $e_j$ respectively. 
    Conditional on $T_1 = t$ and $E_{Z^1,t_0}= e_1$ any such edge $e$ is equally likely to be traversed at time $t_1$ by $X^1$.
    It follows that $\mathbb{P}(E_{X^1,t_1} = e_j\mid T_1 = t, E_{Z^1,t_0}= e_1) \leq \alpha_{min}^{-2}n^{-2}$.  

    Let us now consider the terms with $E_{X^1,t_0}= e_1$.  
    When $\vert t_0 -t_1 \vert >1$ the foregoing argument applies word--for--word and yields that $\mathbb{P}(E_{X^1,t_1} = e_j\mid T_1 = t, E_{X^1,t_0}= e_1) \leq \alpha_{min}^{-2}n^{-2}$. 
    The cases $t_0 = t_1-1$ and $t_0 = t_1 +1$ require a modification. 

    When $t_0 = t_1 - 1$ note that there at least $\alpha_{min}n$ edges $e$ whose ending point is in the same cluster as the ending point of $e_j$ and whose starting point is equal to the starting point of $e_j$. 
    Conditional on $T_1 = t$ and $E_{X^1,t_0}= e_1$ any such edge $e$ is equally likely to be traversed at time $t_1$. 
    Hence, $\mathbb{P}(E_{X^1,t_1} = e_j \mid T_1 = t, E_{X^1,t_0} = e_1) \leq  \alpha_{min}^{-1} n^{-1}$ for all $t_0 = t_1 - 1$.
    
    Similarly, when $t_0 =t_1 + 1$ note that there are at least $\alpha_{min}n$ edges $e$ whose starting point is in the same cluster as the starting point of $e_j$ and whose ending point is equal to the ending point of $e_j$.
    Conditional on $T_1 = t$ and $E_{X^1,t_0}= e_1$ any such edge $e$ is equally likely to be traversed at time $t_1$. Hence,  $\mathbb{P}(E_{X^1,t_1} = e_j \mid T_1 = t, E_{X^1,t_0} = e_1) \leq  \alpha_{min}^{-1} n^{-1}$ for all $t_0 = t_1 + 1$.

    Using these bounds in \cref{eq: UnionBoundN1} yields that 
    \begin{align}
        \mathbb{E}[\hat{N}_{X^1,e_j}^q \mid T_1= t, \neg \mathcal{M}_1 ]&\leq 2\alpha_{min}^{-2} n^{-1}t^{q+1}. \label{eq: cn2Ni}
    \end{align}
    Observe that $\mathbb{P}(\neg \mathcal{M}_1\mid T_1 = t) \leq \mathbb{E}[\hat{N}_{X^1,e_1} + \hat{N}_{Z^1,e_1}\mid T_1 = t]$.
    Note that $X^1$ and $Z^1$ follow the same distribution and apply \cref{eq: TauBoundConditionalExNi} with $q=1$ to deduce that 
    \begin{align}
        \mathbb{P}(\neg \mathcal{M}_1\mid T_1 = t)  \leq  2\mathbb{E}[\hat{N}_{X^1,e_1} \mid T_1 = t] \leq 2\alpha_{min}^{-2}n^{-2}t\label{eq: 2T1}
    \end{align} 
    Hence, \cref{eq: TotalExpec}, \cref{eq: cn2Ni} and \cref{eq: 2T1} yield that
    \begin{align}
        \mathbb{E}[\1_{\neg\mathcal{M}_1}\hat{N}_{X^1,e_j}^q] &\leq 4 \alpha_{min}^{-4} n^{-3}\mathbb{E}[T_1^{q+2}]  
    \end{align} 
    where, as in Part \ref{prt: Nq}, it holds that $\mathbb{E}[T_1^{q+2}]$ is a finite constant which does not depend on $n$. 
    Hence, $\mathbb{E}[\1_{\neg\mathcal{M}_1}\hat{N}_{X^1,e_j}^q] = \Theta_q(n^{-3})$ which may be combined with \cref{eq: SumRed2} and the fact that $\ell = \Theta(n^{-2})$ to deduce the desired result:
    \begin{align}
        \mathbb{E}\bigg[\bigg(\sum_{i=1}^\ell \1_{\neg \mathcal{M}_i}\hat{N}_{X^i,e_j}\bigg)^q\bigg] &= O_q(n^{-1}).
    \end{align}
    This concludes the proof. 
\end{proof} 
\begin{proposition}
    \label{prop: MomentsPower1Sharper}
    Assume that $X$ starts in equilibrium and let $0\leq r \leq R$ be positive integers. Then, for any $m \in \mathbb{Z}_{\geq 0}^{R}$ with $m_{i}=1$ for $i=1,\ldots,r$, it holds that as $n$ tends to infinity  
    \begin{align}
        \max_{\forall i\neq j:e_i\neq e_j}\left\lvert \mathbb{E}\left[ M_{X,e_{1}}^{m_1} \cdots M_{X,e_{R}}^{m_{R}}\right] \right\rvert = O_{m,R}(n^{-\lceil r/2 \rceil })
    \end{align}
    where the maximum runs over all sequences of distinct edges $e_1,\ldots,e_R\in \vec{E}_n$.   
\end{proposition}
We will prove a better bound than \Cref{prop: MomentsPower1Sharper}.
Recall from \Cref{sec: Dependence} that \cite{fleermann2021almost} improved the convergence in probability of \cite{hochstattler2016semicircle} to convergence almost surely provided additional assumptions.
It turns out that the most immediate generalization of their results would be too restrictive to include block Markov chains because the correlations in \Cref{prop: MomentsPower1Sharper} do not decay sufficiently quickly.
  
The correlation between edges $e_i$ and $e_j$ is maximal when the ending point of $e_i$ is equal to the starting point of $e_j$. 
Correspondingly, \Cref{prop: MomentsPower1Sharper} can be improved when there are many edges whose ending point is not the starting point of some other edge. 
This improvement could be a point of departure to strengthen \Cref{thm: SingValN} and \Cref{thm: SingValP} to convergence almost surely.

For any positive integers $0\leq r' \leq R$ let $\vec{E}_{n,r'}^R$ denote the collection of sequences of distinct edges $e_1,\ldots, e_R \in \vec{E}_{n}$ such that for every $i\in \{1,\ldots,r' \}$ it holds that the ending point of $e_i$ is not the starting point of any edge $e_j$ with $j\in \{1,\ldots,R \}\setminus \{i \}$ and the starting point of $e_i$ is not the ending point of any edge $e_j$ with $j\in \{1,\ldots,R \}\setminus\{i \}$. 

\begin{figure}[tb]
    \centering
    \includegraphics[width= 0.6 \textwidth]{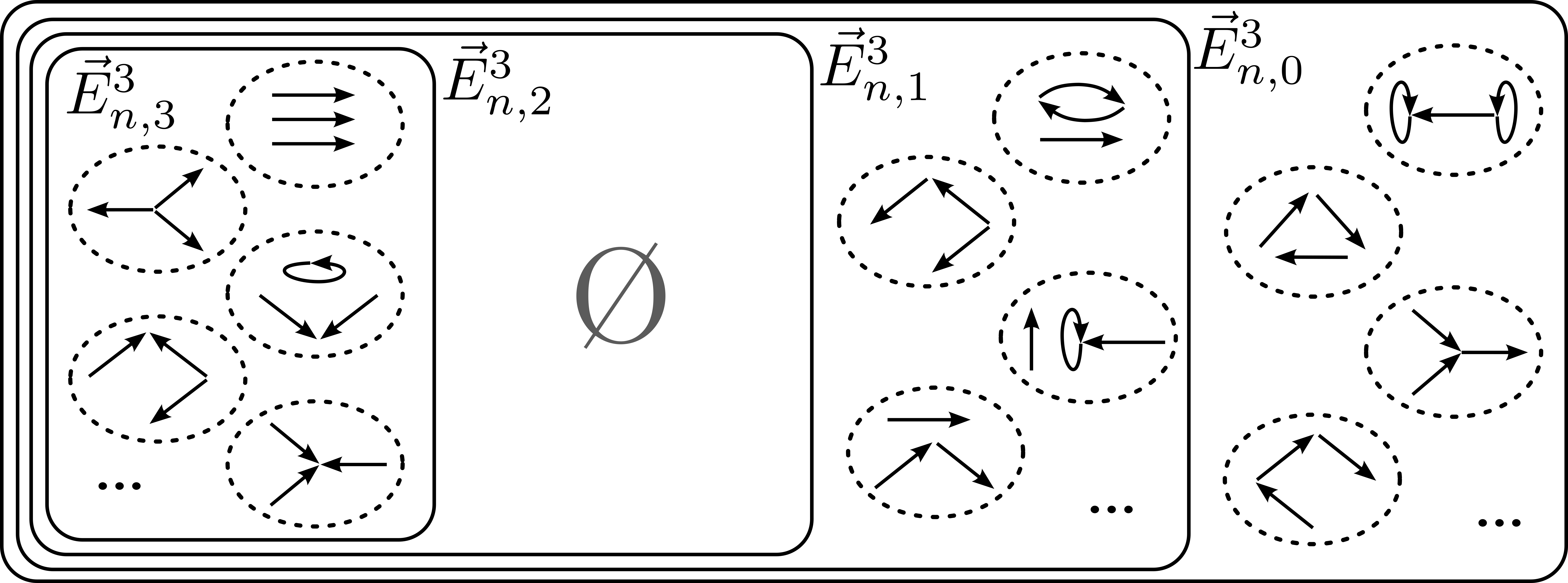}
    \caption{Visualization of some configurations of edges occurring in $\vec{E}_{n,r'}^R$ for varying values of $r'$ when $R = 3$. 
    The visualized configurations are not exhaustive. 
    Observe that $\vec{E}_{n,R}^R = \vec{E}_{n,R-1}^R \subseteq \vec{E}_{n,R-2}^R \subseteq\cdots \subseteq  \vec{E}_{n,0}^R$ as is immediate from the definition.    
    }
    \label{fig: Es}
\end{figure}

See \Cref{fig: Es} for an example with $R=3$.
The following proposition includes \Cref{prop: MomentsPower1Sharper} as the special case with $r' = 0$. 
\begin{proposition}
    \label{prop: PowerSharperSharper}
    Assume that $X$ starts in equilibrium and let $0\leq r' \leq r \leq R$ be positive integers. Then, for any $m \in \mathbb{Z}_{\geq 0}^{R}$ with $m_{i}=1$ for $i=1,\ldots,r$, it holds that as $n$ tends to infinity  
    \begin{align}
        \max_{(e_1,\ldots,e_R) \in \vec{E}_{n,r'}^R}\left\lvert \mathbb{E}\left[ M_{X,e_{1}}^{m_1} \cdots M_{X,e_{R}}^{m_{R}}\right] \right\rvert = O_{m,R}(n^{-\lceil r/2 \rceil - \lceil r'/2 \rceil}).
    \end{align}
\end{proposition}
\begin{proof}
    This proof combines the proof of \Cref{prop: MomentSketch} with an inductive argument as was used in the proof of \Cref{prop: HigherMoments}. 
    The main technical difference is that we can no longer use the Cauchy--Schwarz inequality during the inductive step as in \cref{eq: CauchySchwarz}: the resulting squares would reduce $r$ to zero which weakens the conclusion of the induction hypothesis.  
    Instead, the step analogous to \cref{eq: CauchySchwarz} will employ conditional independence. 
    The price we pay for this argument is that it necessitates a stronger induction hypothesis to account for the added conditioning.\\ 

    The same preliminary reductions as in the proof of \Cref{prop: HigherMoments} are applicable.
    Firstly, precisely as in Part \ref{prt: ReductionK5} of the proof of \Cref{prop: MomentSketch}, it can be assumed that $K\geq 2R + 1$ by splitting a cluster into pieces of asymptotically equal proportions.
    Further, by fixing an isomorphism type for the labeled directed graph $G$ induced by $\{e_1,\ldots,e_R \}$ we may again pretend that the edges $e_1,\ldots, e_R$ stay fixed as $n$ tends to infinity. 
    \proofpart{Set-up of the inductive argument}

    \noindent
    Recall that it is ensured that $K\geq 2R + 1$.
    Hence, there exists some $k \in \{1,\ldots,K \}$ such that $\mathcal{V}_k$ does not contain any endpoint of $e_j$ for $j=1,\ldots,R$.

    For any $d\in\mathbb{Z}_{\geq 0}$, $\ell' \leq \ell$ and $\tau\in \{0,\ldots,\ell'\}^d$ denote $\mathcal{V}_{X,\tau}$ for the event where $X_{\tau_i}\in \mathcal{V}_k$ for every $i=1,\ldots,d$.
    When $d=0$ it is to be understood that $\mathcal{V}_{X,\tau}$ refers to the universal event. 
    In particular, $\mathbb{P}(\mathcal{V}_{X,\tau}) = 1$ in this case.  
    Fix some $d, \ell'$ and $\tau$ with $\mathbb{P}(\mathcal{V}_{X,\tau}) > 0$ and let $Y:= (Y_t)_{t=0}^{\ell'}$ be a sample path from the block Markov chain conditioned on the event $\mathcal{V}_{Y,\tau}$.
    We will show that there exist a constant $\newconstant{O1}\in \mathbb{R}_{\geq 0}$, depending on $d$ but not on $\ell'$ or $\tau$, such that 
    \begin{align}
        \mathbb{E}[M_{Y,e_1}^{m_1}\cdots M_{Y,e_R}^{m_R}] \leq \cte{O1}n^{-\lceil r/2 \rceil - \lceil r'/ 2\rceil}.
        \label{eq: d-induct}
    \end{align} 
    Taking $d = 0$ and $\ell' = \ell$ then recovers the proposition.

    The proof of the claim proceeds by induction on $r$. 
    This is why we require \cref{eq: d-induct} to hold for any $d\in \mathbb{Z}_{\geq 0}$ and $\ell' \leq \ell$ even though the proposition only concerns the case with $d=0$ and $\ell' = \ell$: we will modify $d$ and $\ell'$ when reducing $r$ in the inductive step. 
    The argument for the base case $r=0$ is provided in Part \ref{prt: basecase}. 
    Now let $r>1$ and assume that \cref{eq: d-induct} is known to hold for any smaller value of $r$. 

    Recall that $m_1 = 1$ and $\hat{N}_{Y,e_1} = \sum_{t_1 = 1}^{\ell'} \1_{E_{Y,t_1} = e_1}$. 
    Therefore, 
    \begin{align}
        &\mathbb{E}[ M_{Y,e_{1}}^{m_1} \cdots M_{Y,e_{R}}^{m_{R}}]  = \mathbb{E}[\hat{N}_{Y,e_1}M_{Y,e_{2}}^{m_2} \cdots M_{Y,e_{R}}^{m_{R}} ] - \mathbb{E}[\hat{N}_{Y,e_1}]\mathbb{E}[M_{Y,e_{2}}^{m_2} \cdots M_{Y,e_{R}}^{m_{R}}]\nonumber \\ 
        &=\hspace{-0.05em}\sum_{t_1=1}^{\ell'}\bigl(\mathbb{E}[\1_{E_{Y,t_1}=e_1}M_{Y,e_{2}}^{m_2} \cdots M_{Y,e_{R}}^{m_{R}}] - \mathbb{E}[\1_{E_{Y,t_1}=e_1}]\mathbb{E}[M_{Y,e_{2}}^{m_2} \cdots M_{Y,e_{R}}^{m_{R}}]\bigr) \nonumber \\
        &=\hspace{-0.05em}\sum_{t_1=1}^{\ell'} \mathbb{P}(E_{Y,t_1}= e_1) \bigl(\mathbb{E}[M_{Y,e_{2}}^{m_2} \cdots M_{Y,e_{R}}^{m_{R}} \mid E_{Y,t_1}=e_1] - \mathbb{E}[M_{Y,e_{2}}^{m_2} \cdots M_{Y,e_{R}}^{m_{R}}] \bigr).
        \label{eq: ConditionSum}
    \end{align}
    Note that there are at least $\alpha_{min}^2n^2$ edges $e$ for which the starting and ending point are in the same clusters as the starting and ending point of $e_1$ respectively. 
    These edges $e$ are equally likely to be traversed at time $t_1$ by $Y$. 
    It follows that  
    \begin{align}
        \mathbb{P}(E_{Y,t_1}= e_1) \leq \alpha_{min}^{-2}n^{-2}. \label{eq: PEdgeConditioned}    
    \end{align}
    
    It remains to show that the effect of conditioning on $E_{Y,t_1} = e_1$ in \cref{eq: ConditionSum} has a small effect on $M_{Y,e_{2}}^{m_2} \cdots M_{Y,e_{R}}^{m_{R}}$.
    That is, it has to be shown that $\mathbb{E}[M_{Y,e_{2}}^{m_2} \cdots M_{Y,e_{R}}^{m_{R}}\mid E_{Y,t_1} = e_1] \approx \mathbb{E}[M_{Y,e_{2}}^{m_2} \cdots M_{Y,e_{R}}^{m_{R}}]$ with approximation error uniform over all $t_1\in \{1,\ldots,\ell' \}$ with $\mathbb{P}(E_{Y,t_1}= e_1)>0$.
    Fix such a value of $t_1$.  
    \proofpart{Construction of chains $(Y,Y',Z)$\label{prt: construct}}

    \begin{figure}[tb]
        \centering
        \includegraphics[width= 0.7 \textwidth]{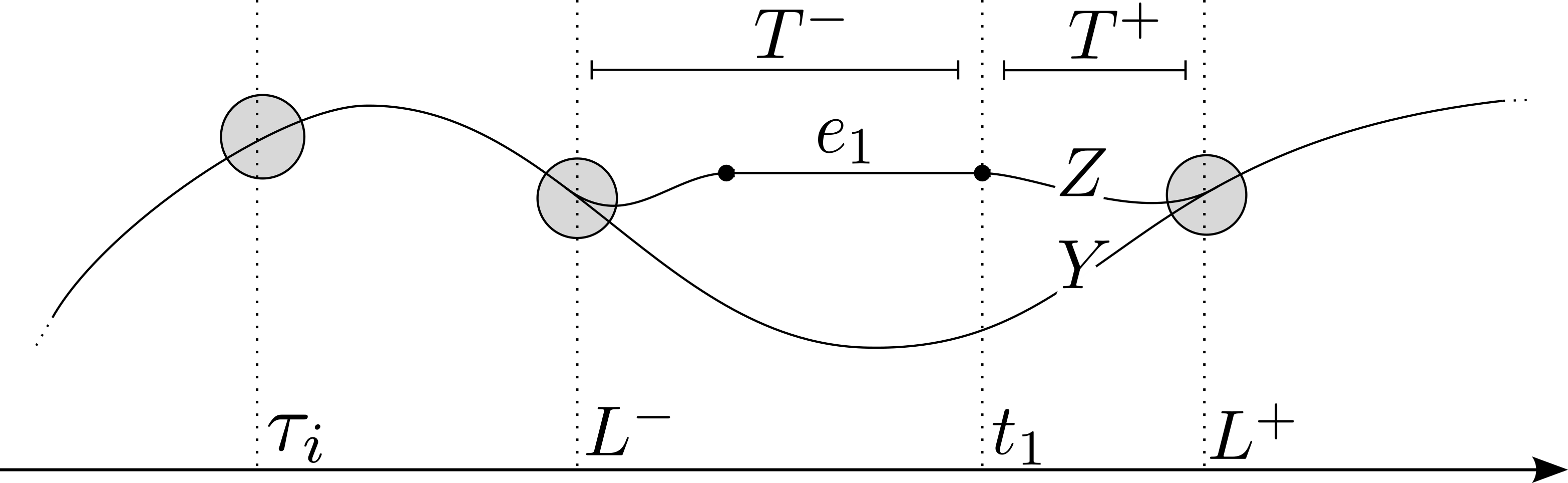}
        \caption{Visualization of the construction of the chains $Y$ and $Z$ in the proof of \Cref{prop: PowerSharperSharper}.
        The chain $Y'$ is found from $Y$ by cutting out the piece between $L^-$ and $L^+$.  
        Both chains are conditioned to be in cluster $\mathcal{V}_k$ at times $\tau_i$ but $Z$ has the additional condition of using edge $e_1$ at time $t_1$. 
        The visualized process of gluing $e_1$ onto $Y$ exploits the cluster structure to ensure that $L^+ - L^-$ is small.  
        }
        \label{fig: YZCond}
    \end{figure}

    \noindent
    Use the following procedure to construct a triple of chains $(Y,Y',Z)$ with $Y,Z$ of length $\ell'+1$ and $Y'$ of a random length at most $\ell'+1$. 
    See also \Cref{fig: YZCond} for a visualization of the construction. 
    \begin{enumerate}[label = \rm{(\roman*)}]
        \item Sample an infinitely long path $(\widetilde{Y}_t)_{t=-\infty}^\infty$ from the block Markov chain conditioned on $\mathcal{V}_{\widetilde{Y},\tau}$. 
        Sample $(\widetilde{W}_{t})_{t=-\infty}^\infty$ as an infinitely long path from the block Markov chain conditioned on $E_{\widetilde{W},t_1} = e_1$ and $\mathcal{V}_{\widetilde{W},\tau}$. 
        Note that it is possible to sample at negative times by use of a time reversal which exists by the assumption that the Markov chain associated with $p$ is acyclic and irreducible. 
        \item Define
        \begin{align}
            T^- &:= t_1 - \sup\{t\in \mathbb{Z}_{< t_1}: \widetilde{Y}_t \in \mathcal{V}_k, \widetilde{W}_t \in \mathcal{V}_k\}\\ 
            T^+ &:= \inf\{t \in \mathbb{Z}_{>t_1}: \widetilde{Y}_t\in \mathcal{V}_k, \widetilde{W}_t \in \mathcal{V}_k \} -t_1\label{eq: T+Original}
        \end{align}
        and note that these values are finite with probability one by the assumption that the Markov chain associated with $p$ is irreducible and acyclic. 
        Let $L^- := \max\{0, t_1 - T^-\}$ and $L^+ := \min\{\ell', t_1 + T^+ \}$. 
        \item Let $Y := (\widetilde{Y}_t)_{t=0}^{\ell'}$ and define $Z := (Z_t)_{t=0}^{\ell'}$ by $Z_t = \widetilde{W}_{t}$ for $t\in \{L^-,\ldots,L^+\}$ and $Z_t = \widetilde{Y}_t$ otherwise. 
        Define $Y'$ to be the path of length $\ell'-(L^+-L^-) + 1\leq \ell + 1$ found by concatenating $(\widetilde{Y}_{t})_{t=0}^{L^-}$ and $(\widetilde{Y}_{t})_{t=L^++1}^{\ell'}$. Thus, $Y'$ is equal to $Y$ except for the fact that a piece was cut out. 
    \end{enumerate}
    Observe that $Y$ and $Z$ are paths of length $\ell'+1$ from the block Markov chain conditioned on $\mathcal{V}_{Y,\tau}$ and the block Markov chain conditioned on $\mathcal{V}_{Z,\tau}$ and $E_{Z,t_1} = e_1$, respectively.
    Correspondingly, $\hat{N}_{Z,e_j}$ has the same distribution as $\hat{N}_{Y,e_j}$ conditioned on the event $E_{Y,t_1} = e_1$.

    \proofpart{Bounding $\mathbb{E}[M_{Y,e_{2}}^{m_2} \cdots M_{Y,e_{R}}^{m_{R}} \mid E_{Y,t_1}=e_1] - \mathbb{E}[M_{Y,e_{2}}^{m_2} \cdots M_{Y,e_{R}}^{m_{R}}]$} 

    \noindent
    We continue with the notation of the previous part, in particular $t_1$ and $\tau$ are still considered to be fixed.
    
    The distribution of $Y'$ is obscure unless one conditions on $L^-$ and $L^+$: even the length of $Y'$ is random without this conditioning. 
    In particular, the usual notation $M_{X}= \hat{N}_X - \mathbb{E}[\hat{N}_{X}]$ is not so useful for $Y'$ because $\mathbb{E}[\hat{N}_{Y'}]$ averages out the critical dependence on $L^+$ and $L^-$. 
    For this reason it is convenient to define a variant of $M_{Y'}$ which takes $L^-$ and $L^+$ into account: 
    \begin{align}
        \mathring{M}_{Y'} &:= \hat{N}_{Y'} - \mathbb{E}[\hat{N}_{Y'}\mid L^+, L']. \label{eq: defMY'}
    \end{align}
    
    For $j=2,\ldots,R$ define $\Delta_{Y,j}:= \hat{N}_{Y,e_j} - \hat{N}_{Y',e_j}$, $\Delta_{Z,j}:= \hat{N}_{Z,e_{j}} - \hat{N}_{Y',{e_j}}$ and set 
    \begin{align}
        \mathcal{E}_{Y,j} &:= \Delta_{Y,j} - ( \mathbb{E}[\hat{N}_{Y,e_j}] - \mathbb{E}[\hat{N}_{Y',e_j}\mid L^-, L^+ ]),\\
        \mathcal{E}_{Z,j} &:= \Delta_{Z,j} -( \mathbb{E}[\hat{N}_{Y,e_j}] - \mathbb{E}[\hat{N}_{Y',e_j}\mid L^-, L^+ ]).
        \label{eq: decompE}
    \end{align} 
    Observe that $M_{Y,e_j} = \mathring{M}_{Y',e_j} + \mathcal{E}_{Y,j}$. 
    Moreover, $M_{Y,e_j}$ conditioned on $E_{Y,t_1} = e_1$ has the same distribution as $\hat{N}_{Z,e_j} - \mathbb{E}[\hat{N}_{Y,e_j}] = \mathring{M}_{Y',e_j} + \mathcal{E}_{Z,j}$.
    Correspondingly,
    \begin{align}
        \mathbb{E}[M_{Y,e_{2}}^{m_2} \cdots M_{Y,e_{R}}^{m_{R}}\mid&  E_{Y,t_1}=e_1] - \mathbb{E}[M_{Y,e_{2}}^{m_2} \cdots M_{Y,e_{R}}^{m_{R}}]\label{eq: extras}\\   
        &= \mathbb{E}\bigg[\prod_{j = 2}^R (\mathring{M}_{Y',e_j} + \mathcal{E}_{Z,j})^{m_j} - \prod_{j=2}^R (\mathring{M}_{Y',e_j} + \mathcal{E}_{Y,j})^{m_j}\bigg].
        \nonumber 
    \end{align}
    The key observation is that the leading-order terms in the expansion of the right-hand side of \cref{eq: extras} cancel out. 
    This is to say that 
    \begin{align}
        \mathbb{E}[M_{Y,e_{2}}^{m_2} \cdots& M_{Y,e_{R}}^{m_{R}} \mid  E_{Y,t_1}=e_1] - \mathbb{E}[M_{Y,e_{2}}^{m_2} \cdots M_{Y,e_{R}}^{m_{R}}] \label{eq: PZPY}\\
        &=  \sum_{0\leq m' \leq m} c_{m'}\bigg(\mathbb{E}\bigg[\prod_{j=2}^R \mathcal{E}_{Z,j}^{m_j'}\mathring{M}_{Y',e_j}^{m_j - m_j'} \bigg]  - \mathbb{E}\bigg[\prod_{j=2}^R \mathcal{E}_{Y,j}^{m_j'} \mathring{M}_{Y',e_j}^{m_j - m_j'} \bigg]\bigg)\nonumber 
    \end{align}
    where the summation runs over vectors of positive integers $m'\in \mathbb{Z}_{\geq 0}^r$ and the $c_{m'}$ are absolute constants with $c_{m'} = 0$ if $m_j' = 0$ for all $j\in \{2,\ldots,R \}$.
          
    Fix some $m'\in \mathbb{Z}^r$ with $m_j' \neq 0$ for some $j\in \{2,\ldots,R \}$. 
    We will consider the terms with $\mathcal{E}_{Z,j}$ in \cref{eq: PZPY}; those with $\mathcal{E}_{Y,j}$ may be treated identically. 
    By the law of total expectation, 
    \begin{align}
        \mathbb{E}\bigg[\prod_{j=2}^R \mathcal{E}_{Z,j}^{m_j'}&\mathring{M}_{Y',e_j}^{m_j - m_j'} \bigg]  = \sum_{\ell^+> \ell^-}  \mathbb{P}\bigg( \begin{aligned}
            L^- = \ell^-\\  L^+ = \ell^+
        \end{aligned}
        \bigg) \mathbb{E}\bigg[\prod_{j=2}^R \mathcal{E}_{Z,j}^{m_j'}\mathring{M}_{Y',e_j}^{m_j - m_j'}\ \bigg\vert \ \begin{aligned}
            L^- = \ell^-\\  L^+ = \ell^+
        \end{aligned}
        \bigg].
        \label{eq: TotalE}
    \end{align}
    The compressed notation for conditional expectation employed in \Cref{eq: TotalE} was formally introduced in \Cref{sec: Notation_TwoLineABCD}. 
    Recall that $\mathcal{V}_k$ was chosen not to contain any endpoint of $e_1,\ldots,e_R$. 
    In particular $\Delta_{Z,j}$ is a function of the values of $\widetilde{Z}_{t}$ for $t\in \{L^-+1 ,\ldots,L^{+}-1\}$ whereas $\mathring{M}_{Y',e_j}$ is a function of the values of $\widetilde{Y}_t$ with $t\in \{0,\ldots,L^{-}-1,L^{+}+ 1,\ldots,\ell'\}$. 
    Further, conditional on $L^- = \ell^-$ and $L^+ = \ell^+$ it holds that $\mathbb{E}[\hat{N}_{Y',e_j}\mid L^-,L^+]$ is a deterministic scalar. 
    Recall the definition for $\mathcal{E}_{Z,j}$ in \cref{eq: decompE} and conclude that $\prod_{j=2}^R \mathcal{E}_{Z,j}^{m_j'}$ is conditionally independent of $\prod_{j=2}^R \mathring{M}_{Y',e_j}^{m_j - m_j'}$ given $L^- = \ell^-$ and $L^+ = \ell^+$.   
    Hence,  
    \begin{align}
        \mathbb{E}\bigg[\prod_{j=2}^R \mathcal{E}_{Z,j}^{m_j'}\mathring{M}_{Y',e_j}^{m_j - m_j'}\  \bigg\vert &\ \begin{aligned}
            L^- = \ell^-\\  L^+ = \ell^+
        \end{aligned}
        \bigg] \label{eq: ConditionalIndependence}\\
        &= \mathbb{E}\bigg[\prod_{ j=2}^R \mathcal{E}_{Z,j}^{m_j'} \ \bigg\vert \ \begin{aligned}
            L^- = \ell^-\\  L^+ = \ell^+
        \end{aligned}\bigg] \mathbb{E}\bigg[ \prod_{j=2}^R \mathring{M}_{Y',e_j}^{m_j - m_j'}\ \bigg\vert \ \begin{aligned}
            L^- = \ell^-\\  L^+ = \ell^+
        \end{aligned}\bigg].
        \nonumber
    \end{align}
    We will use the induction hypothesis \cref{eq: d-induct} to deal with $\prod_{j=2}^R \mathring{M}_{Y',e_j}^{m_j - m_j'}$. 
    Let $Y''$ denote a random path which is distributed as $Y'$ conditioned on $L^- = \ell^-$ and $L^+ = \ell^+$. 
    Recall the notation for $\mathring{M}_{Y',e_j}$ in \cref{eq: defMY'} and conclude that $\prod_{j=2}^R \mathring{M}_{Y',e_j}^{m_j - m_j'}$ conditioned on $L^- = \ell^-$ and $L^+ = \ell^+$ has the same distribution as $\prod_{j=2}^R M_{Y'',e_j}^{m_j - m_j'}$. 
    Observe that $Y''$ is a path of length $\ell' - (\ell^+ - \ell^-)+1$ from the block Markov chain conditioned on $\mathcal{V}_{Y'',\tau'}$ where $\tau'$ is the vector of length $d+1$ such that $\tau_{d+1}' = \ell^-$ and for $i=1,\ldots,d$ it holds that $\tau_i' = \tau_i$ if $\tau_i\in \{0,\ldots,\ell^-\}$ and $\tau_i' = \max\{\ell^- + \tau_i - \ell^+,\ell^- \}$ otherwise. 
    The induction hypothesis \cref{eq: d-induct} is applicable and, by inspection of the exponents $m_i - m_{i}'$ with $i\in \{1,\ldots,r' \}$, we may conclude that 
    \begin{align}
        \mathbb{E}\bigg[ \prod_{j=2}^R \mathring{M}_{Y',e_j}^{m_j - m_j'}\ \bigg\vert \ \begin{aligned}
            L^- = \ell^-\\  L^+ = \ell^+
        \end{aligned}\bigg]
        &= 
            O_{m,R,d}\left(n^{-f_1(r,m') - f_2(r,r',m')}\right)\label{eq: JInductionHypothesis}
    \end{align}
    with $f_1(r,m') := \lceil (r -\# \{j\in \{2,\ldots,r\}:m_j' = 1\} - 1)/{2}\rceil$ and $f_2(r,r',m') :=  \lceil (r' - \# \{j\in \{2,\ldots,r' \}:m_j' = 1\}-1)/2\rceil$. 
    
    In Part \ref{prt: ErrBound} we will show that there exists some constant $\newconstant{O2}\in \mathbb{R}_{>0}$, which does not depend on $t_1$ or $\tau$, such that 
    \begin{align}
        \bigg\lvert\mathbb{E}\bigg[\prod_{j=2}^R\mathcal{E}_{Z,j}^{m_j'} \ \bigg\vert \ \begin{aligned}
            L^- = \ell^-\\  L^+ = \ell^+
        \end{aligned}\bigg]\bigg\rvert \leq \cte{O2} n^{-g_1(m') - g_2(m',r')} (\ell^+ - \ell^-)^{3\Vert m' \Vert_1} \label{eq: Ebound}    
    \end{align}
    where $g_1(m'):= \#\{j\in\{2,\ldots,R\}: m_j' \neq 0 \}$ and 
    \begin{align}
        g_2(m',r') = \begin{dcases}
            0 \qquad &\text{ if } g_1(m') = 0 \text{ or }r' = 0,\\
            g_3(m',r')+1 \qquad &\text{ if }  0\leq g_3(m',r') < g_1(m') \text{ and } r'>0, \\
            g_3(m',r')\qquad &\text{ if } 0< g_3(m',r') = g_1(m') \text{ and }  r'>0,
        \end{dcases}
        \label{eq: defg2}
    \end{align}
    where $g_3(m',r'):= \#\{j\in \{2,\ldots,r' \}: m_j' \neq 0 \}$. 
    It will then follow from \cref{eq: TotalE}--\cref{eq: Ebound} that there exists a constant $\newconstant{O3}\in \mathbb{R}_{>0}$ such that
    \begin{align}
        &\bigg\lvert\mathbb{E}\bigg[\prod_{j=2}^R \mathcal{E}_{Z,j}^{m_j'}\mathring{M}_{Y',e_j}^{m_j - m_j'}\bigg]\bigg\rvert\label{eq: BoundDYPY}
        \\ 
        &\leq \cte{O3} n^{-(f_1(r,m') + g_1(m'))-(f_2(r,r',m') + g_2(m',r'))} \mathbb{E}[(L^+ - L^-)^{3\Vert m' \Vert_1}]. \nonumber
    \end{align} 
    Observe that for any positive integers $q_1,q_2\in \mathbb{Z}_{\geq 0}$ it holds that $\lceil (q_1 - q_2 - 1)/ 2\rceil + \max\{1,q_2 \} \geq \lceil q_1/2 \rceil$.
    The condition below \cref{eq: PZPY} stating that $c_{m'} = 0$ when $m_j' = 0$ for all $j\in\{2,\ldots,R \}$ ensures that $g_1(m')\geq \max\{1,\#\{j\in \{2,\ldots,r\}:m_j' = 1 \} \}$ for all terms with $c_{m'}\neq 0$.
    Therefore, using the definition of $f_1$,
    \begin{align}
        f_1(r,m') + g_1(m')  
        &\geq   \lceil r / 2\rceil \label{eq: f1g1}
    \end{align}
    for all terms with $c_{m'}\neq 0$ in \cref{eq: PZPY}.
    Note that $f_2(r,r',m') + g_2(m',r') = 0$ when $r' = 0$. 
    If $r' >0$, note that $g_2(m',r') \geq \max\{1, \# \{j\in \{2,\ldots,r' \}: m_j' =1 \}\}$ for all terms in \cref{eq: PZPY} with $c_{m'}\neq 0$. 
    Hence, using the definition of $f_2$,
    \begin{align}
        f_2(r,r',m') + g_2(m',r') & \geq \lceil r'/2 \rceil \label{eq: f2g2} 
    \end{align}
    for all terms with $c_{m'}\neq 0$ in \cref{eq: PZPY}.

    It will be shown in Part \ref{prt: T+bound} that $\mathbb{E}[(T^+)^{q}]=O_{q,d}(1)$ for any $q\in \mathbb{Z}_{\geq 0}$ and a similar conclusion holds for $T^-$. 
    Note that it is here also claimed that the bound is uniform in $t_1$ and $\tau$ provided that $d$ is fixed. 
    Now observe that $(L^+ - L^-)^{3\Vert m' \Vert_1} \leq (T^+ + T^-)^{3\Vert m' \Vert_1}$.   
    Expand $(T^+ + T^-)^{3\Vert m' \Vert_1}$ and apply the Cauchy--Schwarz inequality to the resulting monomial terms to derive that $\mathbb{E}[(L^+ - L^-)^{3\Vert m' \Vert_1}] = O_{m',d}(1)$.

    It now follows by \cref{eq: PZPY} and \cref{eq: BoundDYPY}--\cref{eq: f2g2} that there exists a constant $\newconstant{O4}\in \mathbb{R}_{>0}$ such that
    \begin{align}
        \mathbb{E}[M_{Y,e_{2}}^{m_2} \cdots M_{Y,e_{R}}^{m_{R}} \mid E_{Y,t_1}=e_1] - \mathbb{E}[M_{Y,e_{2}}^{m_2} \cdots M_{Y,e_{R}}^{m_{R}}] &\leq \cte{O4} n^{-\lceil r/2\rceil - \lceil r'/2 \rceil}.
    \end{align}  
    Given that there are $\ell' \leq \ell = \Theta(n^2)$ terms in \cref{eq: ConditionSum} and that $\mathbb{P}(E_{Y,e_1} = e_1) = O(n^{-2})$ by \cref{eq: PEdgeConditioned} it follows that there exists a constant $\newconstant{O5}\in \mathbb{R}_{>0}$ such that 
    \begin{align}
        \lvert \mathbb{E}[ M_{Y,e_{1}}^{m_1} \cdots M_{Y,e_{R}}^{m_{R}} ] \rvert \leq \cte{O5}n^{-\lceil r/2\rceil - \lceil r'/2 \rceil}
    \end{align}
    which is the desired result. 
    \proofpart{Remaining Bounds}
    \subproofpart{Base case: $r=0$\label{prt: basecase}}

    \noindent
    Recall that we still have to prove that \cref{eq: d-induct} holds for the base case $r=0$. 
    By repeated application of the Cauchy-Schwarz inequality the claim is reduced to the statement that $\mathbb{E}[M_{Y,e_j}^{q}] = O_{q,d}(1)$ for any $j\in \{2,\ldots,R \}$ and $q\in \mathbb{Z}_{\geq 0}$. 
    By definition $M_{Y,e_j} = \hat{N}_{Y,e_j} - \mathbb{E}[\hat{N}_{Y,e_j}]$ so it suffices to show that $\mathbb{E}[\hat{N}_{Y,e_j}^{q'}] = O_{q',d}(1)$ for any $q'\in \mathbb{Z}_{\geq 0}$. 
    Further, since $\hat{N}_{Y,e_j}$ can only increase when one extends $Y$ to a longer path it may be assumed that $\ell' = \ell$. 
    In this case $Y$ follows the same distribution as $X$ conditioned on the event $\mathcal{V}_{X,\tau}$.

    We claim that for any fixed $d$ there exists a constant $\newconstant{Om}\in \mathbb{R}_{>0}$ independent of $n$ such that for any $\tau\in \{0,\ldots,\ell' \}^d$ with $\mathbb{P}(\mathcal{V}_{X,\tau})> 0$ it holds that 
    \begin{align}
        \mathbb{P}(\mathcal{V}_{X,\tau}) \geq \cte{Om}.
        \label{eq: lowerboundCtilde}
    \end{align}
    Indeed, assume without loss of generality that the $\tau_i$ are nondecreasing in $i$.
    Since $X$ starts in equilibrium it then holds that 
    \begin{align}
        \mathbb{P}(\mathcal{V}_{X,\tau}) &=
        \mathbb{P}(X_{\tau_1}\in \mathcal{V}_{k})\prod_{i=1}^{d-1} \mathbb{P}(X_{\tau_{i+1}} \in \mathcal{V}_k \mid X_{\tau_{i}\in \mathcal{V}_k})
        \\ 
        &=  \pi(k)\prod_{i=1}^{d-1} p^{\tau_{i+1}- \tau_i}(k,k).      
    \end{align}
    This implies \cref{eq: lowerboundCtilde} since, by the Markov chain associated with $p$ being irreducible and acyclic, it holds that $p^{t}(k,k)$ tends to $\pi(k)$ as $t$ tends to infinity and it holds that $\pi(k) >0$. 
    
    By the law of total expectation it follows that 
    \begin{align}
        \mathbb{E}[\hat{N}_{Y,e_{j}}^{q'}]  &= \mathbb{E}[\hat{N}_{X,e_j}^{q'}\mid \mathcal{V}_{X,\tau}]\\
        &\leq \cte{Om}^{-1}\mathbb{E}[\hat{N}_{X,e_{j}}^{q'}]
    \end{align}
    which establishes the desired result since $\mathbb{E}[\hat{N}_{X,e_{j}}^{q'}] = O_{q'}(1)$ by \Cref{cor: SingleHighMoment}.
    \subproofpart{$\mathbb{E}[(T^+)^{q} ]= O_{q,d}(1)$\label{prt: T+bound}}

    \noindent
    Let $q\in \mathbb{Z}_{\geq 0}$ and recall the definition of $T^+$ in \cref{eq: T+Original}. 
    Consider the product chain $\Sigma_{(\widetilde{Y},\widetilde{W})}^{+} =(\sigma(\widetilde{Y}_{t_1 + t}), \sigma(\widetilde{W}_{t_1+t}))_{t=0}^\infty$ on the space of clusters $\{1,\ldots,K \}\times \{1,\ldots,K \}$. 
    Then $T^+$ is the first strictly positive time $\Sigma_{(\widetilde{Y},\widetilde{W})}^+$ is in $(k,k)$. 
    Sample infinitely long sample paths $V :=(V_t)_{t=0}^{\infty}$ and $V' := (V_{t}')_{t=0}^{\infty}$ from the block Markov chain with $V'$ conditioned on $V'_{t_1}$ being the ending point of $e_1$. 
    Consider the product chain $\Sigma_{(V,V')}^{+}=(\sigma(V_{t_1 + t}), \sigma(V'_{t_1+ t}))_{t=0}^\infty$ on the space of clusters $\{1,\ldots,K \}\times \{1,\ldots,K \}$. 
    Denote $T_{V,V'}^+$ for the first strictly positive time $\Sigma_{(V,V')}^+$ is in $(k,k)$. 
    By the Markov chain associated with $p$ being irreducible and acyclic it holds that $\mathbb{P}(T_{V,V'}^+>t)$ shows exponential decay in $t$. 
    In particular, $\mathbb{E}[(T_{V,V'}^+)^{q}]$ is finite and independent of $n$.
    
    Now observe that $\Sigma_{(\widetilde{Y},\widetilde{W})}^+$ has the same distribution as $\Sigma_{(V,V')}^+$ conditioned on the events $\mathcal{V}_{V,\tau}$ and $\mathcal{V}_{V',\tau}$. 
    By \cref{eq: lowerboundCtilde} there exists a constant $\cte{Om}\in \mathbb{R}_{>0}$ independent of $n$ and $t_1$ such that $\mathbb{P}(\mathcal{V}_{V,\tau})\geq \cte{Om}$ for all $\tau\in \{0,\ldots,\ell' \}^d$ with $\mathbb{P}(\mathcal{V}_{V,\tau})> 0$.
    It can similarly be deduced that there exists a constant $\newconstant{Omprime}\in \mathbb{R}_{>0}$ such that $\mathbb{P}(\mathcal{V}_{V',\tau}) \geq \cte{Omprime}$  for all $\tau\in \{0,\ldots,\ell' \}^d$ with $\mathbb{P}(\mathcal{V}_{V',\tau})> 0$. 
    Now, by the law of total expectation and the fact that $\mathcal{V}_{V',\tau}$ is independent of $\mathcal{V}_{V,\tau}$ 
    \begin{align}
        \mathbb{E}[(T^+)^{q}] &= \mathbb{E}[(T^+_{V,V'})^{q}\mid \mathcal{V}_{V,\tau}, \mathcal{V}_{V',\tau}]\\ 
        &\leq (\cte{Om}\cte{Omprime})^{-1}\mathbb{E}[(T^+_{V,V'})^{q} ] \label{eq: TBound}
    \end{align}
    which establishes the desired result.
    \subproofpart{$\lvert\mathbb{E}[\prod_{j=2}^R \mathcal{E}_{Z,j}^{q_j} \mid L^- = \ell^-,  L^+ = \ell^+]\rvert
    \leq \cte{O2}n^{-g_1(q) - g_2(q,r')}(\ell^+ - \ell^-)^{3\Vert q \Vert_1}$ \label{prt: ErrBound}}

    \noindent
    Consider some fixed $q\in \mathbb{Z}_{\geq 0}^R$ and $\ell^-, \ell^+ \in \mathbb{Z}_{\geq 0}$ with $\mathbb{P}(L^- = \ell^-, L^+ = \ell^+)> 0$. 
    Let us remark that the following arguments also apply when $Z$ is replaced by $Y$. 
    In fact, the process $Y$ is easier to deal with since $Z$ is conditioned to be at $e_1$ at time $t_1$. 

    Recall the definition for $\mathcal{E}_{Z,j}$ in \cref{eq: decompE}. 
    Note that conditional on $L^+= \ell^+$ and $L^-= \ell^-$, it holds that $\mathbb{E}[\hat{N}_{Y,e_j}] - \mathbb{E}[\hat{N}_{Y',e_j}\mid L^-, L^+]$ is a deterministic scalar. 
    Therefore, after substituting the definition of $\mathcal{E}_{Z,j}$ in $\prod_{j=2}^R\mathcal{E}_{Z,j}^{q_j}$ and by then expanding in terms of $\Delta_{Z,j}$ and $\mathbb{E}[\hat{N}_{Y,e_j}] - \mathbb{E}[\hat{N}_{Y',e_j}\mid L^-, L^+]$ it follows that 
    \begin{align}
        &\mathbb{E}\bigg[\prod_{j=2}^R \mathcal{E}_{Z,j}^{q_j}  \ \bigg\vert \ \begin{aligned}
            L^- = \ell^-\\  L^+ = \ell^+
        \end{aligned}\bigg] \label{eq: deltaZEExpansion} \\
        &= \sum_{0\leq q' \leq q } c_{q'}\mathbb{E}\bigg[\prod_{j=2}^R \Delta_{Z,j}^{q'_j}  \ \bigg\vert \ \begin{aligned}
            L^- = \ell^-\\  L^+ = \ell^+
        \end{aligned}\bigg] \prod_{j=2}^R \bigg(\mathbb{E}\bigg[\hat{N}_{Y,e_j}\bigg] - \mathbb{E}\bigg[\hat{N}_{Y',e_j} \ \bigg\vert \  \begin{aligned}
            L^- = \ell^-\\  L^+ = \ell^+
        \end{aligned}\bigg]\bigg)^{q_j -q'_j}  \nonumber
    \end{align}
    for certain absolute constants $c_{q'}$. The sum runs here over vectors of integers of length $R$.

    Fix some $0\leq q' \leq q$ and let us establish a bound on $\mathbb{E}[\prod_{j=2}^R\Delta_{Z,j}^{q'_j}\mid L^- = \ell^-, L^+ = \ell^+]$. 
    Observe that the product $\prod_{j=2}^R \Delta_{Z,j}^{q'_j}$ can only be nonzero if $\Delta_{Z,j} \neq 0$ for every $j$ with $q'_j\neq 0$.
    Recall from \cref{eq: defg2} that $g_1(q')$ is the number of nonzero $q_j'$ with $j\in \{2,\ldots,R\}$ and $g_3(q',r')$ is the number of nonzero $q_j'$ with $j\in \{2,\ldots,r'\}$.
    Hence the product $\prod_{j=2}^R \Delta_{Z,j}^{q'_j}$ can only be nonzero if $\sum_{j=2}^{R}\1_{\Delta_{Z,j} >0} \geq g_1(q')$ and $\sum_{j=2}^{r'}\1_{\Delta_{Z,j} >0} \geq g_3(q',r')$. 
    Therefore, by the law of total expectation 
    \begin{align}
        \mathbb{E}\bigg[\prod_{j=2}^R \Delta_{Z,j}^{q'_j} \bigg\vert \ \begin{aligned}
            L^- = \ell^-\\  L^+ = \ell^+
        \end{aligned}\bigg]&=\mathbb{P}\Bigg(\begin{aligned}
            {\textstyle\sum}_{j=2}^R \1_{\Delta_{Z,j} >0} \geq g_1(q') \\ {\textstyle\sum}_{j=2}^{r'}\1_{\Delta_{Z,j} >0}  \geq g_3(q',r')
         \end{aligned}\  \bigg\vert\  \begin{aligned}
             L^- = \ell^-\\  L^+ = \ell^+
         \end{aligned}\Bigg)\label{eq: Delta}\\ 
         &\hphantom{=}\times \mathbb{E}\bigg[\prod_{j=2}^R \Delta_{Z,j}^{q'_j} \bigg\vert \ \begin{aligned}
            L^- = \ell^-\\  L^+ = \ell^+
        \end{aligned},\begin{aligned}
            {\textstyle\sum}_{j=2}^R \1_{\Delta_{Z,j} >0} \geq g_1(q') \\ {\textstyle\sum}_{j=2}^{r'}\1_{\Delta_{Z,j} >0}  \geq g_3(q',r')
         \end{aligned}\bigg]
         \nonumber
         .
    \end{align}
    It follows from $\prod_{j=2}^R \Delta_{Z,j}^{q_j'} \leq (L^+ - L^- )^{\Vert q' \Vert_1}$ that 
    \begin{align}
        \mathbb{E}\bigg[\prod_{j=2}^R \Delta_{Z,j}^{q'_j} \bigg\vert \ 
        \begin{aligned}
            L^- = \ell^-\\  L^+ = \ell^+
        \end{aligned},\begin{aligned}
            {\textstyle\sum}_{j=2}^R \1_{\Delta_{Z,j} >0} \geq g_1(q') \\ {\textstyle\sum}_{j=2}^{r'}\1_{\Delta_{Z,j} >0}  \geq g_3(q',r')
         \end{aligned}\bigg]\leq (\ell^+ - \ell^-)^{\Vert q' \Vert_1}
         .
         \label{eq: Lbound}
    \end{align}
    The goal thus becomes to establish a bound on the probability in the right-hand side of \cref{eq: Delta}.
    This bound will be established in two steps. First, we show that the event described by the probability implies that $(Z_t)_{t\in\{\ell^-,\ldots,\ell^+ \}\setminus \{t_1, t_1 - 1\}}$ has to visit the endpoints of $e_2,\ldots,e_R$ often.  
    This step is achieved in \cref{eq: RewriteIntoC}.
    Second, we show that visiting many endpoints is a rare event.  
    This step is achieved in \cref{eq: T'bound}.

    For every $j\in \{1,\ldots,R \}$ consider the following collections of times which measure when we are in $e_j$ or its endpoints 
    \begin{align}
        \mathcal{T}_{e_j} &:= \{t\in  \{\ell^- +1,\ldots,\ell^+ \}: E_{Z,t} = e_j\}, \qquad \mathcal{T}_{e_j}' &:= \cup_{t\in \mathcal{T}_{e_j}}\{t, t-1\}.
        \label{def: Tej}
    \end{align}
    Observe that in this notation it holds that $\Delta_{Z,j} = \# \mathcal{T}_{e_j}$. 
    Hence, the probability in the right-hand side of \cref{eq: Delta} may be rewritten as  
    \begin{align}
        \mathbb{P}\Bigg(&\begin{aligned}
            {\textstyle\sum}_{j=2}^R \1_{\Delta_{Z,j} >0} \geq g_1(q') \\ {\textstyle\sum}_{j=2}^{r'}\1_{\Delta_{Z,j} >0}  \geq g_3(q',r')
         \end{aligned}\  \bigg\vert\  \begin{aligned}
             L^- = \ell^-\\  L^+ = \ell^+
         \end{aligned}\Bigg)
         \label{eq: restatement2}
        \\ &
         = \mathbb{P}\Bigg(\begin{aligned}
            {\textstyle\sum}_{j=2}^R \1_{\mathcal{T}_{e_j} \neq \emptyset} \geq g_1(q') \\ {\textstyle\sum}_{j=2}^{r'} \1_{\mathcal{T}_{e_j}\neq \emptyset} \geq g_3(q',r')
         \end{aligned}\  \bigg\vert\  \begin{aligned}
             L^- = \ell^-\\  L^+ = \ell^+
         \end{aligned}\Bigg).  \nonumber  
    \end{align}
    We claim that whenever the event described in the probability in the right-hand side of \cref{eq: restatement2} holds, it follows that $\#(\cup_{j=2}^R\mathcal{T}_{e_j}'\setminus \{t_1, t_1-1\}) \geq g_1(q') + g_2(q',r')$.
    The key difficulty is that the sets $\mathcal{T}_{e_1}'$, $\ldots$, $\mathcal{T}_{e_R}'$ may not be disjoint due to the fact that the $e_j$ can share endpoints; see \Cref{fig: paths}.

    \begin{figure}[tb]
        \centering
        \includegraphics[width=0.7\textwidth]{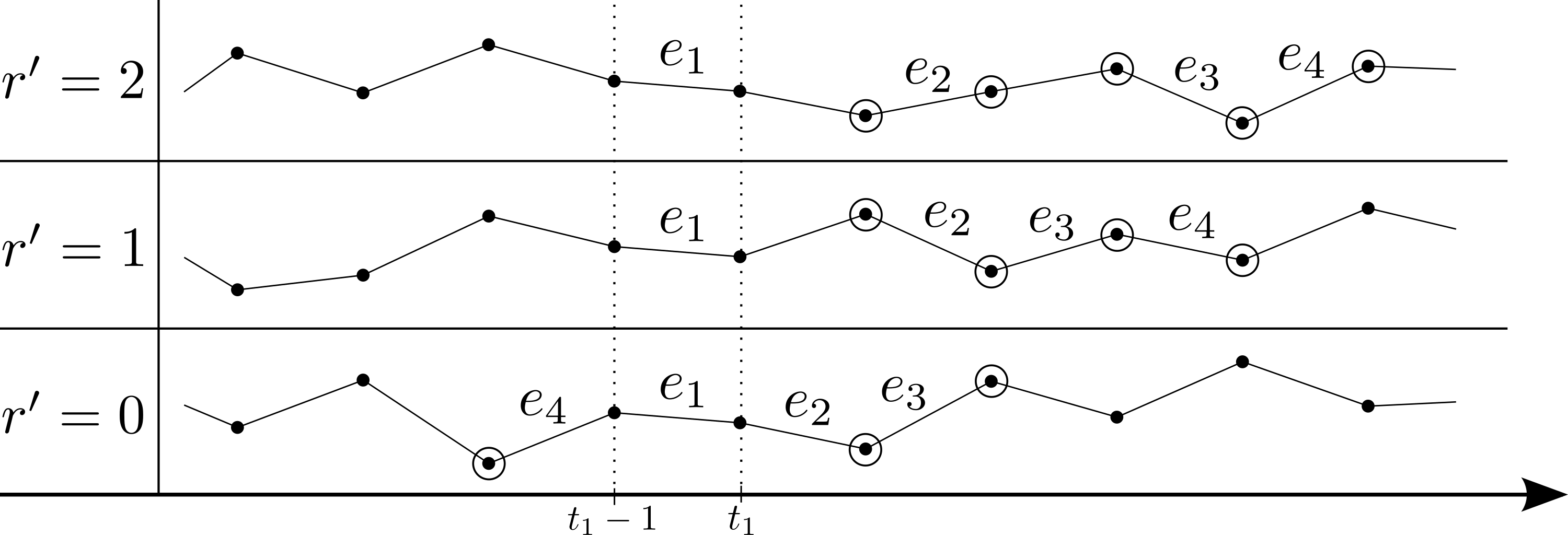}
        \caption{Visualization of some `worst case' realizations of $Z$ for which $\sum_{j=2}^R\1_{\mathcal{T}_{e_j} \neq \emptyset}\geq g_1(q')$ and $\sum_{j=2}^{r'}\1_{\mathcal{T}_{e_j} \neq \emptyset}  \geq g_3(q',r')$ when $g_1(q') = R-1 = 3$ and $g_3(q',r')=r'$ varies. 
        The points contributing to $\#(\cup_{j=2}^R\mathcal{T}_{e_j}'\setminus \{t_1, t_1-1\})$ are circled.
        Note that the number of circled points increases as $r'$ does so. 
        Indeed, $r' >0$ avoids losing points due to the exclusion of $\{t_1, t_1 - 1\}$ whereas $r'\geq j$ with $j\geq 2$ avoids losses due to the possibility that $e_j$ shares endpoints with the other $e_i$.  
        }
        \label{fig: paths}
    \end{figure}

    Recall the definition of $\vec{E}_{n,r'}^R$ from the paragraph preceding \Cref{prop: PowerSharperSharper}. 
    It follows that for every $(i,j) \in \{1,\ldots,R \}\times \{1,\ldots,r'\}$ with $i\neq j$ it holds that $\mathcal{T}_{e_i}' \cap \mathcal{T}_{e_j}' = \emptyset$.
    Hence, 
    \begin{align}
        \#\Big(\bigcup_{j=2}^R&\mathcal{T}_{e_j}'\setminus \{t_1, t_1-1\}\Big)\\ 
        &= \sum_{j=2}^{r'}\#\Big(\mathcal{T}_{e_j}'\setminus \{t_1, t_1-1\}\Big)+ \#\Big(\bigcup_{j=r' + 1}^R\mathcal{T}_{e_j}'\setminus \{t_1, t_1-1\}\Big)\\ 
        &= \sum_{j=2}^{r'}\#\mathcal{T}_{e_j}' + \#\Big(\bigcup_{j=r' + 1}^R\mathcal{T}_{e_j}'\setminus \{t_1, t_1-1\}\Big)
        \label{eq: unionrewrite}  
    \end{align}
    where in \cref{eq: unionrewrite} we used that $\{t_1, t_1 - 1 \} \subseteq \mathcal{T}_{e_1}'$. 

    Note that for every $j \in \{2, \ldots,R \}$ with $\mathcal{T}_{e_j}\neq \emptyset$ it holds that $\#\mathcal{T}_{e_j}' \geq 2$.
    Thus, it follows that 
    \begin{align}
        \sum_{j=2}^{r'}\#\mathcal{T}_{e_j}' \geq 2 \sum_{j=2}^{r'}\1_{\mathcal{T}_{e_j}\neq \emptyset}. \label{eq: firstpartun}
    \end{align} 
    When $r' >0$ it holds that $\mathcal{T}_{e_1}'$ is disjoint from $\mathcal{T}_{e_j}'$ for every $j \in \{2,\ldots, R\}$.
    In particular $\{t_1,t_1 - 1 \}$ is disjoint from $\mathcal{T}_{e_j}'$ for every $j \in \{2,\ldots, R\}$.  
    Hence,  
    \begin{align}
        \bigcup_{j=r' + 1}^R\mathcal{T}_{e_j}'\setminus \{t_1, t_1-1\} = \bigcup_{j=r' + 1}^R\mathcal{T}_{e_j}', \quad \text{ if }r' >0. \label{eq: r'ca}
    \end{align}  
    Whenever $\cup_{j = r' + 1}^R \mathcal{T}_{e_j}\neq \emptyset$ we can construct a subset of $\cup_{j=r' + 1}^R \mathcal{T}_{e_j}'$ as  
    \begin{align}
        \bigg\{\min\Big\{t-1: t\in \bigcup_{j = r' + 1}^R \mathcal{T}_{e_j}\Big\}\bigg\}\cup  \bigcup_{j = r' + 1}^R  \mathcal{T}_{e_j} \subseteq \bigcup_{j = r' + 1}^R\mathcal{T}_{e_j}'. \label{eq: subT}   
    \end{align}
    Observe that the left-hand side of \cref{eq: subT} is a union of disjoint sets due to the fact that $e_2,\ldots,e_R$ are distinct edges.  
    It now follows from \cref{eq: r'ca} that if $r' > 0 $,
    \begin{align}
        \# \Big(\bigcup_{j=r' + 1}^R\mathcal{T}_{e_j}'\setminus \{t_1, t_1-1\} \Big)
        &
        \geq  \1_{(\cup_{j=r' + 1}^R \mathcal{T}_{e_j}) \neq \emptyset} + \sum_{j=r' + 1}^R \# \mathcal{T}_{e_j}
        \\ 
        &
        \geq \1_{(\cup_{j=r' + 1}^R \mathcal{T}_{e_j}) \neq \emptyset} + \sum_{j=r' + 1}^R \1_{\mathcal{T}_{e_j} \neq \emptyset}.
        \label{eq: caser'>0}
    \end{align}  
    When $r' = 0$ we still have an injection from $\cup_{j=r' + 1}^R\mathcal{T}_{e_j}$ into $\cup_{j=r' + 1}^R\mathcal{T}_{e_j}'\setminus \{t_1, t_1-1\}$ defined by $t\mapsto t \1_{t > t_1} + (t-1)\1_{t < t_1}$.
    Hence, 
    \begin{align}
        \# \Big(\bigcup_{j=r' + 1}^R\mathcal{T}_{e_j}'\setminus \{t_1, t_1-1\} \Big) \geq \#\Big(\bigcup_{j=r' + 1}^R\mathcal{T}_{e_j}\Big) \geq \sum_{j=r' + 1}^R \1_{\mathcal{T}_{e_j} \neq \emptyset}.
        \label{eq: caser'=0}
    \end{align} 
    Combine \cref{eq: unionrewrite}--\cref{eq: caser'=0} to deduce that 
    \begin{align}
        \#\Big(\bigcup_{j=2}^R\mathcal{T}_{e_j}'\setminus \{&t_1, t_1-1\}\Big)\label{eq: conunion}
        \\ 
        & \geq \begin{cases}
            \sum_{j=2}^R \1_{\mathcal{T}_{e_j}\neq \emptyset} + \sum_{j=2}^{r'}\1_{\mathcal{T}_{e_j}\neq \emptyset} + \1_{(\cup_{j=r' + 1}^R \mathcal{T}_{e_j})\neq \emptyset}&\text{ if }r' >0,\\ 
            \sum_{j=2}^R \1_{\mathcal{T}_{e_j}\neq \emptyset}&\text{ if }r' = 0.
        \end{cases}
        \nonumber 
    \end{align}
    Recall the condition described in the probability on the right-hand side of \cref{eq: restatement2}. 
    Note that whenever this condition holds it follows that $\sum_{j=2}^R \1_{\mathcal{T}_{e_j}\neq \emptyset}\geq g_1(q')$ and $\sum_{j=2}^{r'}\1_{\mathcal{T}_{e_j}\neq \emptyset} + \1_{(\cup_{j=r' + 1}^R \mathcal{T}_{e_j})\neq \emptyset}\allowbreak \geq g_3(q',r') + \1_{g_1(q')>g_3(q',r')}$. 
    Further, recall from \cref{eq: defg2} that the definition of $g_2$ in terms of $g_3$ has three cases. 
    Now, using \cref{eq: conunion} in \cref{eq: restatement2} and checking each case from the definition of $g_2$ individually we find that 
    \begin{align}
        \mathbb{P}\Bigg(&\begin{aligned}
            {\textstyle\sum}_{j=2}^R \1_{\Delta_{Z,j} >0} \geq g_1(q') \\ {\textstyle\sum}_{j=2}^{r'}\1_{\Delta_{Z,j} >0}  \geq g_3(q',r')
         \end{aligned}\  \bigg\vert\  \begin{aligned}
             L^- = \ell^-\\  L^+ = \ell^+
         \end{aligned}\Bigg)\label{eq: RewriteIntoC}\\ 
         & \leq \mathbb{P} \bigg(\#\Big(\bigcup_{j=2}^R \mathcal{T}'_{e_j} \setminus \{t_1,t_1 - 1 \}\Big) \geq g_1(q') + g_2(q',r')\  \bigg\vert\  \begin{aligned}
             L^- = \ell^-\\  L^+ = \ell^+
         \end{aligned}\bigg).
         \nonumber
    \end{align} 
    Let us denote $V(e_2,\ldots,e_R)$ for the set of cardinality $\leq 2(R-1)$ consisting of the endpoints of the edges $e_2,\ldots,e_R$. 
    Observe that $\#(\cup_{j=2}^R \mathcal{T}'_{e_j} \setminus \{t_1,t_1 - 1 \}) \geq g_1(q') + g_2(q',r')$ implies that there exists some $T' \subseteq \{\ell^-,\ldots,\ell^+ \}\setminus \{t_1, t_1 -1 \}$ with $\#T' =  g_1(q') + g_2(q',r')$ such that $Z_{t'}\in V(e_2,\ldots,e_R)$ for all $t' \in T'$.  
    Consequently, by the union bound 
    \begin{align}
        \mathbb{P} &\bigg(\#\Big(\bigcup_{j=2}^R \mathcal{T}'_{e_j} \setminus \{t_1,t_1 - 1 \}\Big) \geq g_1(q') + g_2(q',r')\  \bigg\vert\  \begin{aligned}
            L^- = \ell^-\\  L^+ = \ell^+
        \end{aligned}\bigg)\label{eq: T'}\\ 
        &\leq \sum_{\#T' =  g_1(q') + g_2(q',r')}\mathbb{P}\bigg( \{ Z_{t'}: t'\in T' \} \subseteq V(e_2,\ldots, e_R)\  \bigg\vert\  \begin{aligned}
            L^- = \ell^-\\  L^+ = \ell^+
        \end{aligned}\bigg)
    \end{align}  
    where the sum runs over all subsets $T' \subseteq \{\ell^-, \ldots, \ell^+ \}\setminus \{t_1, t_1 - 1 \}$ with $\#T' =  g_1(q') + g_2(q',r')$.
    Fix such a subset $T'$ and denote $\Sigma_{Z} := (\sigma(Z_t))_{t =0}^{\ell'}$ for chain of clusters associated to $Z$. 
    Then, the law of total probability yields that 
    \begin{align}
        &\mathbb{P}\bigg( \{ Z_{t'}: t'\in T' \} \subseteq V(e_2,\ldots, e_R)\  \bigg\vert\  \begin{aligned}
            L^- = \ell^-\\  L^+ = \ell^+
        \end{aligned}\bigg)\\ 
        &= \sum_{s_Z} \mathbb{P}\bigg(\Sigma_Z = s_Z\  \bigg\vert\  \begin{aligned}
            L^- = \ell^-\\  L^+ = \ell^+
        \end{aligned} \bigg) \mathbb{P}\bigg( \{ Z_{t'}: t'\in T' \} \subseteq V(e_2,\ldots,e_R) \  \bigg\vert\  \Sigma_{Z} = s_Z\bigg)\nonumber  
    \end{align}
    where the sum runs over all $s_Z \subseteq \{1,\ldots,K \}^{\ell' + 1}$ with $\mathbb{P}(\Sigma_Z = s_Z \mid L^- = \ell^-, L^+ = \ell^+) \neq 0$. 
    Now, observe that $(Z_t)_{t\in \{0,\ldots,\ell' \}\setminus\{t_1, t_1 -1 \}}$ is uniformly distributed in $\prod_{t\in\{0,\ldots,\ell' \}\setminus \{t_1, t_1 - 1 \}} \mathcal{V}_{s_{Z,t}}$ given that $\Sigma_Z = s_Z$. 
    Hence, by using that $\# V(e_2,\ldots,e_R) \leq 2(R-1)$ and $\# \mathcal{V}_{k} \geq \alpha_{min}n$ for every $k\in \{1,\ldots,K \}$
    \begin{align}
        \mathbb{P}\big(\{Z_{t'}:t'\in T' \}\subseteq V(e_2,\ldots,e_R)&\mid \Sigma_Z = s_Z\big)\\
        &= \prod_{t' \in T'} \frac{\#(V(e_2,\ldots, e_R)\cap \mathcal{V}_{s_{Z,t'}})}{\# \mathcal{V}_{s_{Z,t'}}}\\
        &\leq(2(R-1)\alpha_{min}^{-1} n^{-1})^{g_1(q') + g_2(q',r')}\\ 
        &\leq  \cte{55} n^{-g_1(q') - g_2(q',r')} \label{eq: BoundNVProb}
    \end{align}  
    where we defined $\newconstant{55}:= \max_{0\leq q' \leq q}(2(R-1) \alpha_{min}^{-1})^{g_1(q') + g_2(q',r')}$.
    Now, by \cref{eq: T'}--\cref{eq: BoundNVProb} with the bound $\binom{\ell^+ - \ell^- -1 }{g_1(q')+ g_2(q', r')}\leq (\ell^+ - \ell^-)^{g_1(q') + g_2(q', r')}$ for the number of terms in the sum of \cref{eq: T'} 
    \begin{align}
        & 
        \mathbb{P} \bigg(
            \#\Big(\bigcup_{j=2}^R \mathcal{T}'_{e_j} \setminus \{t_1,t_1 - 1 \}\Big) \geq g_1(q') + g_2(q',r')\  \bigg\vert\  \begin{aligned}
            L^- = \ell^-\\  L^+ = \ell^+
            \end{aligned}
        \bigg)
        \label{eq: T'bound}
        \\ &
        \leq \cte{55} n^{-(g_1(q') + g_2(q',r'))}(\ell^+ - \ell^-)^{g_1(q')+ g_2(q',r')} 
        .
        \nonumber  
    \end{align}
    Combine \cref{eq: Delta}, \cref{eq: Lbound} and \cref{eq: RewriteIntoC} with the bound \cref{eq: T'bound} to conclude that 
    \begin{align}
        \mathbb{E}\bigg[\prod_{j=2}^R \Delta_{Z,j}^{q'_j} \bigg\vert \ \begin{aligned}
            L^- = \ell^-\\  L^+ = \ell^+
        \end{aligned}\bigg] \leq \cte{55} (\ell^+ - \ell^-)^{3\Vert q' \Vert_1}n^{-(g_1(q') + g_2(q',r'))}
        \label{eq: deltaZConclusion}
    \end{align}
    where it was used that $g_1(q') \leq \Vert q' \Vert_1$ and $g_2(q') \leq \Vert q' \Vert_1$. 
    This establishes the desired upper bound for the conditional expectation of $\prod_{j=2}^R \Delta_{Z,j}^{q'_j}$. 

    Let us remark that the only property about the chain $Z$ which was used in the foregoing argument is that $(Z_t)_{t\in \{0,\ldots,\ell' \}\setminus \{t_1, t_1 - 1 \}}$ is uniformly distributed in $\prod_{t\in \{0,\ldots,\ell' \}\setminus \{t_1, t_1 - 1 \}} \mathcal{V}_{s_Z}$ when conditioned on $\Sigma_Z = s_Z$. 
    Hence, the conclusion \cref{eq: deltaZConclusion} also applies to other chains with this property such as $Y$. 
    This is to say that by repeating the argument for \cref{eq: deltaZConclusion} word--for--word one finds a constant $\newconstant{Y}\in \mathbb{R}_{>0}$ such that 
    \begin{align}
        \mathbb{E}\bigg[\prod_{j=2}^R \Delta_{Y,j}^{q'_j} \bigg\vert \ \begin{aligned}
            L^- = \ell^-\\  L^+ = \ell^+
        \end{aligned}\bigg] \leq \cte{Y} (\ell^+ - \ell^-)^{3\Vert q' \Vert_1}n^{-(g_1(q') + g_2(q',r'))}
        \label{eq: DeltaYProdBound}
    \end{align}
    for any $0\leq q' \leq q$. 
    
    Next, we consider the factors $(\mathbb{E}[\hat{N}_{Y,e_j}] - \mathbb{E}[\hat{N}_{Y',e_j} \mid L^- = \ell^-, L^+ = \ell^+])$ in \cref{eq: deltaZEExpansion}.   
    Fix some $j\in \{2,\ldots,R\}$ and observe that 
    \begin{align}
        \mathbb{E}[\hat{N}_{Y,e_j}]
        -
        &
        \mathbb{E}\bigg[\hat{N}_{Y',e_j} \bigg\vert \ \begin{aligned}
            L^- = \ell^-\\  L^+ = \ell^+
        \end{aligned}\bigg] 
        = 
        \bigg(\mathbb{E}[\hat{N}_{Y,e_j}] - \mathbb{E}\bigg[\hat{N}_{Y,e_j} \bigg\vert\ \begin{aligned}
            L^- = \ell^-\\  L^+ = \ell^+
        \end{aligned}\bigg]\bigg)
        \nonumber \\ &
        + 
        \bigg(\mathbb{E}\bigg[\hat{N}_{Y,e_j} \bigg\vert\ 
        \begin{aligned}
            L^- = \ell^-\\  L^+ = \ell^+
        \end{aligned}\bigg] -\mathbb{E}\bigg[\hat{N}_{Y',e_j} \bigg\vert\ 
        \begin{aligned}
            L^- = \ell^-\\  L^+ = \ell^+
        \end{aligned}\bigg]
        \bigg).
        \label{eq: Ndiff}
    \end{align}
    Here, since $\hat{N}_{Y,e_j} - \hat{N}_{Y',e_j} = \Delta_{Y,j}$, the bound \cref{eq: DeltaYProdBound} yields a constant $\newconstant{56} \in \mathbb{R}_{>0}$ such that 
    \begin{align}
        \mathbb{E}\bigg[\hat{N}_{Y,e_j}\ \bigg\vert\ 
        \begin{aligned}
            L^- = \ell^-\\  L^+ = \ell^+
        \end{aligned}\bigg] -\mathbb{E}\bigg[\hat{N}_{Y',e_j}\ \bigg\vert&\ 
        \begin{aligned}
            L^- = \ell^-\\  L^+ = \ell^+
        \end{aligned}\bigg] 
        = 
        \mathbb{E}\bigg[\Delta_{Y,j}\ \bigg\vert\ \begin{aligned}
            L^- = \ell^-\\  L^+ = \ell^+
        \end{aligned}\bigg].
        \label{eq: DY}\\ 
        &\leq  \begin{dcases}
            \cte{56}n^{-1}(\ell^+ - \ell^-)^{3}\qquad\text{ if } r' = 0,\\
            \cte{56}n^{-2}(\ell^+ - \ell^-)^{3}\qquad\text{ if } r' >0.
        \end{dcases}
        \label{eq: SecondE}
    \end{align}
    It remains to consider the term $\mathbb{E}[\hat{N}_{Y,e_j}] - \mathbb{E}[\hat{N}_{Y,e_j} \mid L^- = \ell^-, L^+ = \ell^+]$ in \cref{eq: Ndiff}. 
    This term may be studied by means of a coupling argument.   
    Construct a path $G$ of length $\ell' + 1$ by using the following procedure: 
    \begin{enumerate}[label = \rm{(\roman*)}]
        \item Let $\widetilde{Y} := (\widetilde{Y}_{t})_{t=-\infty}^{\infty}$ be the path used in the construction of $Y$ in Part \ref{prt: construct}.
        Independently sample $\widetilde{G}:=(\widetilde{G}_t)_{t=-\infty}^\infty$ from the same distribution as $\widetilde{Y}$ conditioned on $L^- = \ell^{-}$ and $L^+ = \ell^+$. 
        \item Define
        \begin{align}
                \widetilde{T}^- &= \ell^- - \sup\{t\in \mathbb{Z}_{<\ell^-}: \widetilde{Y}_t \in \mathcal{V}_k, \widetilde{G}\in \mathcal{V}_k \} \\ 
                \widetilde{T}^+ &= \inf\{t\in \mathbb{Z}_{>\ell^+}: \widetilde{Y}_t \in \mathcal{V}_k, \widetilde{G}_t\in \mathcal{V}_k \}-\ell^+\label{eq: T+}
        \end{align}
        and note that these values are finite with probability one by the assumption that the Markov chain associated with $p$ is irreducible and acyclic. 
        Let $\widetilde{L}^- = \max\{0, \ell^- - \widetilde{T}^-\}$ and $\widetilde{L}^+ = \min\{\ell', \ell^+ + \widetilde{T}^+ \}$.
        \item Define $G:= (G_t)_{t=0}^{\ell'}$ by $G_t := \widetilde{G}_t$ for $t\in \{\widetilde{L}^-,\ldots,\widetilde{L}^+\}$ and $G_t = \widetilde{Y}_t$ otherwise.   
    \end{enumerate}
    Note that $G$ is here implicitly dependent on $\ell^-$ and $\ell^+$ but this is suppressed in the notation.
    Indeed, $G$ has the same distribution as $Y$ conditioned on $L^- = \ell^-$ and $L^+ = \ell^+$.
    For any $j\in \{2,\ldots,R\}$ let $\widetilde{\Delta}_{G,j}$ and $\widetilde{\Delta}_{Y,j}$ denote the number of times $(G_t)_{t= \widetilde{L}^-}^{\widetilde{L}^+}$ and $(Y_t)_{t=\widetilde{L}^-}^{\widetilde{L}^+}$ traversed edge $e_j$. 
    Then,  
    \begin{align}
        \bigg\lvert\mathbb{E}\bigg[\hat{N}_{Y,e_j}\bigg] - \mathbb{E}\bigg[\hat{N}_{Y,e_j} \ \bigg\vert \ \begin{aligned}
            L^- = \ell^-\\  L^+ = \ell^+
        \end{aligned}\bigg]\bigg\lvert &= \bigg\lvert\mathbb{E}\bigg[\hat{N}_{Y,e_j} - \hat{N}_{G,e_j}\bigg]\bigg\rvert\\
        &\leq \lvert \mathbb{E}[\widetilde{\Delta}_{Y,j}]\rvert + \lvert\mathbb{E}[ \widetilde{\Delta}_{G,j}]\rvert. 
        \label{eq: ECoupling}
    \end{align}
    As in \cref{eq: DeltaYProdBound} it may be established that for any $\widetilde{\ell}^+, \widetilde{\ell}^-\in \mathbb{Z}_{\geq 0}$ with $\mathbb{P}(\widetilde{L}^+ = \widetilde{\ell}^+ , \widetilde{L}^- =\widetilde{\ell}^-)> 0$ it holds that
    \begin{align}
        \bigg\lvert \mathbb{E}\bigg[\widetilde{\Delta}_{Y,j}\ \bigg\vert \ \begin{aligned}
            \widetilde{L}^- = \widetilde{\ell}^-\\  \widetilde{L}^+ = \widetilde{\ell}^+
        \end{aligned}\bigg] \bigg\rvert \leq \begin{dcases}
            \cte{56}n^{-1}(\widetilde{\ell}^+ - \widetilde{\ell}^-)^{3}\qquad &\text{ if }r' = 0,\\
            \cte{56}n^{-2}(\widetilde{\ell}^+ - \widetilde{\ell}^-)^{3}\qquad &\text{ if }r' >0.
        \end{dcases}
        \label{eq: tilde}
    \end{align}
    By the law of total expectation it now follows that 
    \begin{align}
        \lvert\mathbb{E}[\widetilde{\Delta}_{Y,j}] \leq 
        \begin{dcases}
            \cte{56}n^{-1}\mathbb{E}[(\widetilde{L}^+ - \widetilde{L}^-)^{3}]\qquad &\text{ if }r' =0,\\
            \cte{56}n^{-2} \mathbb{E}[(\widetilde{L}^+ - \widetilde{L}^-)^{3}]\qquad &\text{ if }r' >0,
        \end{dcases}
        \label{eq: TotalEZG}
    \end{align}
    and a similar conclusion applies to $\widetilde{\Delta}_{G,j}$. 

    We claim that $\mathbb{E}[(\widetilde{T}^+)^m] = O_{m}(1)$ and similarly $\mathbb{E}[(\widetilde{T}^-)^m] = O_{m}(1)$ for any $m \in \mathbb{Z}_{\geq 0}$.
    Indeed, this may be deduced as in \cref{eq: TBound} by consideration of $\Sigma_{(\widetilde{Y},\widetilde{G})}^+ := (\sigma(\widetilde{Y}_{\ell^+ + t}), \sigma(\widetilde{G}_{\ell^+ + t}))_{t=0}^\infty$, $\widetilde{V}:= (\widetilde{V}_t)_{t=0}^\infty$ and $\widetilde{V}':= (\widetilde{V}_t)_{t=0}^\infty$ where $\widetilde{V}$ is sample path from the block Markov chain and $\widetilde{V}'$ is sample path from the block Markov chain conditioned on the event $\widetilde{V}_{\ell^+}\in \mathcal{V}_k$. 
    Next, use a binomial expansion on $\widetilde{L}^+ - \widetilde{L}^- = \ell^+ - \ell^- +(\widetilde{T}^+ + \widetilde{T}^-)$ to write
    \begin{align}
        \mathbb{E}[(\widetilde{L}^+ - \widetilde{L}^-)^3]&= \sum_{m = 0}^3 \binom{3}{m}\mathbb{E}\big[(\widetilde{T}^+ + \widetilde{T}^-)^{m}\big](\ell^+ - \ell^-)^{3-m} \\ 
        &= \sum_{m=0}^3 \binom{3}{m}\sum_{m'=0 }^m \binom{m}{m'} \mathbb{E}\big[(\widetilde{T}^+)^{m'} (\widetilde{T}^-)^{m - m'}\big](\ell^+ - \ell^-)^{3-m}
        \label{eq: bqL}   
    \end{align}
    The coefficients $\mathbb{E}[(\widetilde{T}^+)^{m'} (\widetilde{T}^-)^{m - m'}]$ in \cref{eq: bqL} are $O_{m,m'}(1)$ by the Cauchy--Schwarz inequality.  
    Therefore, since $(\ell^+ - \ell^-)\geq 1$, there exists some constant $\newconstant{Inline1}\in \mathbb{R}_{>0}$ such that 
    \begin{align}
        \mathbb{E}[(\widetilde{L}^+ - \widetilde{L}^-)^3] \leq \cte{Inline1}(\ell^+ - \ell^-)^3.
        \label{eq: Inl}        
    \end{align}
    By \cref{eq: ECoupling}---\cref{eq: Inl} it follows that there exists some constant $\newconstant{CInl}\in \mathbb{R}_{>0}$ such that 
    \begin{align}
        \bigg\lvert\mathbb{E}\bigg[\hat{N}_{Y,e_j}\bigg] - \mathbb{E}\bigg[\hat{N}_{Y,e_j} \ \bigg\vert \ \begin{aligned}
            L^- = \ell^-\\  L^+ = \ell^+
        \end{aligned}\bigg]\bigg\lvert&\leq \begin{dcases}
            \cte{CInl}n^{-1}(\ell^+ - \ell^-)^{3}\qquad &\text{ if }r' = 0,\\
            \cte{CInl}n^{-2}(\ell^+ - \ell^-)^{3}\qquad &\text{ if }r' >0.
        \end{dcases}
        \label{eq: YcouplBound}
    \end{align}
    By \cref{eq: Ndiff}, \cref{eq: SecondE} and \cref{eq: YcouplBound} we may conclude that
    \begin{align}
        \bigg\lvert \mathbb{E}[\hat{N}_{Y,e_j}]-\mathbb{E}\bigg[\hat{N}_{Y',e_j} \ \bigg\vert \ \begin{aligned}
            L^- = \ell^-\\  L^+ = \ell^+
        \end{aligned}\bigg] \bigg\rvert 
        \leq
        \begin{dcases}
            \cte{O55}n^{-1} (\ell^+ - \ell^-)^{3}\quad &\text{ if }r' = 0,\\ 
            \cte{O55}n^{-2} (\ell^+ - \ell^-)^{3}\quad &\text{ if }r' >0.
        \end{dcases}
      \label{eq: Econclusion}
    \end{align}
    for some constant $\newconstant{O55}\in \mathbb{R}_{>0}$. 
    Combine \cref{eq: deltaZConclusion} and \cref{eq: Econclusion} to find a constant $\newconstant{c31}\in \mathbb{R}_{>0}$ such that for any $0\leq q' \leq q$
    \begin{align}
        &\mathbb{E}\bigg[\prod_{j=2}^R \Delta_{Z,j}^{q'_j}  \ \bigg\vert \ \begin{aligned}
            L^- = \ell^-\\  L^+ = \ell^+
        \end{aligned}\bigg] \prod_{j=2}^R \bigg(\mathbb{E}\bigg[\hat{N}_{Y,e_j}\bigg] - \mathbb{E}\bigg[\hat{N}_{Y',e_j} \ \bigg\vert \  \begin{aligned}
            L^- = \ell^-\\  L^+ = \ell^+
        \end{aligned}\bigg]\bigg)^{q_j -q'_j}\\
        &\leq 
        \begin{dcases}
            \cte{c31}n^{-g_1(q') - g_2(q',r') - \sum_{j=2}^R (q_j - q_j')} (\ell^+ - \ell^-)^{3\Vert q \Vert_1 }\quad &\text{ if }r' = 0,\\ 
            \cte{c31} n^{-g_1(q') - g_2(q',r') - 2\sum_{j=2}^R(q_j - q_j')} (\ell^+ - \ell^-)^{3\Vert q \Vert_1}\quad &\text{ if }r' >0,
        \end{dcases}
        \label{eq: p1}
    \end{align}
    where it was used that $\Vert q' \Vert_1  + \Vert q - q' \Vert_1 = \Vert q \Vert_1$. 
    Recall the definition of $g_1$ and $g_2$ in \cref{eq: defg2} and observe that each additional nonzero coordinate in $q'$ increases $g_1(q') + g_2(q',r')$ by at most $1 + \1_{r' >0}$. 
    Hence, the worst bound in \cref{eq: p1} is attained at $q' = q$ and we may conclude that   
    \begin{align}
        &\mathbb{E}\bigg[\prod_{j=2}^R \Delta_{Z,j}^{q'_j}  \ \bigg\vert \ \begin{aligned}
            L^- = \ell^-\\  L^+ = \ell^+
        \end{aligned}\bigg] \prod_{j=2}^R \bigg(\mathbb{E}\bigg[\hat{N}_{Y,e_j}\bigg] - \mathbb{E}\bigg[\hat{N}_{Y',e_j} \ \bigg\vert \  \begin{aligned}
            L^- = \ell^-\\  L^+ = \ell^+
        \end{aligned}\bigg]\bigg)^{q_j -q'_j}\\ 
        &\leq  
            \cte{c31}n^{-g_1(q) - g_2(q,r')} (\ell^+ - \ell^-)^{3\Vert q \Vert_1 }.
            \label{eq: conclusionDE}
    \end{align}
    Plug \cref{eq: conclusionDE} into \cref{eq: deltaZEExpansion} to find a constant $\newconstant{O56}\in \mathbb{R}_{>0}$ such that 
    \begin{align}
        \bigg\lvert
        \mathbb{E}\bigg[\prod_{j=2}^R \mathcal{E}_{Z,j}^{q_j}  \ \bigg\vert \ \begin{aligned}
            L^- = \ell^-\\  L^+ = \ell^+
        \end{aligned}\bigg] \bigg\rvert \leq \cte{O56}n^{-g_1(q) - g_2(q,r')}(\ell^+ - \ell^-)^{3\Vert q \Vert_1}
    \end{align}
    which is the desired result. 
\end{proof}

\section*{Acknowledgements}

This work is part of the project \emph{Clustering and Spectral Concentration in Markov Chains}
with project number OCENW.KLEIN.324 of the research programme \emph{Open Competition Domain
Science -- M} which is partly financed by the Dutch Research Council (NWO). 

We would like to thank Sem Borst, Martijn G\"osgens, Gianluca Kosmella, Albert Senen--Cerda and Haodong Zhu for providing feedback on a draft of this manuscript. 

\bibliographystyle{plain}

\end{document}